\documentclass[a4paper,10pt]{amsart}
\usepackage{amssymb}
\usepackage{amsmath}
\usepackage{enumitem}
\usepackage{tikz}
\usepackage[all,cmtip]{xy}
\newtheorem{thm}{Theorem}[section] 
\newtheorem{lem}[thm]{Lemma} 
 
\newtheorem{prop}[thm]{Proposition}
 
\theoremstyle{definition}
\newtheorem*{ack}{Acknowledgments}
\newtheorem{defi}[thm]{Definition} 
\newtheorem{rem}[thm]{Remark}

\title[]{Homology of the complexes of finite Verma modules over $CK_6$}

\author{Lucia Bagnoli}
\subjclass[2010]{08A05, 17B05 (primary), 17B65, 17B70 (secondary)}
\keywords{conformal superalgebras, linearly compact Lie superalgebras, finite Verma modules, singular vectors}
\address{Lucia Bagnoli, Dipartimento di matematica, Universit\`a di Bologna, Piazza di Porta San Donato 5, 40126 Bologna, Italy}

\email{luciabagnoli93@gmail.com, lucia.bagnoli4@unibo.it}

\DeclareMathOperator{\Sym}{Sym}
\DeclareMathOperator{\Hom}{Hom}
\DeclareMathOperator{\Ind}{Ind}
\DeclareMathOperator{\Cur}{Cur}
\DeclareMathOperator{\Lie}{Lie}
\DeclareMathOperator{\Sing}{Sing}
\DeclareMathOperator{\Ker}{Ker}
\DeclareMathOperator{\Ima}{Im}

\DeclareMathOperator{\End}{End}
\DeclareMathOperator{\I}{I}

\DeclareMathOperator{\degr}{deg}
\DeclareMathOperator{\spann}{Span}
\DeclareMathOperator{\Gr}{Gr}

\newcommand{\C}{\mathbb{C}}

\newcommand{\Z}{\mathbb{Z}}

\newcommand{\slq}{\mathfrak{sl}_4}

\newcommand{\inlinewedge}{\textrm{\raisebox{0.6mm}{\footnotesize $\bigwedge$}}}
\newcommand{\displaywedge}{\textrm{\raisebox{0.6mm}{\tiny $\bigwedge$}}}
\newcommand{\g}{\mathfrak {g}}

\allowdisplaybreaks
\begin{document}
	\maketitle
	\begin{abstract}
	We compute the homology of the first and third quadrants of the complexes of finite Verma modules over the annihilation superalgebra $\mathcal{A}(CK_{6})\cong E(1,6)$, associated with the conformal superalgebra $CK_6$, obtained in \cite{ck6}. This computation allows us to explicitly realize the irreducible quotients of degenerate finite Verma modules over $\mathcal{A}(CK_{6})$.
	\end{abstract}
\section{Introduction} 
Finite simple conformal superalgebras were completely classified in \cite{fattorikac} and consist of the following list: $\Cur \mathfrak{g}$, where $\mathfrak{g}$ is a simple finite$-$dimensional Lie superalgebra, $W_{n} (n\geq 0)$, $S_{n,b}$ and $\tilde{S}_{n}$ $(n\geq 2, \, b \in \mathbb{C})$, $K_{n} (n\geq 0, \, n \neq 4)$, $K'_{4}$, $CK_{6} $. The finite irreducible modules over the conformal superalgebras $\Cur \mathfrak{g}$, $K_{0}$, $K_{1}$ were classified in \cite{chengkac,chengkacE}. The classification of all the finite irreducible modules over the conformal superalgebras of type $W$ and $S$ was obtained in \cite{bklr}. The finite irreducible modules over the conformal superalgebras $S_{2,0}$, $K_{n}$, for $n =2,3, 4$ were classified in \cite{chenglam}. The classification of all the finite irreducible modules over the conformal superalgebras of type $K_{n}$ was obtained in \cite{kac1}. The classification and explicit realization of all the finite irreducible modules over the conformal superalgebra $K'_{4}$ were obtained in \cite{bagnolitesi,bagnolicaselli,bagnoli2}. The classification of all the finite irreducible modules over the conformal superalgebra $CK_{6}$ was obtained in \cite{ck6} and \cite{zm} with different techniques.  
	In \cite{ck6} the classification of all the finite irreducible modules over the conformal superalgebra $CK_6$ was obtained by their correspondence with irreducible \textit{finite conformal} modules over the annihilation superalgebra $\g:=\mathcal{A}(CK_{6})$, that is isomorphic to the exceptional Lie superalgebra $E(1,6)$, associated with $CK_6$. In order to obtain this classification, the authors classified all highest weight singular vectors, i.e. highest weight vectors which are annihilated by $\g_{>0}$, of finite Verma modules, that are the generalized Verma modules $\Ind(F)=U(\g) \otimes _{U(\g_{\geq 0})} F$ such that $F$ is a finite$-$dimensional irreducible $\g_{\geq 0}$-module \cite{kacrudakov,chenglam}. 
	Using the equivalence between the classification of singular vectors of finite Verma modules and the classification of morphisms between degenerate, i.e. not irreducible, finite Verma modules, in \cite{ck6} it was proved that these morphisms can be arranged in an infinite number of complexes as in Figure \ref{figurack6}, which is similar to those obtained for the exceptional Lie superalgebras $E(3,6)$, $E(3,8)$ and $E(5,10)$ (see \cite{kacrudakovE36,kacrudakov,kacrudakovE38,E36III, cantacaselliE510, cantacasellikacE510, rudakovE510}).\\
	In this work we compute the homology of the complexes in quadrants \textbf{A} and \textbf{C} in Figure \ref{figurack6} and then provide an explicit construction of all the irreducible quotients of finite Verma modules over $\mathcal A(CK_6)$ in such quadrants. We compute the homology of quadrant \textbf{A} through the theory of spectral sequences of bicomplexes, following \cite{kacrudakovE36} and \cite{bagnoli2}. Then we use an important result on conformal duality, shown in \cite{cantacasellikac} , to obtain the homology for quadrant  \textbf{C}. We show that the complexes in quadrants \textbf{A} and \textbf{C} in Figure \ref{figurack6} are exact in each point except for the origin of the first quadrant and the point of coordinates $(1,0)$ in the third quadrant, in which the homology spaces are isomorphic to the trivial representation. \\
	The computation of the homology allows us to show that all the irreducible quotients of finite Verma modules over $\mathcal A(CK_6)$ in those quadrants occur among cokernels, kernel and images of complexes in Figure \ref{figurack6}. In future it would be interesting to understand if analogous arguments can be applied for the computation of the homology for quadrant \textbf{B} and for the complexes recently obtained for $E(5,10)$ in \cite{cantacasellikacE510}.\\
	The paper is organized as follows. In Section 2 we recall some notions on conformal superalgebras. In Section 3 we recall the definition of the conformal superalgebra $CK_6$ and the classification of singular vectors obtained in \cite{ck6}. In Section 4 we find an explicit expression for the morphisms represented quadrant \textbf{A} and between quadrants \textbf{A} and \textbf{B} in Figure \ref{figurack6}. In Section 5 we recall some preliminaries on spectral sequences. In Section 6 we compute the homology of the complexes in quadrants \textbf{A} and \textbf{C} in Figure \ref{figurack6}.
 \section{Preliminaries on conformal superalgebras}
We recall some notions on conformal superalgebras. For further details see \cite[Chapter 2]{kac1vertex}, \cite{dandrea}, \cite{bklr}, \cite{kac1}, \cite{bagnolitesi}, \cite{bagnolicaselli}.
\begin{defi}[Conformal superalgebra]
A \textit{conformal superalgebra} $R$ is a left $\Z_{2}-$graded $\C[\partial]-$module endowed with a $\C-$linear map, called $\lambda-$bracket, $R \otimes R \rightarrow \C[\lambda]\otimes R$, $a \otimes b \mapsto [a_{\lambda}b]$, that satisfies the following properties for all $a,b,c \in R$:
\begin{align*}
(1)\,\, &conformal \,\, sesquilinearity: &&[\partial a_{\lambda}b]=-\lambda [a_{\lambda}b], \quad  [a_{\lambda} \partial b]=(\lambda+\partial)[a_{\lambda}b]; \\
(2)\,\, &skew-symmetry:             &&[a_{\lambda}b]=-(-1)^{p(a)p(b)}[b_{-\lambda-\partial}a] ;    \\
(3)\,\, &Jacobi \,\, identity:           &&[a_{\lambda}[b_{\mu}c]]=[[a_{\lambda}b]_{\lambda+\mu}c]+(-1)^{p(a)p(b)}[b_{\mu}[a_{\lambda}c]];
\end{align*}
where $p(a)$ denotes the parity of the element $a \in R$ and $p(\partial a)=p(a)$ for all $a \in R$.
\end{defi}
We call $n-$products the coefficients $(a_{(n)}b)$ that appear in $[a_{\lambda}b]=\sum_{j\geq 0} \frac{\lambda^{j}}{j!}(a_{(j)}b)$  and give an equivalent definition of conformal superalgebra.
\begin{defi}[Conformal superalgebra]
A \textit{conformal superalgebra} $R$ is a left $\Z_{2}-$graded $\C[\partial]-$module endowed with a $\C-$bilinear product $(a_{(n)}b): R\otimes R\rightarrow R$, defined for every $n \geq 0$, that satisfies the following properties for all $a,b,c \in R$, $m,n \geq 0$:
\begin{enumerate}
  \item  $(a_{(n)}b)=0, \,\, for \,\, n \gg 0$;
	\item $({\partial a}_{(0)}b)=0$ and $({\partial a}_{(n+1)}b)=-(n+1) (a_{(n)}b)$;
	\item $(a_{(n)}b)=-(-1)^{p(a)p(b)}\sum_{j \geq 0}(-1)^{j+n} \frac{\partial^{j}}{j!}(b_{(n+j)}a)$;
	\item $ (a_{(m)}(b_{(n)}c))=\sum^{m}_{j=0}\binom{m}{j}((a_{(j)}b)_{(m+n-j)}c)+(-1)^{p(a)p(b)}(b_{(n)}(a_{(m)}c))$;
\end{enumerate}
where $p(\partial a)=p( a)$ for all $a \in R$.
\end{defi}
Using conditions $(2)$ and $(3)$ it is easy to show that for all $a,b \in R$, $n \geq 0$:
\begin{equation*}
(a_{(n)}\partial b)=\partial (a_{(n)} b)+n (a_{(n-1)}b).
\end{equation*}
Due to this relation and $(2)$, the map $\partial: R\rightarrow R$, $a \mapsto \partial a$ is a derivation with respect to the $n-$products.
We say that a conformal superalgebra $R$ is \textit{finite} if it is finitely generated as a $\C[\partial]-$module. 
An \textit{ideal} $I$ of $R$ is a $\C[\partial]-$submodule of $R$ such that $a_{(n)}b\in I$ for every $a \in R$, $b \in I$, $n \geq 0$. A conformal superalgebra $R$ is \textit{simple} if it has no non-trivial ideals and the $\lambda-$bracket is not identically zero. We denote by $R'$ the \textit{derived subalgebra} of $R$, i.e. the $\C-$span of all $n-$products.
\begin{defi}
A module $M$ over a conformal superalgebra $R$ is a left $\Z_{2}-$graded $\C[\partial]-$module endowed with $\C-$linear maps $R \rightarrow \End_{\C} M$, $a\mapsto a_{(n)}$, defined for every $n \geq 0$, that satisfy the following properties for all $a,b \in R$, $v \in M$, $m,n \geq 0$:
\begin{enumerate}
  \item[(i)] $a_{(n)}v=0 \,\, for \,\, n \gg 0$;
	\item[(ii)] $(\partial a)_{(n)}v=[\partial,a_{(n)}]v=-n a_{(n-1)}v $;
	\item[(iii)] $[a_{(m)},b_{(n)}]v=\sum_{j=0}^{m} \binom{m}{j}(a_{(j)}b)_{(m+n-j)}v$.
\end{enumerate}
\end{defi}
For an $R-$module $M$, we define for all $a \in R$ and $v\in M$:
\begin{align*}
a_{\lambda}v=\sum_{n \geq 0}\frac{\lambda^{n}}{n!}a_{(n)}v.
\end{align*}
A module $M$ is called \textit{finite} if it is a finitely generated $\C[\partial]-$module.\\
Let us recall the construction of the annihilation superalgebra associated with a conformal superalgebra $R$.
Let $\widetilde{R}=R[y,y^{-1}]$, set $p(y)=0$ and $\widetilde{\partial}=\partial+\partial_{y}$. We define the following $n-$products on $\widetilde{R}$, for all $a,b \in R$, $f,g \in \C[y,y^{-1}]$, $k\geq 0$:
\begin{align*}
(af_{(k)}bg)=\sum_{j \in \Z_{+}}(a_{(k+j)}b) \Big(\frac{\partial_{y}^{j}}{j!}f \Big)g .
\end{align*}
In particular if $f=y^{m}$ and $g=y^{n}$ we have for all $n \geq 0$:
\begin{align*}
({ay^{m}}_{ (k)}by^{n})=\sum_{j \in \Z_{+}}\binom{m}{j}(a_{(k+j)}b)y^{m+n-j}.
\end{align*}
We observe that $\widetilde{\partial}\widetilde{R}$ is a two sided ideal of $\widetilde{R}$ with respect to the $0-$product. The quotient $\Lie R:=\widetilde{R}/ \widetilde{\partial}\widetilde{R}$ has a structure of Lie superalgebra with the bracket induced by the $0-$product, i.e. for all $a,b \in R$, $f,g \in \C[y,y^{-1}]$: 
\begin{align}
\label{bracketannihilation}
[af,bg]=\sum_{j \in \Z_{+}}(a_{( j)}b)\Big(\frac{\partial_{y}^{j}}{j!}f \Big)g .
\end{align}
\begin{defi}
The annihilation superalgebra $\mathcal{A}(R)$ of a conformal superalgebra $R$ is the subalgebra of $\Lie R$ spanned by all elements $ay^{n}$ with $n\geq 0$ and $a\in R$. \\
The extended annihilation superalgebra $\mathcal{A}(R)^{e}$ of a conformal superalgebra $R$ is the Lie superalgebra $\C \partial \ltimes \mathcal{A}(R)$. The semidirect sum $\C \partial \ltimes \mathcal{A}(R)$ is the vector space $\C \partial \oplus \mathcal{A}(R)$ endowed with the structure of Lie superalgebra determined by the bracket:
\begin{align*}
[\partial,ay^{m}]=-\partial_{y}(ay^{m})=-m a y^{m-1},
\end{align*}
for all $a \in R$ and the fact that $\C \partial$, $\mathcal{A}(R)$ are Lie subalgebras.
\end{defi}
For all $a \in R$ we consider the following formal power series in $\mathcal{A}(R)[[\lambda]]$:
\begin{align}
\label{powerseries}
a_{\lambda}=\sum_{n \geq 0}\frac{\lambda^{n}}{n!}ay^{n}.
\end{align}
For all $a,b \in R$, we have: $[a_{\lambda},b_{\mu}]=[a_{\lambda}b]_{\lambda+\mu}$ and $(\partial a)_{\lambda}=-\lambda a_{\lambda}$ (for a proof see \cite{cantacasellikac}).  
\begin{prop}[\cite{chengkac}]
\label{propcorrispmoduli}
Let $R$ be a conformal superalgebra.
	If $M$ is an $R$-module then $M$ has a natural structure of $\mathcal A(R)^e$-module, where the action of $ay^n$ on $M$ is uniquely determined by $a_\lambda v=\sum_{n \geq 0}\frac{\lambda^{n}}{n!}ay^{n}.v$ for all $v\in M$. Viceversa if $M$ is a $\mathcal A(R)^e$-module such that for all $a\in R$, $v\in M$ we have $ay^n.v=0$ for $n\gg0$, then $M$ is also an $R$-module by letting $a_\lambda v=\sum_{n}\frac{\lambda^{n}}{n!}ay^{n}.v$.
\end{prop}
Proposition \ref{propcorrispmoduli} reduces the study of modules over a conformal superalgebra $R$ to the study of a class of modules over its (extended) annihilation superalgebra.
The following proposition states that, under certain hypotheses, it is sufficient to consider the annihilation superalgebra. We recall that, given a $\Z-$graded Lie superalgebra $\g=\oplus_{i \in \Z}\g_{i}$, we say that $\g$ has finite depth $d\geq0$ if $\g_{-d}\neq 0$ and $\g_{i}=0$ for all $i<-d$.
\begin{prop}[\cite{kac1},\cite{chenglam}]
\label{keythmannihi}
Let $\g$ be the annihilation superalgebra of a conformal superalgebra $R$. Assume that $\g$ satisfies the following conditions:
\begin{description}
	\item[L1] $\g$ is $\Z-$graded with finite depth $d$;
	\item[L2] There exists an element whose centralizer in $\g$ is contained in $\g_{0}$;
	\item[L3] There exists an element $\Theta \in \g_{-d}$ such that $\g_{i-d}=[\Theta,\g_{i}]$, for all $i\geq 0$.
\end{description}
 Finite modules over $R$ are the same as modules $V$ over $\g$, called \textit{finite conformal}, that satisfy the following properties:
\begin{enumerate}
	\item for every $v \in V$, there exists $j_{0} \in \Z$, $j_{0}\geq -d$, such that $\g_{j}.v=0$ when $j\geq j_{0}$;
	\item $V$ is finitely generated as a $\C[\Theta]-$module.
\end{enumerate}
\end{prop}
\begin{rem}
\label{gradingelement}
We point out that condition \textbf{L2} is automatically satisfied when $\g$ contains a \textit{grading element}, i.e. an element $t \in \g$ such that $[t,b]=\degr (b) b$ for all $b \in \g$.
\end{rem}
Let $\g=\oplus_{i \in \Z} \g_{i}$ be a $\Z-$graded Lie superalgebra. We will use the notation $\g_{>0}=\oplus_{i>0}\g_{i}$, $\g_{<0}=\oplus_{i<0}\g_{i}$ and $\g_{\geq 0}=\oplus_{i\geq 0}\g_{i}$. We denote by $U(\g) $ the universal enveloping algebra of $\g$.
\begin{defi}
Let $F$ be a $\g_{\geq 0}-$module. The \textit{generalized Verma module} associated with $F$ is the $\g-$module $\Ind (F)$ defined by:
\begin{equation*}
\Ind (F):= \Ind ^{\g}_{\g_{\geq 0}} (F)=U(\g) \otimes _{U(\g_{\geq 0})} F.
\end{equation*}
\end{defi}
If $F$ is a finite$-$dimensional irreducible $\g_{\geq 0}-$module we will say that $\Ind(F)$ is a \textit{finite Verma module}. We will identify $\Ind (F)$ with $U(\g_{<0}) \otimes  F$ as vector spaces via the Poincar\'e$-$Birkhoff$-$Witt Theorem. The $\Z-$grading of $\g$ induces a $\Z-$grading on $U(\g_{<0})$ and $\Ind (F)$. We will invert the sign of the degree, so that we have a $\Z_{\geq 0}-$grading on $U(\g_{<0})$ and $\Ind (F)$. We will say that an element $v \in U(\g_{<0})_{k}$ is homogeneous of degree $k$. Analogously an element $m \in U(\g_{<0})_{k}  \otimes  F$ is homogeneous of degree $k$. For a proof of the following proposition see \cite{bagnolicaselli}.
\begin{prop}
Let $\g=\oplus_{i \in \Z} \g_{i}$ be a $\Z-$graded Lie superalgebra. If $F$ is an irreducible finite$-$dimensional $\mathfrak{g}_{\geq 0}-$module, then $\Ind (F)$ has a unique maximal submodule. We denote by $\I (F)$ the quotient of $\Ind (F)$ by the unique maximal submodule.
\end{prop}
\begin{defi}
Given a $\g-$module $V$, we call \textit{singular vectors} the elements of:
\begin{align*}
\Sing (V) =\left\{v \in V \,\, | \, \, \g_{>0}.v=0\right\}.
\end{align*}
Homogeneous components of singular vectors are still singular vectors so we often assume that singular vectors are homogeneous without loss of generality.
In the case $V=\Ind (F)$, for a $\g_{\geq 0}-$module $F$, we will call \textit{trivial singular vectors} the elements of $\Sing (V) $ of degree 0 and \textit{nontrivial singular vectors} the nonzero elements of $\Sing (V) $ of positive degree.
\end{defi}
\begin{thm}[\cite{kacrudakov},\cite{chenglam}]
\label{keythmsingular}
Let $\g$ be a Lie superalgebra that satisfies \textbf{L1}, \textbf{L2}, \textbf{L3}, then:
\begin{enumerate}
\item[(i)] if $F$ is an irreducible finite$-$dimensional $\mathfrak{g}_{\geq 0}-$module, then $\mathfrak{g}_{> 0}$ acts trivially on it;
	\item[(ii)] the map $F \mapsto \I (F)$ is a bijective map between irreducible finite$-$dimensional $\mathfrak{g}_{ 0}-$modules and irreducible finite conformal $\mathfrak{g}-$modules;
	\item[(iii)] the $\mathfrak{g}-$module $\Ind (F)$ is irreducible if and only if the $\mathfrak{g}_{0}-$module $F$ is irreducible and $\Ind (F)$ has no nontrivial singular vectors.
	\end{enumerate}
	\end{thm}
We recall the notion of duality for conformal modules (see for further details \cite{bklr}, \cite{cantacasellikac}, \cite{bagnoli2}). Let $R$ be a conformal superalgebra and $M$ a conformal module over $R$.
\begin{defi}
The conformal dual $M^{*}$ of $M$ is defined by:
\begin{align*}
M^{*}=\left\{f_{\lambda}:M\rightarrow \C[\lambda] \,\, | \,\, f_{\lambda}(\partial m)=\lambda f_{\lambda}(m), \,\, \forall m \in M\right\}.
\end{align*}
The structure of $\C[\partial]-$module is given by $(\partial f)_{\lambda}(m)=-\lambda  f_{\lambda}(m)$, for all $f \in M^{*}$, $ m \in M$. The $\lambda-$action of $R$ is given, for all $a \in R$, $m \in M$, $f \in M^{*}$, by:
\begin{align*}
(a_{\lambda}f)_{\mu}(m)=-(-1)^{p(a)p(f)}f_{\mu-\lambda}(a_{\lambda}m).
\end{align*}
\end{defi}
\begin{defi}
Let $T:M\rightarrow N$ be a morphism of $R-$modules, i.e. a linear map such that for all $a\in R$ and $m\in M$:
\begin{itemize}
	\item [i:] $T(\partial m)=\partial T(m)$,
	\item [ii:] $T(a_{\lambda} m)=a_{\lambda} T(m)$.
\end{itemize}
The dual morphism $T^{*}:N^{*} \rightarrow M^{*}$ is defined, for all $f \in N^{*}$ and $m \in M$, by:
\begin{align*}
\left[T^{*}(f)\right]_{\lambda}(m)=-f_{\lambda}\left(T(m)\right).
\end{align*}
\end{defi}
\begin{thm}[\cite{bklr}, Proposition 2.6]
\label{dualeconforme}
Let $R$ be a conformal superalgebra and $M,N$ $R-$modules. Let $T:M\longrightarrow N$ be a homomorphism of $R-$modules such that $N/ \Ima T$ is a finitely generated torsion$-$free $\C[\partial]-$module. Then the standard map $\Psi:N^{*} / \Ker T^{*}\longrightarrow(M/ \Ker T)^{*}$, given by $[\Psi(\overline{f})]_{\lambda}(\overline{m})=f_{\lambda}(T(m))$
 (where by the bar we denote the corresponding class in the quotient), is an isomorphism of $R-$modules.
\end{thm}
We denote by $\textbf{F}$ the functor that maps a conformal module $M$ over a conformal superalgebra $R$ to its conformal dual $M^{*}$ and maps a morphism between conformal modules $T:M \rightarrow N$ to its dual $T^{*}:N^{*} \rightarrow M^{*}$. For a proof of the following proposition, see \cite{bagnoli2}.
\begin{prop}[\cite{bagnoli2}]
\label{esattezzafuntoreduale}
The functor $\textbf{F}$ is exact if we consider only morphisms $T:M \rightarrow N$, where $N/ \Ima T$ is a finitely generated torsion free $\C[\partial]-$module.
\end{prop}
\section{The conformal superalgebra $CK_{6}$}
In this section we recall the definition and some properties of the conformal superalgebra $CK_{6}$ from \cite{ck6}.
Let $\inlinewedge(N)$ be the Grassmann superalgebra in the $N$ odd indeterminates $\xi_{1},...,\xi_{N}$. Let $t$ be an even indeterminate and $\inlinewedge (1,N)=\C[t,t^{-1}] \otimes \inlinewedge(N)$. We recall that the Lie superalgebra of derivations of $\inlinewedge (1,N)$ is defined as follows:
\begin{equation*}
W(1,N)=\bigg\{ D=a \partial_{t}+\sum ^{N}_{i=1} a_{i} \partial_{i} \,\, | \,\, a,a_{i} \in \displaywedge (1,N)\bigg\},
\end{equation*}
where $\partial_{t}=\frac{\partial}{\partial{t}}$ and $\partial_{i} =\frac{\partial}{\partial{\xi_{i}}}$ for every $i \in \left\{1,...,N \right\}$.\\
Let us consider the contact form $\omega = dt-\sum_{i=1}^{N}\xi_{i} d\xi_{i} $. The contact Lie superalgebra $K(1,N)$ is defined by:
\begin{equation*}
K(1,N)=\left\{D \in W(1,N) \,\, | \,\, D\omega=f_{D}\omega \,\, for \,\, some \,\, f_{D} \in \displaywedge (1,N)\right\}.
\end{equation*}
Analogously, let $\inlinewedge  (1,N)_{+}=\C[t] \otimes \inlinewedge (N)$. We define the Lie superalgebra $W(1,N)_{+}$ (resp. $K(1,N)_{+}$) similarly to $W(1,N)$ (resp. $K(1,N)$) using $\inlinewedge  (1,N)_{+}$ instead of $\inlinewedge  (1,N)$.
We can define on $\inlinewedge (1,N)$ a Lie superalgebra structure as follows. For all $f,g \in \inlinewedge(1,N)$ we let:
\begin{equation}
\label{bracketlie}
[f,g]=\Big(2f-\sum_{i=1}^{N} \xi_{i}  \partial_{i} f \Big)(\partial_{t}{g})-(\partial_{t}{f})\Big(2g-\sum_{i=1}^{N} \xi_{i} \partial_{i} g\Big)+(-1)^{p(f)}\Big(\sum_{i=1}^{N} \partial_{i}  f \partial_{i}  g \Big).
\end{equation}
We recall that $K(1,N) \cong \inlinewedge(1,N)$ as Lie superalgebras via the following map (see \cite{{chengkac2}}):
\begin{gather*}
\displaywedge(1,N) \longrightarrow  K(1,N) \\
f \longmapsto 2f \partial_{t}+(-1)^{p(f)} \sum_{i=1}^{N} (\xi_{i} \partial_{t} f+ \partial_{i}f )(\xi_{i} \partial_{t} + \partial_{i}).
\end{gather*}
We will always identify elements of $K(1,N) $ with elements of $\inlinewedge(1,N)$ and we will omit the symbol $\wedge$ between the $\xi_{i}$'s. We consider on $K(1,N)$ the standard grading, i.e. for every $t^{m} \xi_{i_{1}} \cdots \xi_{i_{s}} \in K(1,N)$ we have $\degr(t^{m} \xi_{i_{1}} \cdots \xi_{i_{s}})=2m+s-2$.\\
We recall that the conformal superalgebra of type $K$ is defined as $K_{N}:=\C[\partial] \otimes \inlinewedge(N)$. 
On $K_{N}$ the $\lambda-$bracket for $f,g \in \inlinewedge(N)$, $f= \xi_{i_{1}} \cdots \xi_{i_{r}}$ and $g= \xi_{j_{1}} \cdots \xi_{j_{s}}$, is given by (see \cite{kac1},\cite{fattorikac}):
\begin{align*}
[f_{\lambda}g]=(r-2)\partial(fg)+(-1)^{r}\sum^{N}_{i=1}(\partial_{i}f)(\partial_{i}g)+\lambda(r+s-4)fg.
\end{align*}
The associated annihilation superalgebra is (see \cite{kac1},\cite{fattorikac}):
\begin{align*}
\mathcal{A}(K_{N})=K(1,N)_{+}.
\end{align*}
We adopt the following notation: we denote by $\mathcal I$ the set of finite sequences of elements in $\{1,\ldots,N\}$; we will write $I=i_1\cdots i_r$ instead of $I=(i_1,\ldots,i_r)$. 
If $I=i_1 \cdots i_r\in \mathcal I$  we let $\xi_{I}=\xi_{i_1}\cdots \xi_{i_r}$ and $|\xi_{I}|=|I|=r$. 
 We focus on $N=6$. We let $\xi_{*}=\xi_{123456}$.
Following \cite{ck6}, for $\xi_{I} \in \inlinewedge(6)$ we define the \textit{modified Hodge dual} $\xi^{*}_{I}$ to be the unique monomial such that $\xi_{I}\xi^{*}_{I}=\xi_{*}$. We extend the definition of modified Hodge dual to elements $\sum_{k,I} \alpha_{k,I}t^{k}\xi_{I}\in \inlinewedge(1,6)_{+}$ letting $(\sum_{k,I} \alpha_{k,I}t^{k}\xi_{I})^{*}=\sum_{k,I} \alpha_{k,I}t^{k}\xi_{I}^{*}$.\\
The conformal superalgebra $CK_{6}$ is the subalgebra of $K_{6}$ defined by (see construction in \cite{new6}):
\begin{align*}
CK_{6}=\C[\partial]-\spann\left\{f-i(-1)^{\frac{|f|(|f|+1)}{2}}(-\partial)^{3-|f|}f^{*}, \, f \in \displaywedge(6), 0 \leq |f| \leq 3   \right\}.
\end{align*}
We introduce the linear operator $A: K(1,6)_{+}  \longrightarrow K(1,6)_{+}$ that is given for monomials with $d$ odd variables by:
\begin{align*}
A(f)=(-1)^{\frac{d(d+1)}{2}}\left(\frac{d}{dt}\right)^{3-d}f^{*},
\end{align*}
where $\left(\frac{d}{dt}\right)^{-1}$ indicates the integration with respect to $t$; $A$ is extended by linearity.
The annihilation superalgebra associated with $CK_{6}$ is the subalgebra of $K(1,6)_{+}$ given by the image of $Id-iA$; it is isomorphic to the Lie superalgebra $\g:=E(1,6)$ (see \cite{chengkac2}), that is one of the exceptional Lie superalgebras classified in \cite{kac98,Shche}; on $\g$ the bracket is given by \eqref{bracketlie}.
The map $A$ preserves the $\Z-$grading, then $E(1,6)$ inherits the $\Z-$grading. The non$-$positive spaces of $E(1,6)$ and $K(1,6)_{+}$ coincide and are:
\begin{align*}
&E(1,6)_{-2}=\langle 1 \rangle, \\
&E(1,6)_{-1}=\langle \xi_{1},\xi_{2},...,\xi_{6} \rangle, \\
&E(1,6)_{0}=\langle t,\xi_{i}\xi_{j}, \,\, 1\leq i,j\leq 6 \rangle .
\end{align*}
In the following we will call $\Theta$ the element $-\frac{1}{2} \in \g_{-2}$.
Let us focus on $\g_{0}=\langle t, \xi_{i}\xi_{j} \quad 1 \leq i<j \leq 6 \rangle \cong  \C t \oplus \mathfrak{sl}(4)$. 
We take as a basis of a Cartan subalgebra $\mathfrak{h}$ the following:
\begin{align*}
h_{1}=-i \xi_{34}-i\xi_{56}, \quad h_{2}=-i \xi_{12}+i\xi_{34}, \quad h_{3}=-i \xi_{34}+i\xi_{56}.
\end{align*}
Let $\alpha_{1},\alpha_{2},\alpha_{3}  \in \mathfrak{h}^{*}$ such that $\alpha_{i}(h_{j})=2\delta_{i,j}$. The set of positive roots is $\Delta^{+}=\left\{\alpha_{1},\alpha_{2},\alpha_{3}\right\}$. We define:
\begin{align*}
e_{1}=\frac{-\xi_{35}+\xi_{46}+i\xi_{36}+i\xi_{45}}{2};\\
e_{2}=\frac{-\xi_{13}-\xi_{24}-i\xi_{14}+i\xi_{23}}{2};\\
e_{3}=\frac{-\xi_{35}-\xi_{46}-i\xi_{36}+i\xi_{45};}{2}\\
f_{1}=\frac{\xi_{35}-\xi_{46}+i\xi_{36}+i\xi_{45}}{2};\\
f_{2}=\frac{\xi_{13}+\xi_{24}-i\xi_{14}+i\xi_{23}}{2};\\
f_{3}=\frac{\xi_{35}+\xi_{46}-i\xi_{36}+i\xi_{45}}{2}.
\end{align*}
We have that $e_{i}$ is the root vector with respect to $\alpha_{i}$ and $f_{i}$ is the root vector with respect to $-\alpha_{i}$.
From now on we will always use the following identification of $\mathfrak{sl}(4)$ with vector fields:
\begin{gather*}
x_{2} \partial_{1}  \leftrightarrow f_{1}, \\
x_{1} \partial_{2}   \leftrightarrow e_{1}, \\
x_{1} \partial_{1}-x_{2} \partial_{2}   \leftrightarrow h_{1} ,\\
x_{3} \partial_{2}   \leftrightarrow f_{2}, \\
x_{2} \partial_{3}   \leftrightarrow e_{2} ,\\
x_{2} \partial_{2}-x_{3} \partial_{3}   \leftrightarrow h_{2}, \\
x_{4} \partial_{3} \leftrightarrow f_{3}, \\
x_{3} \partial_{4} \leftrightarrow e_{3}, \\
x_{3} \partial_{3}-x_{4} \partial_{4}   \leftrightarrow h_{3}.
\end{gather*}
\begin{rem}
\label{L3}
We point out that $\g$ satisfies \textbf{L1},\textbf{L2},\textbf{L3}. Indeed \textbf{L1} and \textbf{L2} are obvious since $\g$ is $\Z-$graded of finite depth 2 and $t$ is a grading element. \textbf{L3} is satisfied by the choice of $\Theta$: indeed it is a straightforward verification, using \eqref{bracketlie}, that $$t^{m}\xi_I-i(-1)^{\frac{|I|(|I|+1)}{2}}\left(\frac{d}{dt}\right)^{3-|I|}t^{m}\xi_I^{*}=-\frac{1}{m+1}\left[\Theta,t^{m+1}\xi_I-i(-1)^{\frac{|I|(|I|+1)}{2}}\left(\frac{d}{dt}\right)^{3-|I|}t^{m+1}\xi_I^{*}\right],$$
for all $m\geq 0$ and $I \in \mathcal I$.
\end{rem}
 We will denote by $M(n_{1},n_{2},n_{3},n_{0})$ the finite Verma module $\Ind (F(n_{1},n_{2},n_{3},n_{0}))$ over $\g$, where $n_{0}$ is the weight with respect to the central element $t \in \g_0$ and $n_{i}$ is the weight with respect to $h_{i}$. Let $\lambda=(n_{1},n_{2},n_{3})$ be a dominant weight for $\mathfrak{sl}(4)$, we use the notation $F(n_{1},n_{2},n_{3})$ to indicate the irreducible $\mathfrak{sl}(4)-$module of highest weight $\lambda$.
We point out that, since $M(n_{1},n_{2},n_{3},n_{0}) \cong U(\mathfrak{g}_{<0}) \otimes F(n_{1},n_{2},n_{3},n_{0})$, it follows that $M(n_{1},n_{2},n_{3},n_{0})\cong \C[\Theta] \otimes \inlinewedge(6) \otimes F(n_{1},n_{2},n_{3},n_{0})$. Indeed, let us denote by $\eta_{i}$ the image in $U(\g)$ of $\xi_{i} \in  \inlinewedge(6) $, for all $i \in \left\{1,2,3,4,5,6\right\}$.
	In $U(\g)$ we have that $\eta_{i}^{2}=\Theta$, for all $i \in \left\{1,2,3,4,5,6\right\}$: since $[\xi_{i},\xi_{i}]=-1$ in $\g$, we have $\eta_{i}\eta_{i}=-\eta_{i}\eta_{i}-1$ in $U(\g)$. From now on it is always assumed that $F(n_{1},n_{2},n_{3},n_{0})$ is a finite$-$dimensional irreducible $\mathfrak{g}_{\geq 0}-$module. 
\begin{rem}
\label{remwck6}
By a straightforward computation, it is possible to show that $\g_{-1}$ is an irreducible $\mathfrak{sl}_{4}-$module of highest weight (0,1,0) with respect to $h_{1},h_{2},h_{3}$. In particular $\g_{-1}$ is isomorphic to $\inlinewedge^{2}\C^{4}$ and the isomorphism is given by:
\begin{align}
\label{notazionig-1ck6}
&\xi_{2}+i\xi_{1} \longleftrightarrow x_{12}, &&\xi_{2}-i\xi_{1} \longleftrightarrow -x_{34},\\ \nonumber
&\xi_{4}+i\xi_{3} \longleftrightarrow x_{13}, &&\xi_{4}-i\xi_{3} \longleftrightarrow x_{24},\\ \nonumber
&\xi_{6}+i\xi_{5} \longleftrightarrow x_{14}, &&\xi_{6}-i\xi_{5} \longleftrightarrow -x_{23}.
\end{align}
Motivated by relations \eqref{notazionig-1ck6}, from now on we will use the notation:
\begin{align} 
\label{notazioinewck6}
&w_{12}=\eta_{2}+i\eta_{1}, &&w_{34}=-\eta_{2}+i\eta_{1},\\ \nonumber
&w_{13}=\eta_{4}+i\eta_{3}, &&w_{24}=\eta_{4}-i\eta_{3},\\ \nonumber
&w_{14}=\eta_{6}+i\eta_{5}, &&w_{23}=-\eta_{6}+i\eta_{5}.
\end{align}
We point out that $[w_{12},w_{34}]=-4\Theta$, $[w_{13},w_{24}]=4\Theta$, $[w_{14},w_{23}]=-4\Theta$ and all the other brackets between the $w$'s are zero. Moreover, in $U(\g_{<0})$, $w_{kl}^{2}=0$ for all $1 \leq k<l \leq 4$: indeed, for example, $w_{12}^{2}=(\eta_{2}+i\eta_{1})(\eta_{2}+i\eta_{1})=\Theta+i\eta_{2}\eta_{1}+i\eta_{1}\eta_{2}-\Theta=0$.
\end{rem}
\begin{rem}
\label{remlowest}
We have, by straightforward computation, that $E(1,6)_{1}$ is the sum of two irreducible $\g_{0}-$modules, in particular $\g_{1}\cong F(0,1,0) \oplus F(2,0,0)$ and the following are the corresponding lowest weight vectors in $E(1,6)_{1}$:
\begin{align*}
v_{1}&= t\xi_{1}+it\xi_{2},\\
v_{2}&=-\xi_{1}\xi_{3}\xi_{5}-i\xi_{2}\xi_{4}\xi_{6}+\xi_{2}\xi_{4}\xi_{5}+i\xi_{1}\xi_{3}\xi_{6}-\xi_{1}\xi_{4}\xi_{6}-i\xi_{2}\xi_{3}\xi_{5}-\xi_{2}\xi_{3}\xi_{6}-i\xi_{1}\xi_{4}\xi_{5}.
\end{align*}
Since $\g_{i}=\g_{1}^{i}$ for all $i\geq 2$ (see Section 4.2 in \cite{chengkac2}), then in order to check whether a vector $\vec{m} \in \Ind(F)$ is an highest weight singular vector it is sufficient to check that it is annihilated by $e_{1},e_{2},e_{3},v_{1},v_{2}$.
\end{rem}
In \cite{ck6}, Boyallian, Kac and Liberati, obtained a complete classification of all the highest weight singular vectors for $CK_{6}$. For the classification the authors used a bound on the degree of singular vectors that was proved in \cite{bagnoli1}. Since, by Remark \ref{L3}, $\g$ satisfies \textbf{L1},\textbf{L2},\textbf{L3}, by Theorem \ref{keythmsingular} Boyallian, Kac and Liberati obtained the following complete classification of degenerate finite Verma modules for $CK_6$.
\begin{thm}[\cite{ck6}]
\label{teoremavettorisingolarick6ultimo}
Let $F$ be an irreducible finite$-$dimensional $\g_{0}-$module of highest weight $\mu=(n_{1},n_{2},n_{3},n_{0})$. Therefore $ \Ind(F)$ is degenerate if and only if $\mu$ is one of the following:\\
\begin{enumerate}
	\item[\textbf{A}:] $\mu=(n_{1},n_{2},0,-n_{2}-\frac{n_{1}}{2})$ with $n_{1} \geq 0$ and $n_{2} \geq 0$,\\
	\item[\textbf{B}:] $\mu=(n_{1},0,n_{3},\frac{n_{3}}{2}-\frac{n_{1}}{2}+2)$ with $n_{1} \geq 0$ and $n_{3} \geq 0$,\\
	\item[\textbf{C}:] $\mu=(0,n_{2},n_{3},n_{2}+\frac{n_{3}}{2}+4)$ with $n_{2} \geq 0$, $n_{3} \geq 0$  and $(n_{2},n_{3})\neq (0,0)$.
	\end{enumerate}
\end{thm}
We introduce the following notation
\begin{align*}
M^{n_{1},n_{2}}_{A}&:=M\left( n_{1},n_{2},0,-n_{2}-\frac{n_{1}}{2}\right)=U(\g_{-}) \otimes V^{n_{1},n_{2}}_{A},\\
M^{n_{1},n_{3}}_{B}&:=M\left(n_{1},0,n_{3}, \frac{n_{3}}{2}-\frac{n_{1}}{2}+2\right)=U(\g_{-}) \otimes V^{n_{1},n_{3}}_{B},\\
M^{n_{2},n_{3}}_{C}&:=M\left( 0,n_{2},n_{3},n_{2}+\frac{n_{3}}{2}+4\right)=U(\g_{-}) \otimes V^{n_{2},n_{3}}_{C},
\end{align*}
where as $\mathfrak{sl}(4)-$modules, $V^{n_{1},n_{2}}_{A}\cong F(n_{1},n_{2},0)$, $V^{n_{1},n_{3}}_{B}\cong F(n_{1},0,n_{3})$, $V^{n_{2},n_{3}}_{C}\cong F(0,n_{2},n_{3})$. The element $t$ acts as multiplication by $-n_{2}-\frac{n_{1}}{2}$ on $V^{n_{1},n_{2}}_{A}$, as multiplication by $\frac{n_{3}}{2}-\frac{n_{1}}{2}+2$ on $V^{n_{1},n_{3}}_{B}$ and as multiplication by $n_{2}+\frac{n_{3}}{2}+4$ on $V^{n_{2},n_{3}}_{C}$.
\begin{rem}
 We will think $V^{n_{1},n_{2}}_{A}$ as the irreducible submodule of 
\begin{equation*}
\Sym^{n_{1}}(\C^{4}) \otimes \Sym^{n_{2}}(\displaywedge^{2}\C^{4} ) 
\end{equation*}
generated by the highest weight vector $v_{\lambda}:={x_{1}}^{n_{1}}x_{12}^{n_{2}}$, where $\left\{x_{1},x_{2},x_{3},x_{4}\right\}$ is the standard basis of $\C^{4}$ and $\left\{x_{ij}=x_i \wedge x_j,1\leq i <j \leq 4\right\}$ is the standard basis of $\displaywedge^{2}\C^{4}$. 
We will think $V^{n_{1},n_{3}}_{B}$ as the irreducible submodule of  
\begin{equation*}
   \Sym^{n_{1}}(\C^{4})  \otimes \Sym^{n_{3}}((\C^{4})^{*})
\end{equation*}
generated by the highest weight vector $v_{\lambda}:=x_{1}^{n_{1}}{x_{4}^{*}}^{n_{3}}$, where $\left\{x_{1}^{*},x_{2}^{*},x_{3}^{*},x_{4}^{*}\right\}$ is a basis for $(\C^{4})^{*}$.\\
We will think $V^{n_{2},n_{3}}_{C}$ as the irreducible submodule of 
\begin{equation*}
 \Sym^{n_{2}}(\displaywedge^{2}(\C^{4})^{*})\otimes  \Sym^{n_{3}}((\C^{4})^{*})  
\end{equation*}
generated by the highest weight vector $v_{\lambda}:= {x^{*}_{34}}^{n_{2}}{x^{*}_{4}}^{n_{3}}$, where $\left\{x^{*}_{ij}=x^{*}_{i} \wedge x^{*}_{j},1\leq i <j \leq 4\right\}$ is a basis for $\inlinewedge^{2}(\C^{4})^{*}$.\\
We observe that $t$ acts as $-\frac{x_{1}\partial_{x_{1}}+x_{2}\partial_{x_{2}}+x_{3}\partial_{x_{3}}+x_{4}\partial_{x_{4}}}{2}$ on vectors of $V^{n_{1},n_{2}}_{A}$, as $-\frac{x_{1}\partial_{x_{1}}+x_{2}\partial_{x_{2}}+x_{3}\partial_{x_{3}}+x_{4}\partial_{x_{4}}}{2}+2$ on vectors of $V^{n_{1},n_{3}}_{B}$ and as $-\frac{x_{1}\partial_{x_{1}}+x_{2}\partial_{x_{2}}+x_{3}\partial_{x_{3}}+x_{4}\partial_{x_{4}}}{2}+4$ on vectors of $V^{n_{2},n_{3}}_{C}$.
\end{rem}
\begin{lem}
\label{L1L2L3L4}
Let us point out that 
\begin{align*}
V^{n_{1},n_{2}}_{A}=\left\{p \in \Sym^{n_{1}}(\C^{4}) \otimes \Sym^{n_{2}}(\displaywedge^{2}\C^{4}) \,\,\, |\,\,\, p \in \Ker D \cap \bigcap^{4}_{i=1}\Ker D_{i}\right\},
\end{align*}
where
\begin{align*}
D=\partial_{12}\partial_{34}-\partial_{13}\partial_{24}+\partial_{14}\partial_{23},\\
D_{1}=\partial_{23}\partial_{4}-\partial_{24}\partial_{3}+\partial_{34}\partial_{2},\\
D_{2}=\partial_{13}\partial_{4}-\partial_{14}\partial_{3}+\partial_{34}\partial_{1},\\
D_{3}=\partial_{12}\partial_{4}-\partial_{14}\partial_{2}+\partial_{24}\partial_{1},\\
D_{4}=\partial_{12}\partial_{3}-\partial_{13}\partial_{2}+\partial_{23}\partial_{1}.
\end{align*}
Thus in the following we will write equivalently that $p \in V^{n_1,n_2}_{A}$, or $p \in \Sym^{n_{1}}(\C^{4}) \otimes \Sym^{n_{2}}(\displaywedge^{2}\C^{4})$ and $p \in \Ker D\cap\bigcap^{4}_{i=1} \Ker D_i$, or $p \in \Sym^{n_{1}}(\C^{4}) \otimes \Sym^{n_{2}}(\displaywedge^{2}\C^{4})$ and $p$ can be obtained by the action of elements of $\slq$ on $x^{n_1}_{1}x^{n_2}_{12}$.
\end{lem}
\begin{proof}
We recall that the irreducible $\mathfrak{sl}_{5}-$module $F(n_1,n_2,0,0)$, where the weights are written with respect to $x_{1}\partial_{1}-x_{2}\partial_{2},x_{2}\partial_{2}-x_{3}\partial_{3},x_{3}\partial_{3}-x_{4}\partial_{4},x_{4}\partial_{4}-x_{5}\partial_{5}$, can be seen as the irreducible $\mathfrak{sl}_{5}-$submodule of $\Sym^{n_1}(\C^{5})\otimes \Sym^{n_2}(\inlinewedge^{n_2}\C^{5})$ generated by $x^{n_1}_{1}x^{n_2}_{12}$. In \cite{rudakovE510} it was shown that
\begin{align}
\label{sl5irr}
F(n_1,n_2,0,0)=\big\{ & p \in \Sym^{n_1}(\C^{5})\otimes \Sym^{n_2}(\displaywedge^{n_2}\C^{5}), \,\, | \,\, for \,\, all \,\, 1 \leq a,b,c,d \leq 5,\\ \nonumber
&(\partial_{ab}\partial_{cd}-\partial_{ac}\partial_{bd}+\partial_{ad}\partial_{bc})p=0, (\partial_{ab}\partial_{c}-\partial_{ac}\partial_{b}+\partial_{bc}\partial_{a})p=0  \big\}:=S^{n_1,n_2}_{5}.
\end{align}
We claim that 
\begin{align}
\label{sl4irr}
V^{n_{1},n_{2}}_{A}=\big\{ & p \in \Sym^{n_1}(\C^{4})\otimes \Sym^{n_2}(\displaywedge^{n_2}\C^{4}), \,\, | \,\, for \,\, all \,\, 1 \leq a,b,c,d \leq 4,\\ \nonumber
&(\partial_{ab}\partial_{cd}-\partial_{ac}\partial_{bd}+\partial_{ad}\partial_{bc})p=0, (\partial_{ab}\partial_{c}-\partial_{ac}\partial_{b}+\partial_{bc}\partial_{a})p=0  \big\}:=S^{n_1,n_2}_{4}.
\end{align}
By \eqref{sl5irr}, it is obvious that $V^{n_{1},n_{2}}_{A} \subset S^{n_1,n_2}_{4}$. We now show that the equality holds by proving that $S^{n_1,n_2}_{4}$ is irreducible as $\mathfrak{sl}_{4}-$module. Indeed let us suppose that $S^{n_1,n_2}_{4}$ is not irreducible as $\mathfrak{sl}_{4}-$module. Since $S^{n_1,n_2}_{5}$ is completely reducible as $\mathfrak{sl}_{4}-$module, this implies that there exists in $S^{n_1,n_2}_{4}$ a highest weight vector $v$, with respect to $\mathfrak{sl}_{4}$, different from $x^{n_1}_{1}x^{n_2}_{12}$. Since $ v \in S^{n_1,n_2}_{4}$, $x_{4} \partial_{5}v=0$; hence this implies that in $F(n_1,n_2,0,0)$ $v$ is a  highest weight vector, with respect to $\mathfrak{sl}_{5}$, different from $x^{n_1}_{1}x^{n_2}_{12}$. This contradicts \eqref{sl5irr}.
\end{proof}
\begin{rem}
\label{teoremavettorisingolarick6ultimo}
We point out that, in \cite{ck6}, the authors proved that there are only highest weight singular vectors of degree $5$, $3$ and $1$ and they obtained the explicit expressions for all the highest weight singular vectors. In particular, they showed that the modules $M^{n_{1},n_{2}}_A$, with $n_{1}\geq 0$ and $n_{2}\geq 0$, contain a unique (up to scalars) highest weight singular vector of degree $1$ given by:
$$\vec{m}_{1a}=w_{12} \otimes {x_{1}}^{n_{1}}{x_{12}}^{n_{2}}.$$
The modules $M^{n_{1},n_{3}}_B$, with $n_{1}\geq 0$ and $n_{3}\geq 1$, contain a unique (up to scalars) highest weight singular vector of degree $1$ that we call $\vec{m}_{1b}$.
The modules $M^{n_{1},0}_B$, with $n_{1}\geq 0$, contain a unique (up to scalars) highest weight singular vector of degree $3$ given by:
$$\vec{m}_{3b}=w_{12}w_{13}w_{14} \otimes x_{1}^{n_{1}}.$$
The modules $M^{n_{2},n_{3}}_C$, with $n_{2}\geq 1$ and $n_{3}\geq 0$, contain a unique (up to scalars) highest weight singular vector of degree $1$, that we call $\vec{m}_{1c}$.
The modules $M^{0,n_{3}}_C$, with $n_{3}\geq 1$, contain a unique (up to scalars) highest weight singular vector of degree $3$, that we call $\vec{m}_{3c}$.
The module $M^{0,1}_C$ contains a unique (up to scalars) highest weight singular vector of degree $5$, given by
\begin{align}
\label{m5c}
\vec{m}_{5c}=&\Theta^{2}w_{12}\otimes x^{*}_{2}+\Theta^{2}w_{13}\otimes x^{*}_{3}+\Theta^{2}w_{14}\otimes x^{*}_{4}+\Theta w_{12}w_{13}w_{14}\otimes x^{*}_{1}\\ \nonumber
&+\frac{\Theta w_{12}}{4}(w_{23}w_{14}-w_{14}w_{23}+w_{13}w_{24}-w_{24}w_{13})\otimes x^{*}_{2}+\frac{\Theta w_{13}}{4}(w_{23}w_{14}-w_{14}w_{23}-w_{12}w_{34}+w_{34}w_{12})\otimes x^{*}_{3}\\ \nonumber
&+\frac{\Theta w_{14}}{4}(w_{13}w_{24}-w_{24}w_{13}-w_{12}w_{34}+w_{34}w_{12})\otimes x^{*}_{4}-\frac{iw_{12}}{16}(w_{13}w_{24}-w_{24}w_{13})(w_{23}w_{14}-w_{14}w_{23})x^{*}_{2}\\ \nonumber
&-\frac{iw_{13}}{16}(w_{34}w_{12}-w_{12}w_{34})(w_{23}w_{14}-w_{14}w_{23})x^{*}_{3}-\frac{iw_{14}}{16}(w_{34}w_{12}-w_{12}w_{34})(w_{13}w_{24}-w_{24}w_{13})x^{*}_{4}.
\end{align}
\end{rem}
In the next remark we recall that there is a correspondence between singular vectors and morphisms of $\g-$modules.
\begin{rem}
\label{esistenzamappe}
	We point out that, given $M(n_{1},n_{2},n_{3},n_{0})$ and $M(\widetilde{n}_{1},\widetilde{n}_{2},\widetilde{n}_{3},\widetilde{n}_{0})$ finite Verma modules over $\g$, we can construct a non trivial morphism of $\g-$modules from the former to the latter if and only if there exists a highest weight singular vector $\vec{m}$ in $M(\widetilde{n}_{1},\widetilde{n}_{2},\widetilde{n}_{3},\widetilde{n}_{0})$ of highest weight $\mu=(n_{1},n_{2},n_{3},n_{0})$. The map is uniquely determined by:
	\begin{align*}
	\nabla: \, M(n_{1},n_{2},n_{3},n_{0})&\longrightarrow M(\widetilde{n}_{1},\widetilde{n}_{2},\widetilde{n}_{3},\widetilde{n}_{0})\\
	v_{\mu} &\longmapsto \vec{m},
	\end{align*}
	where $v_{\mu}$ is a highest weight vector of $F(n_{1},n_{2},n_{3},n_{0})$: indeed $M(n_{1},n_{2},n_{3},n_{0})$ is generated as $\g-$module by $v_{\mu}$.
	If $\vec{m}$ is a singular vector of degree $d$, we say that $\nabla$ is a morphism of degree $d$.
\end{rem} 
Using Remark \ref{esistenzamappe}, in \cite{ck6} the authors obtained the sequences of morphisms between finite Verma modules in Figure \ref{figurack6}. In Figure \ref{figurack6} the modules $M_{A}$ are represented in the first quadrant, the modules $M_{B}$ in the second quadrant, the modules $M_{C}$ in the third quadrant.
\section{The morphisms}
\begin{figure}
     \begin{center}
             \begin{tikzpicture}
             \draw[fill=white]{(0,0) circle(3pt)};
             \draw[fill=white]{(1,0) circle(3pt)};
             \draw[fill=none]{(2,0) circle(3pt)};
             \draw[fill=none]{(3,0) circle(3pt)};
             \draw[fill=none]{(4,0) circle(3pt)};
             \draw[fill=none]{(5,0) circle(3pt)};
             \draw[fill=none]{(6,0) circle(3pt)};
             \draw[fill=none]{(6,1) circle(3pt)};
             \draw[fill=none]{(6,2) circle(3pt)};
             \draw[fill=none]{(6,3) circle(3pt)};
             \draw[fill=none]{(6,4) circle(3pt)};
             \draw[fill=none]{(6,5) circle(3pt)};
             \draw[fill=none]{(-6,-6) circle(3pt)};
             \draw[fill=none]{(-1,-6) circle(3pt)};
             \draw[fill=none]{(-2,-6) circle(3pt)};
             \draw[fill=none]{(-3,-6) circle(3pt)};
             \draw[fill=none]{(-4,-6) circle(3pt)};
             \draw[fill=none]{(-5,-6) circle(3pt)};
             \draw[fill=none]{(0,1) circle(3pt)};
             \draw[fill=none]{(1,1) circle(3pt)};
             \draw[fill=none]{(2,1) circle(3pt)};
             \draw[fill=none]{(3,1) circle(3pt)};
             \draw[fill=none]{(4,1) circle(3pt)};
             \draw[fill=none]{(5,1) circle(3pt)};
             \draw[fill=none]{(0,2) circle(3pt)};
             \draw[fill=none]{(1,2) circle(3pt)};
             \draw[fill=none]{(2,2) circle(3pt)};
             \draw[fill=none]{(3,2) circle(3pt)};
             \draw[fill=none]{(4,2) circle(3pt)};
             \draw[fill=none]{(5,2) circle(3pt)};
             \draw[fill=none]{(0,3) circle(3pt)};
             \draw[fill=none]{(1,3) circle(3pt)};
             \draw[fill=none]{(2,3) circle(3pt)};
             \draw[fill=none]{(3,3) circle(3pt)};
             \draw[fill=none]{(4,3) circle(3pt)};
             \draw[fill=none]{(5,3) circle(3pt)};
             \draw[fill=none]{(0,4) circle(3pt)};
             \draw[fill=none]{(1,4) circle(3pt)};
             \draw[fill=none]{(2,4) circle(3pt)};
             \draw[fill=none]{(3,4) circle(3pt)};
             \draw[fill=none]{(4,4) circle(3pt)};
             \draw[fill=none]{(5,4) circle(3pt)};
             \draw[fill=none]{(0,5) circle(3pt)};
             \draw[fill=none]{(1,5) circle(3pt)};
             \draw[fill=none]{(2,5) circle(3pt)};
             \draw[fill=none]{(3,5) circle(3pt)};
             \draw[fill=none]{(4,5) circle(3pt)};
             \draw[fill=none]{(5,5) circle(3pt)};

             \draw[fill=none]{(0.5,-6.5) circle(3pt)};
             \draw[fill=none]{(1.5,-6.5) circle(3pt)};
             \draw[fill=none]{(2.5,-6.5) circle(3pt)};
             \draw[fill=none]{(3.5,-6.5) circle(3pt)};
             \draw[fill=none]{(4.5,-6.5) circle(3pt)};
             \draw[fill=none]{(5.5,-6.5) circle(3pt)};
             \draw[fill=none]{(0.5,-1.5) circle(3pt)};
             \draw[fill=none]{(1.5,-1.5) circle(3pt)};
             \draw[fill=none]{(2.5,-1.5) circle(3pt)};
             \draw[fill=none]{(3.5,-1.5) circle(3pt)};
             \draw[fill=none]{(4.5,-1.5) circle(3pt)};
             \draw[fill=none]{(5.5,-1.5) circle(3pt)};
             \draw[fill=none]{(0.5,-2.5) circle(3pt)};
             \draw[fill=none]{(1.5,-2.5) circle(3pt)};
             \draw[fill=none]{(2.5,-2.5) circle(3pt)};
             \draw[fill=none]{(3.5,-2.5) circle(3pt)};
             \draw[fill=none]{(4.5,-2.5) circle(3pt)};
             \draw[fill=none]{(5.5,-2.5) circle(3pt)};
             \draw[fill=none]{(0.5,-3.5) circle(3pt)};
             \draw[fill=none]{(1.5,-3.5) circle(3pt)};
             \draw[fill=none]{(2.5,-3.5) circle(3pt)};
             \draw[fill=none]{(3.5,-3.5) circle(3pt)};
             \draw[fill=none]{(4.5,-3.5) circle(3pt)};
             \draw[fill=none]{(5.5,-3.5) circle(3pt)};
             \draw[fill=none]{(0.5,-4.5) circle(3pt)};
             \draw[fill=none]{(1.5,-4.5) circle(3pt)};
             \draw[fill=none]{(2.5,-4.5) circle(3pt)};
             \draw[fill=none]{(3.5,-4.5) circle(3pt)};
             \draw[fill=none]{(4.5,-4.5) circle(3pt)};
             \draw[fill=none]{(5.5,-4.5) circle(3pt)};
             \draw[fill=none]{(0.5,-5.5) circle(3pt)};
             \draw[fill=none]{(1.5,-5.5) circle(3pt)};
             \draw[fill=none]{(2.5,-5.5) circle(3pt)};
             \draw[fill=none]{(3.5,-5.5) circle(3pt)};
             \draw[fill=none]{(4.5,-5.5) circle(3pt)};
             \draw[fill=none]{(5.5,-5.5) circle(3pt)};
             \draw[fill=none]{(-1,-7) circle(3pt)};
             \draw[fill=none]{(-2,-7) circle(3pt)};
             \draw[fill=none]{(-3,-7) circle(3pt)};
             \draw[fill=none]{(-4,-7) circle(3pt)};
             \draw[fill=none]{(-5,-7) circle(3pt)};
             \draw[fill=none]{(-6,-7) circle(3pt)};
             \draw[fill=none]{(-6,-1) circle(3pt)};
             \draw[fill=none]{(-1,-1) circle(3pt)};
             \draw[fill=none]{(-2,-1) circle(3pt)};
             \draw[fill=none]{(-3,-1) circle(3pt)};
             \draw[fill=none]{(-4,-1) circle(3pt)};
             \draw[fill=none]{(-5,-1) circle(3pt)};
             \draw[fill=none]{(-6,-2) circle(3pt)};
             \draw[fill=none]{(-1,-2) circle(3pt)};
             \draw[fill=none]{(-2,-2) circle(3pt)};
             \draw[fill=none]{(-3,-2) circle(3pt)};
             \draw[fill=none]{(-4,-2) circle(3pt)};
             \draw[fill=none]{(-5,-2) circle(3pt)};
             \draw[fill=none]{(-6,-3) circle(3pt)};
             \draw[fill=none]{(-1,-3) circle(3pt)};
             \draw[fill=none]{(-2,-3) circle(3pt)};
             \draw[fill=none]{(-3,-3) circle(3pt)};
             \draw[fill=none]{(-4,-3) circle(3pt)};
             \draw[fill=none]{(-5,-3) circle(3pt)};
             \draw[fill=none]{(-6,-4) circle(3pt)};
             \draw[fill=none]{(-1,-4) circle(3pt)};
             \draw[fill=none]{(-2,-4) circle(3pt)};
             \draw[fill=none]{(-3,-4) circle(3pt)};
             \draw[fill=none]{(-4,-4) circle(3pt)};
             \draw[fill=none]{(-5,-4) circle(3pt)};
             \draw[fill=none]{(-6,-5) circle(3pt)};
             \draw[fill=none]{(-1,-5) circle(3pt)};
             \draw[fill=none]{(-2,-5) circle(3pt)};
             \draw[fill=none]{(-3,-5) circle(3pt)};
             \draw[fill=none]{(-4,-5) circle(3pt)};
             \draw[fill=none]{(-5,-5) circle(3pt)};
             \draw[-latex, black](-6.1,-1)--(-6.5,-1);
             \draw[black](0.1,0)--(0.9,0);
             \draw[black](1.1,0)--(1.9,0);
             \draw[black](2.1,0)--(2.9,0);
             \draw[black](3.1,0)--(3.9,0);
             \draw[black](4.1,0)--(4.9,0);
             \draw[black](5.1,0)--(5.9,0);
             \draw[black](6.1,0)--(6.5,0);
             \draw[black](0.6,-1.5)--(1.4,-1.5);
             \draw[black](1.6,-1.5)--(2.4,-1.5);
             \draw[black](2.6,-1.5)--(3.4,-1.5);
             \draw[black](3.6,-1.5)--(4.4,-1.5);
             \draw[black](4.6,-1.5)--(5.4,-1.5);
             \draw[black](5.6,-1.5)--(6.4,-1.5);

             \draw[black](0.5,-1.6)--(0.5,-2.4);
             \draw[black](0.5,-2.6)--(0.5,-3.4);
             \draw[black](0.5,-3.6)--(0.5,-4.4);
             \draw[black](0.5,-4.6)--(0.5,-5.4);
             \draw[black](0.5,-5.6)--(0.5,-6.4);
             \draw[black](0.5,-6.6)--(0.5,-7);
             \draw[black](-1,-7.1)--(-1,-7.5);
             \draw[black](-1,-1.1)--(-1,-1.9);
             \draw[black](-1,-2.1)--(-1,-2.9);
             \draw[black](-1,-3.1)--(-1,-3.9);
             \draw[black](-1,-4.1)--(-1,-4.9);
             \draw[black](-1,-5.1)--(-1,-5.9);
             \draw[black](-1,-6.1)--(-1,-6.9);

             \draw[-latex, black](0,0.9)--(0,0.1);
             \draw[-latex, black](1,0.9)--(1,0.1);
             \draw[-latex, black](2,0.9)--(2,0.1);
             \draw[-latex, black](3,0.9)--(3,0.1);
             \draw[-latex, black](4,0.9)--(4,0.1);
             \draw[-latex, black](5,0.9)--(5,0.1);
             \draw[-latex, black](6,0.9)--(6,0.1);
             \draw[-latex, black](0,1.9)--(0,1.1);
             \draw[-latex, black](1,1.9)--(1,1.1);
             \draw[-latex, black](2,1.9)--(2,1.1);
             \draw[-latex, black](3,1.9)--(3,1.1);
             \draw[-latex, black](4,1.9)--(4,1.1);
             \draw[-latex, black](5,1.9)--(5,1.1);
             \draw[-latex, black](6,1.9)--(6,1.1);
             \draw[-latex, black](0,2.9)--(0,2.1);
             \draw[-latex, black](1,2.9)--(1,2.1);
             \draw[-latex, black](2,2.9)--(2,2.1);
             \draw[-latex, black](3,2.9)--(3,2.1);
             \draw[-latex, black](4,2.9)--(4,2.1);
             \draw[-latex, black](5,2.9)--(5,2.1);
             \draw[-latex, black](6,2.9)--(6,2.1);
             \draw[-latex, black](0,3.9)--(0,3.1);
             \draw[-latex, black](1,3.9)--(1,3.1);
             \draw[-latex, black](2,3.9)--(2,3.1);
             \draw[-latex, black](3,3.9)--(3,3.1);
             \draw[-latex, black](4,3.9)--(4,3.1);
             \draw[-latex, black](5,3.9)--(5,3.1);
             \draw[-latex, black](6,3.9)--(6,3.1);
             \draw[-latex, black](0,4.9)--(0,4.1);
             \draw[-latex, black](1,4.9)--(1,4.1);
             \draw[-latex, black](2,4.9)--(2,4.1);
             \draw[-latex, black](3,4.9)--(3,4.1);
             \draw[-latex, black](4,4.9)--(4,4.1);
             \draw[-latex, black](5,4.9)--(5,4.1);
             \draw[-latex, black](6,4.9)--(6,4.1);
             \draw[-latex, black](0,5.5)--(0,5.1);
             \draw[-latex, black](1,5.5)--(1,5.1);
             \draw[-latex, black](2,5.5)--(2,5.1);
             \draw[-latex, black](3,5.5)--(3,5.1);
             \draw[-latex, black](4,5.5)--(4,5.1);
             \draw[-latex, black](5,5.5)--(5,5.1);
             \draw[-latex, black](6,5.5)--(6,5.1);
             \draw[-latex, black](-5.1,-7)--(-5.9,-7);
             \draw[-latex, black](-1.1,-7)--(-1.9,-7);
             \draw[-latex, black](-2.1,-7)--(-2.9,-7);
             \draw[-latex, black](-3.1,-7)--(-3.9,-7);
             \draw[-latex, black](-4.1,-7)--(-4.9,-7);
             \draw[-latex, black](-6.1,-7)--(-6.5,-7);
             \draw[-latex, black](-6.1,-2)--(-6.5,-2);
             \draw[-latex, black](-6.1,-3)--(-6.5,-3);
             \draw[-latex, black](-6.1,-4)--(-6.5,-4);
             \draw[-latex, black](-6.1,-5)--(-6.5,-5);
             \draw[-latex, black](-6.1,-6)--(-6.5,-6);
             \draw[-latex, black](-5.1,-1)--(-5.9,-1);
             \draw[-latex, black](-1.1,-1)--(-1.9,-1);
             \draw[-latex, black](-2.1,-1)--(-2.9,-1);
             \draw[-latex, black](-3.1,-1)--(-3.9,-1);
             \draw[-latex, black](-4.1,-1)--(-4.9,-1);
             \draw[-latex, black](-5.1,-2)--(-5.9,-2);
             \draw[-latex, black](-1.1,-2)--(-1.9,-2);
             \draw[-latex, black](-2.1,-2)--(-2.9,-2);
             \draw[-latex, black](-3.1,-2)--(-3.9,-2);
             \draw[-latex, black](-4.1,-2)--(-4.9,-2);
             \draw[-latex, black](-5.1,-3)--(-5.9,-3);
             \draw[-latex, black](-1.1,-3)--(-1.9,-3);
             \draw[-latex, black](-2.1,-3)--(-2.9,-3);
             \draw[-latex, black](-3.1,-3)--(-3.9,-3);
             \draw[-latex, black](-4.1,-3)--(-4.9,-3);
             \draw[-latex, black](-5.1,-4)--(-5.9,-4);
             \draw[-latex, black](-1.1,-4)--(-1.9,-4);
             \draw[-latex, black](-2.1,-4)--(-2.9,-4);
             \draw[-latex, black](-3.1,-4)--(-3.9,-4);
             \draw[-latex, black](-4.1,-4)--(-4.9,-4);
             \draw[-latex, black](-5.1,-5)--(-5.9,-5);
             \draw[-latex, black](-1.1,-5)--(-1.9,-5);
             \draw[-latex, black](-2.1,-5)--(-2.9,-5);
             \draw[-latex, black](-3.1,-5)--(-3.9,-5);
             \draw[-latex, black](-4.1,-5)--(-4.9,-5);
             \draw[-latex, black](-5.1,-6)--(-5.9,-6);
             \draw[-latex, black](-1.1,-6)--(-1.9,-6);
             \draw[-latex, black](-2.1,-6)--(-2.9,-6);
             \draw[-latex, black](-3.1,-6)--(-3.9,-6);
             \draw[-latex, black](-4.1,-6)--(-4.9,-6);

             \draw[-latex, black](1.4,-1.6)--(0.6,-2.4);
             \draw[-latex, black](2.4,-1.6)--(1.6,-2.4);
             \draw[-latex, black](3.4,-1.6)--(2.6,-2.4);
             \draw[-latex, black](4.4,-1.6)--(3.6,-2.4);
             \draw[-latex, black](5.4,-1.6)--(4.6,-2.4);
             \draw[-latex, black](6,-2)--(5.6,-2.4); 
             \draw[-latex, black](1.4,-2.6)--(0.6,-3.4);
					   \draw[-latex, black](2.4,-2.6)--(1.6,-3.4);
             \draw[-latex, black](3.4,-2.6)--(2.6,-3.4);
             \draw[-latex, black](4.4,-2.6)--(3.6,-3.4);
             \draw[-latex, black](5.4,-2.6)--(4.6,-3.4);
             \draw[-latex, black](6,-3)--(5.6,-3.4);
             \draw[-latex, black](1.4,-3.6)--(0.6,-4.4);
             \draw[-latex, black](2.4,-3.6)--(1.6,-4.4);
						
             \draw[-latex, black](3.4,-3.6)--(2.6,-4.4);
             \draw[-latex, black](4.4,-3.6)--(3.6,-4.4);
             \draw[-latex, black](5.4,-3.6)--(4.6,-4.4);
             \draw[-latex, black](6,-4)--(5.6,-4.4);
             \draw[-latex, black](1.4,-4.6)--(0.6,-5.4);
             \draw[-latex, black](2.4,-4.6)--(1.6,-5.4);
             \draw[-latex, black](3.4,-4.6)--(2.6,-5.4);
             \draw[-latex, black](4.4,-4.6)--(3.6,-5.4);
             \draw[-latex, black](5.4,-4.6)--(4.6,-5.4);
             \draw[-latex, black](6,-5)--(5.6,-5.4);
             \draw[-latex, black](1.4,-5.6)--(0.6,-6.4);
             \draw[-latex, black](2.4,-5.6)--(1.6,-6.4);
             \draw[-latex, black](3.4,-5.6)--(2.6,-6.4);
             \draw[-latex, black](4.4,-5.6)--(3.6,-6.4);
						
             \draw[-latex, black](5.4,-5.6)--(4.6,-6.4);
             \draw[-latex, black](6,-6)--(5.6,-6.4);
             \draw[-latex, black](1.4,-6.6)--(1.1,-6.9);
             \draw[-latex, black](2.4,-6.6)--(2.1,-6.9);
             \draw[-latex, black](3.4,-6.6)--(3.1,-6.9);
             \draw[-latex, black](4.4,-6.6)--(4.1,-6.9);
             \draw[-latex, black](5.4,-6.6)--(5.1,-6.9);

             \draw[-latex, black](0.9,-0.1)--(-0.9,-1.9);
             \draw[-latex, black](1.9,-0.1)--(0.6,-1.4);
						\draw[-latex, black](0.4,-1.6)--(-0.9,-2.9);
             \draw[-latex, black](2.9,-0.1)--(1.6,-1.4);
             \draw[-latex, black](3.9,-0.1)--(2.6,-1.4);
             \draw[-latex, black](4.9,-0.1)--(3.6,-1.4);
             \draw[-latex, black](5.9,-0.1)--(4.6,-1.4);
             \draw[-latex, black](6.5,-0.5)--(5.6,-1.4);
             \draw[-latex, black](0.4,-2.6)--(-0.9,-3.9);
             \draw[-latex, black](0.4,-3.6)--(-0.9,-4.9);
             \draw[-latex, black](0.4,-4.6)--(-0.9,-5.9);
             \draw[-latex, black](0.4,-5.6)--(-0.9,-6.9);
             \draw[-latex, black](0.4,-6.6)--(-0.3,-7.3);

%
             \node at (-7,-1) {$n_{2}$};
             \node at (7,0) {$n_{1}$};
             \node at (7,-1.5) {$n_{1}$};
             \node at (0,5.8) {$n_{2}$};
             \node at (-1,-7.8) {$n_{3}$};
             \node at (0.5,-7.8) {$n_{3}$};
             \node at (3.5,6){$M(n_{1},n_{2},0,-n_{2}-\frac{n_{1}}{2})$};
						\node at (6,6){\textbf{A}};
             \node at (-4,-8){$M(0,n_{2},n_{3},n_{2}+\frac{n_{3}}{2}+4)$};
						\node at (-7,-8){\textbf{C}};
             \node at (3.5,-7.8){$M(n_{1},0,n_{3},\frac{n_{3}}{2}-\frac{n_{1}}{2}+2)$};
						\node at (6.6,-7.8){\textbf{B}};
             \end{tikzpicture}
     \end{center}
		\caption{} 
	\label{figurack6}
\end{figure}
The aim of this section is to find an explicit expression for the maps in the first quadrant and between the first and second quadrant of Figure \ref{figurack6}.
We call $M_{X}=\oplus_{m,n}M^{m,n}_{X}$ and $V_{X}=\oplus_{m,n}V^{m,n}_{X}$, for $X=A,B,C$.\\
We follow the notation in \cite{kacrudakovE36} and define, for every $u \in U(\g_{<0}) $ and $\phi \in \Hom(V_{X},V_{Y})$, the map $u \otimes \phi :  M_{X}\longrightarrow M_{Y}$ by:
\begin{align}
\label{uuprimo}
(u \otimes \phi)(u' \otimes v)= u' \,u \otimes \phi(v),
\end{align}
for every $ u' \otimes v \in U(\g_{<0}) \otimes V_{X}$. From this definition it is clear that the map $u \otimes \phi$ commutes with the action of $\g_{<0}$.\\
We now recall some results on maps constructed as in \eqref{uuprimo} from \cite{bagnoli2}.
\begin{lem}[\cite{bagnoli2}]
\label{commutacong0basiduali}
Let $u \otimes \phi$ be a map as in \eqref{uuprimo}. Let us suppose that $u \otimes \phi=\sum_{i}u_{i} \otimes \phi_{i}$ where $\left\{u_{i}\right\}_{i}$ and $\left\{\phi_{i}\right\}_{i}$ are bases of dual $\g_{0}-$modules and $u_{i}$ is the dual of $\phi_{i}$ for all $i$. Then $u \otimes \phi$ commutes with $\g_{0}$. 
\end{lem}
\begin{lem}[\cite{bagnoli2}]
\label{g0commuta}
Let us consider a map $u \otimes \phi \in U(\g_{<0}) \otimes \Hom(V_{X},V_{Y}) $. In order to show that $u \otimes \phi$ commutes with $\g_{0}$, it is sufficient to show that $w u \otimes \phi(v)=u \otimes \phi(w.v)$ for all $v \in V_{X}$, $w \in \g_{0}$.
\end{lem}
\begin{lem}[\cite{bagnoli2}]
\label{nablag0g+}
Let $\Phi:M_{X}\rightarrow M_{Y}$ be a linear map. Let us suppose that $\Phi$ commutes with $\g_{\leq 0}$ and that $\Phi(v)$ is a singular vector for every $v$ highest weight vector in $V_{X}^{m,n}$ and for all $m,n \in \Z$. Then $\Phi$ is a morphism of $\g-$modules.
\end{lem}
We define, using \eqref{uuprimo}, the following map between the modules $M_{A}$ in the first quadrant:
\begin{gather}
\label{nablaA}
\nabla_{A}: M\left(n_{1},n_{2},0,-n_{2}-\frac{n_{1}}{2}\right) \longrightarrow M\left(n_{1},n_{2}-1,0,-(n_{2}-1)-\frac{n_{1}}{2}\right)\\ \nonumber
\nabla_{A}=w_{12} \otimes \partial_{12}+w_{13} \otimes \partial_{13}+w_{23} \otimes \partial_{23}+w_{14} \otimes \partial_{14}+w_{24} \otimes \partial_{24}+w_{34} \otimes \partial_{34},
\end{gather}
where $\partial_{ij} $ denotes the derivative with respect to $x_{ij}$. We assume that $\partial_{ij}=-\partial_{ji}  $ for all $i,j$.
\begin{rem}
\label{ck6nablavettsing}
The map $\nabla_{A}$ is constructed so that it sends the highest weight vector $x_{1}^{n_{1}}x_{12}^{n_{2}} $ of $M(n_{1},n_{2},0,-n_{2}-\frac{n_{1}}{2})$ to $\vec{m}=w_{12} \otimes n_{2} x_{12}^{n_{2}-1} x_{1}^{n_{1}}$ that is the highest weight singular vector $\vec{m}_{1a}$ of $M\left(n_{1},n_{2}-1,0,-(n_{2}-1)-\frac{n_{1}}{2}\right)$ (see Remark \ref{teoremavettorisingolarick6ultimo}).
\end{rem}
\begin{lem}
\label{nablaAcommuta}
The maps $\nabla_{A}$, defined in \eqref{nablaA}, are the explicit expression of the morphisms of $\g-$modules of degree 1 represented in quadrant \textbf{A} and $\nabla_{A}^{2}=0$.
\end{lem}
\begin{proof}
The map $\nabla_{A}$ commutes with $\g_{<0}$ by \eqref{uuprimo}.
By Remark \ref{ck6nablavettsing} and Lemmas \ref{commutacong0basiduali}, \ref{nablag0g+} it follows that $\nabla_{A}$ is a morphism of $\g-$modules. The property $\nabla_{A}^{2}=0$ follows from the fact that $\nabla_{A}$ is a map between Verma modules that contain only highest weight singular vectors of degree 1 and there are no singular vectors of degree 2, by Remark \ref{teoremavettorisingolarick6ultimo}.
\end{proof}
We define, using \eqref{uuprimo}, the following map from the first to the second quadrant of Figure \ref{figurack6}:
\begin{align}
\label{explicitnabla3}
\centering
\nabla_{3}: M\left(n_{1},0,0,-\frac{n_1}{2}\right)\rightarrow M\left(n_{1}-2,0,0,-\frac{(n_1-2)}{2}+2\right) 
\end{align}
\begin{align*}
\nabla_{3}&=w_{12} w_{13} w_{14} \otimes \partial_{1}^{2}-w_{12}w_{24}w_{23}\otimes \partial_{2}^{2}-w_{13}w_{34}w_{23}\otimes \partial_{3}^{2}-w_{34}w_{24}w_{14}\otimes \partial_{4}^{2}\\
&-\frac{w_{12}}{2}(w_{14}w_{23}- w_{23}w_{14}+ w_{24}w_{13}- w_{13}w_{24})\otimes \partial_{1}\partial_{2}-\frac{w_{13}}{2}(w_{14}w_{23}- w_{23}w_{14}-w_{34}w_{12}+ w_{12}w_{34})\otimes \partial_{1}\partial_{3}\\ 
&-\frac{w_{14}}{2}(w_{24}w_{13}- w_{13}w_{24}+ w_{12}w_{34}- w_{34}w_{12})\otimes \partial_{1}\partial_{4}+\frac{w_{23}}{2}(-w_{13}w_{24}+ w_{24}w_{13}- w_{12}w_{34}+ w_{34}w_{12})\otimes \partial_{2}\partial_{3}\\
&+\frac{w_{24}}{2}(w_{34}w_{12}- w_{12}w_{34}+ w_{14}w_{23}- w_{23}w_{14})\otimes \partial_{2}\partial_{4}-\frac{w_{34}}{2}(w_{24}w_{13}- w_{13}w_{24}-w_{14}w_{23}+ w_{23}w_{14})\otimes \partial_{3}\partial_{4},
\end{align*}
where $\partial_{i} $ denotes the derivative with respect to $x_{i}$.
\begin{rem}
\label{ck6nablavettsing3}
The map $\nabla_{3}$ is constructed so that it sends the highest weight vector $x_{1}^{n_{1}} $ of $M\left(n_{1},0,0,-\frac{n_{1}}{2} \right)$ to $\vec{m}=n_{1}(n_{1}-1)w_{12} w_{13} w_{14} \otimes x_{1}^{n_{1}-2}$ that is the highest weight singular vector $\vec{m}_{3b}$ of $M\left(n_{1}-2,0,0,-\frac{n_{1}-2}{2}+2\right)$ (see Remark \ref{teoremavettorisingolarick6ultimo}).
\end{rem}
\begin{prop}
\label{morfismonabla3}
The maps $\nabla_{3}$ are the explicit expression of the morphisms of $\g-$modules of degree 3 from quadrant \textbf{A} to quadrant \textbf{B} and $\nabla_{3}\nabla_{A}=0$.
\end{prop}
\begin{proof}
 First we show that the map $\nabla_{3}$ is a morphism of $\g-$modules. It commutes with $\g_{<0}$ due to \eqref{uuprimo}.
	Due to Lemmas \ref{g0commuta} and \ref{nablag0g+}, it is sufficient to show that $w u \otimes \nabla_{3}(v)=u \otimes \nabla_{3}(w.v)$ for every $w \in \g_{0}$, $v \in V_{A}$. This property follows by direct computations.
 The property $\nabla_{3}\nabla_{A}=0$ follows from the fact that there are no singular vectors of degree 4, by Remark \ref{teoremavettorisingolarick6ultimo}.
\end{proof}
In the following we call $\nabla_B$ (resp. $\nabla_C$) the $\g-$morphism from $M^{n_1,n_3}_{B}$ to $M^{n_1-1,n_3+1}_{B}$ (resp. from $M^{n_2,n_3}_{C}$ to $M^{n_2+1,n_3}_{C}$) that maps the highest weight vector of $V^{n_1,n_3}_{B}$ (resp. $V^{n_2,n_3}_{C}$) to the highest weight singular vector of degree 1 of $M^{n_1-1,n_3+1}_{B}$ (resp. $M^{n_2+1,n_3}_{C}$). We point out that $\nabla_{B}^{2}=\nabla_{C}^{2}=0$ since there are no highest weight singular vectors of degree 2 by Remark \ref{teoremavettorisingolarick6ultimo}. We call $\nabla_{5}$ the $\g-$morphism, from $M^{1,0}_{A}$ to $M^{0,1}_{C}$, that maps the highest weight vector $x_{1}$ of $V^{1,0}_{A}$ to the highest weight singular vector of degree 5 of $M^{0,1}_{C}$ (see Remark \ref{teoremavettorisingolarick6ultimo}). We point out that $\nabla_{5}\nabla_{A}=\nabla_{C}\nabla_{5}=0$ since there are no highest weight singular vectors of degree 6 by Remark \ref{teoremavettorisingolarick6ultimo}. The morphism $\nabla_{5}$ is represented from the first to the third quadrant in Figure \ref{figurack6}. We call $\widetilde{\nabla}_{3}: M^{0,n_3}_{B} \rightarrow M^{0,n_3+2}_{C}$ the $\g-$morphisms that map the highest weight vector of $V^{0,n_3}_{B}$ to the unique highest weight singular vector of degree 3 in $M^{0,n_3+2}_{C}$ (see Remark \ref{teoremavettorisingolarick6ultimo}). These are the morphisms of degree 3 represented in the Figure \ref{figurack6} from the quadrant \textbf{B} to the quadrant \textbf{C}. We point out that $\widetilde{\nabla}_{3}\nabla_{B}=\nabla_{C}\widetilde{\nabla}_{3}=0$ since there are no highest weight singular vectors of degree 4 by Remark \ref{teoremavettorisingolarick6ultimo}.\\
The following sections are dedicated to the computation of the homology of complexes in the first an third quadrants of Figure \ref{figurack6}. We will use techniques of spectral sequences, that we briefly recall for the reader's convenience in the next section.
\section{Preliminaries on spectral sequences}
In this section we recall some notions about spectral sequences; for further details see \cite[Appendix]{kacrudakovE36}, \cite[Chapter XI]{maclane} and \cite{bagnoli2}. We follow the notation used in \cite{kacrudakovE36}.\\
Let $A$ be a module with a filtration:
\begin{align}
\label{filtration}
...\subset F_{p-1}A \subset F_{p}A \subset F_{p+1}A\subset...  ,
\end{align}
where $p \in \Z$. A filtration is called \textit{convergent above} if $A=\cup_{p}\, F_{p}A$. 
We suppose that $A$ is endowed with a differential $d:A\longrightarrow A$ such that:
\begin{align}
\label{differential}
 d^{2}=0 \quad \text{and} \quad d( F_{p}A)\subset F_{p-s+1}A,
\end{align}
 for fixed $s$ and all $p$ in $\Z$. The classical case studied in \cite[Chapter XI, Section 3]{maclane} corresponds to $s=1$. We will need the case $s=0$. The filtration \eqref{filtration} induces a filtration on the module $H(A)$ of the homology spaces of $A$: indeed, for every $p \in \Z$, $F_{p}H(A) $ is defined as the image of $H(F_{p}A) $ under the injection $F_{p}A \longrightarrow A$.
\begin{defi}
\label{famigliacondiff}
Let $E=\left\{E_{p}\right\}_{p\in \Z}$ be a family of modules. A differential $d:E\longrightarrow E$ of degree $-r\in \Z$ is a family of homorphisms $\left\{d_{p}: E_{p}\longrightarrow E_{p-r}\right\}_{p \in \Z}$ such that $d_{p}\circ d_{p+r}=0$ for all $p \in \Z$. We denote by $H(E)=H(E,d)$ the homology of $E$ under the differential $d$ that is the family $\left\{H_{p}(E,d)\right\}_{p\in \Z}$, where:
\begin{align*}
H_{p}(E,d)=\frac{\Ker(d_{p}: E_{p}\longrightarrow E_{p-r})}{\Ima (d_{p+r}: E_{p+r}\longrightarrow E_{p})}.
\end{align*}
\end{defi}
\begin{defi}[Spectral sequence]
\label{defispec}
A \textit{spectral sequence} $E=\left\{(E^{r},d^{(r)})\right\}_{r\in \Z}$ is a sequence of families of modules with differential $(E^{r},d^{(r)})$ as in Definition \ref{famigliacondiff}, such that, for all $r$, $d^{(r)}$ has degree $-r$ and:
\begin{align*}
H(E^{r},d^{(r)}) \cong E^{r+1}.
\end{align*}
\end{defi}
\begin{prop}
\label{spectral1}
Let $A$ be a module with a filtration as in \eqref{filtration} and differential as in \eqref{differential}. Therefore it is uniquely determined a spectral sequence, as in Definition \ref{defispec}, $E=\left\{(E^{r},d^{(r)})\right\}_{r \in \Z}$ such that:
\begin{gather}
 H(E^{r},d^{(r)}) \cong E^{r+1}, \label{spectral1prima} \\
E^{r}_{p} \cong F_{p}A / F_{p-1}A \quad \text{for}  \quad r\leq s-1, \label{spectral1seconda} \\
d^{(r)}=0 \quad \text{for}  \quad r <s-1, \label{spectral1terza}\\
d^{(s-1)}=\Gr d, \label{spectral1quarta}\\
E^{s}_{p} \cong H(F_{p}A / F_{p-1}A) . \label{grado0graded} 
\end{gather}
\end{prop}
\begin{proof}
For the proof see \cite[Appendix]{kacrudakovE36}.
\end{proof}
We point out that \eqref{grado0graded} states that $E^{s}$ is isomorphic to the homology of the module $\Gr A$ with respect to the differential induced by $d$. 
\begin{rem}
Let $\left\{(E^{r},d^{(r)})\right\}_{r \in \Z}$ be a spectral sequence as in Definition \ref{defispec}. We know that $E^{1}_{p}\cong H_{p}(E^{0},d^{(0)})$. We denote $E^{1}_{p}\cong C^{0}_{p}/B^{0}_{p}$, where $C^{0}_{p}=\Ker d^{(0)}_{p}$ and $B^{0}_{p}=\Ima d^{(0)}_{p+r}$. Analogously $E^{2}_{p}\cong H_{p}(E^{1},d^{(1)})$ and $E^{2}_{p}\cong C^{1}_{p}/B^{1}_{p}$, where $C^{1}_{p}/B^{0}_{p}=\Ker d^{(1)}_{p}$, $B^{1}_{p}/B^{0}_{p}=\Ima d^{(1)}_{p+r}$ and $B^{1}_{p} \subset C^{1}_{p}$. Thus, by iteration we obtain the following inclusions:
\begin{align*}
B^{0}_{p}\subset B^{1}_{p} \subset B^{2}_{p} \subset... \subset C^{2}_{p} \subset C^{1}_{p}\subset C^{0}_{p}.
\end{align*}
\end{rem}
\begin{defi}
Let $A$ be a module with a filtration as in \eqref{filtration} and differential as in \eqref{differential}. Let $\left\{(E^{r},d^{(r)})\right\}_{r \in \Z}$ be the spectral sequence determined by Proposition \ref{spectral1}.
We define $E^{\infty}_{p}$ as:
\begin{align*}
E^{\infty}_{p}=\frac{\bigcap_{r}C^{r}_{p}}{\bigcup _{r}B^{r}_{p}}.
\end{align*}
Let $B$ be a module with a filtration as in \eqref{filtration}. We say that the spectral sequence converges to $B$ if, for all $p$:
\begin{align*}
E^{\infty}_{p}\cong F_{p}B/F_{p-1}B.
\end{align*}
\end{defi}
\begin{prop}
\label{spectral2}
Let $A$ be a module with a filtration as in \eqref{filtration} and differential as in \eqref{differential}. Let us suppose that the filtration is convergent above and, for some $N$, $F_{-N}A=0$. Then the spectral sequence converges to the homology of $A$, that is:
$$E^{\infty}_{p} \cong F_{p} H(A) / F_{p-1}H(A).$$
\end{prop}
\begin{proof}
For the proof see \cite[Appendix]{kacrudakovE36}.
\end{proof}
\begin{rem}
\label{spectralgraduato}
Let $A$ be a module with a filtration as in \eqref{filtration} and differential as in \eqref{differential}. We moreover suppose that $A=\oplus_{n\in \Z}A_{n}$ is a $\Z-$graded module and $d:A_{n}\longrightarrow A_{n-1}$ for all $n\in \Z$. Therefore the filtration \eqref{filtration} induces a filtration on each $A_{n}$. The family $\left\{F_{p}A_{n}\right\}_{p,n \in \Z}$ is indexed by $(p,n)$. It is customary to write the indices as $(p,q)$, where $p$ is the degree of the filtration and $q=n-p$ is the complementary degree. The filtration is called \textit{bounded below} if, for all $n\in \Z$, there exists a $s=s(n)$ such that $F_{s}A_{n}=0$.\\
In this case the spectral sequence $E=\left\{(E^{r},d^{(r)})\right\}_{r \in \Z}$, determined as in Proposition \ref{spectral1}, is a family of modules $E^{r}=\left\{E^{r}_{p,q}\right\}_{p ,q\in \Z}$ indexed by $(p,q)$, where $E^{r}_{p}=\sum_{q \in \Z}E^{r}_{p,q}$, with the differential $d^{(r)}=\left\{d^{(r)}_{p,q}: E_{p,q}\longrightarrow E_{p-r,q+r-1}\right\}_{p ,q\in \Z}$ of bidegree $(-r,r-1)$ such that $d_{p,q}\circ d_{p+r,q-r+1}=0$ for all $p,q \in \Z$.
 Equations \eqref{spectral1prima}, \eqref{spectral1seconda} ,\eqref{spectral1terza}, \eqref{spectral1quarta} and \eqref{grado0graded} can be written so that the role of $q$ is explicit. For instance, Equation \eqref{spectral1prima} can be written as:
\begin{align*}
H_{p,q}(E^{r},d^{(r)})=\frac{\Ker(d^{(r)}_{p,q}: E^{r}_{p,q}\longrightarrow E^{r}_{p-r,q+r-1})}{\Ima (d^{(r)}_{p+r,q-r+1}: E^{r}_{p+r,q-r+1}\longrightarrow E^{r}_{p,q})}\cong E^{r+1}_{p,q}.
\end{align*}
for all $p,q \in \Z$.
Equation \eqref{grado0graded} can be written as $E^{s}_{p,q} \cong H(F_{p}A_{p+q} / F_{p-1}A_{p+q})$ for all $p,q \in \Z$.
\end{rem}
We now recall some results on spectral sequences of bicomplexes; for further details see \cite{kacrudakovE36} and \cite[Chapter XI, Section 6]{maclane}.
\begin{defi}[Bicomplex]
\label{definizionebicomplesso}
A bicomplex $K$ is a family $\left\{K_{p,q}\right\}_{p,q \in \Z}$ of modules endowed with two families of differentials, defined for all integers $p,q$, $d'$ and $d''$ such that
\begin{align*}
d':K_{p,q}\longrightarrow K_{p-1,q} , \quad d'':K_{p,q}\longrightarrow K_{p,q-1}
\end{align*}
and ${d'}^{2}={d''}^{2}=d'd''+d''d'=0$.
\end{defi}
We can also  think $K$ as a $\Z-$bigraded module where $K=\sum_{p,q \in \Z}K_{p,q}$.
\begin{defi}[Second homology]
Let $K$ be a bicomplex. The \textit{second homology} of $K$ is the homology computed with respect to $d''$, i.e.:
\begin{align*}
H''_{p,q}(K)=\frac{\Ker(d'':K_{p,q}\longrightarrow K_{p,q-1})}{d''(K_{p,q+1})}.
\end{align*}
The second homology of $K$ is a bigraded complex with differential $d':H''_{p,q}(K)\longrightarrow H''_{p-1,q}(K)$ induced by the original $d'$. Its homology is defined as:
\begin{align*}
H'_{p}H''_{q}(K)=\frac{\Ker(d':H''_{p,q}(K)\longrightarrow H''_{p-1,q})}{d'(H''_{p+1,q}(K))},
\end{align*}
and it is a bigraded module.
\end{defi}
\begin{defi}[First homology]
Let $K$ be a bicomplex. The \textit{first homology} of $K$ is the homology computed with respect to $d'$, i.e.:
\begin{align*}
H'_{p,q}(K)=\frac{\Ker(d':K_{p,q}\longrightarrow K_{p-1,q})}{d'(K_{p+1,q})}.
\end{align*}
The first homology of $K$ is a bigraded complex with differential $d'':H'_{p,q}(K)\longrightarrow H'_{p,q-1}(K)$ induced by the original $d''$. Its homology is defined as:
\begin{align*}
H''_{q}H'_{p}(K)=\frac{\Ker(d'':H'_{p,q}(K)\longrightarrow H'_{p,q-1})}{d''(H'_{p,q+1}(K))},
\end{align*}
and it is a bigraded module. 
\end{defi}
\begin{defi}[Total complex]
A bicomplex $K$ defines a single complex $T=Tot(K)$:
\begin{align*}
T_{n}=\sum_{p+q=n}K_{p,q}, \quad d=d'+d'':T_{n}\longrightarrow T_{n-1}.
\end{align*}
From the properties of $d'$ and $d''$, it follows that $d^{2}=0$. 
\end{defi}
The first filtration $F'$ of $T=Tot(K)$ is defined as:
\begin{align*}
(F'_{p}T)_{n}=\sum_{h\leq p}K_{h,n-h}.
\end{align*}
The associated spectral sequence $E'$ is called \textit{first spectral sequence}. Analogously we can define the second filtration and the \textit{second spectral sequence}.
\begin{prop}[\cite{maclane},\cite{kacrudakovE36}]
\label{spectralbicomplex}
Let $(K,d',d'')$ be a bicomplex with total differential $d$. The first spectral sequence $E'=\left\{(E'^{r},d^{(r)})\right\}$, $E'^{r}=\sum_{p,q}E'^{r}_{p,q}$ has the property:
\begin{align*}
(E'^{0},d^{(0)})\cong (K,d'') , \quad (E'^{1},d^{(1)})\cong (H(K,d''),d'), \quad E'^{2}_{p,q}\cong H'_{p}H''_{q}(K).
\end{align*}
The second spectral sequence $E''=\left\{(E''^{r},\delta^{(r)})\right\}$, $E''^{r}=\sum_{p,q}E''^{r}_{p,q}$ has the property:
\begin{align*}
(E''^{0},\delta^{(0)})\cong (K,d') , \quad (E''^{1},\delta^{(1)})\cong (H(K,d'),d''), \quad E''^{2}_{p,q}\cong H''_{q}H'_{p}(K).
\end{align*}
If the first filtration is bounded below and convergent above, then the first spectral sequence converges to the homology of $T$ with respect to the total differential $d$.\\
If the second filtration is bounded below and convergent above, then the second spectral sequence converges to the homology of $T$ with respect to the total differential $d$.\\
Any of the above spectral sequences converges to the homology of $T$ with respect to the total differential $d$ if for every $n$ the set
$$\left\{(p,q) \,\,| \,\, p+q=n, \,\, E'^{2}_{p,q}\neq 0\right\} $$ is finite.
\end{prop}
\begin{rem}
\label{d2esplicito}
Let $K$ be a bicomplex with total differential $d=d'+d''$, such that $d'^{2}=d''^{2}=d'd''+d''d'=0.$ as in Definition \ref{definizionebicomplesso}. By Proposition \ref{spectralbicomplex} we know that it is possible to construct a spectral sequence such that $$(E'^{0},d^{(0)})=(K,d''), \quad (E'^{1},d^{(1)})=(H(K,d''),d'), \quad E'^{2}_{p,q}\cong H'_{p}H''_{q}(K).$$ 
We want to write explicitly how $d_{p,q}^{(2)}:E'^{2}_{p,q}\rightarrow E'^{2}_{p-2,q+1}$ works. We recall that $E'^{1}_{p,q}=\frac{\Ker d''_{p,q}}{\Ima d''_{p,q+1}}$.  
\begin{align*}
 d^{(1)}_{p,q}=\overline{d'_{p,q}}:\frac{\Ker d''_{p,q}}{\Ima d''_{p,q+1}}&\rightarrow \frac{\Ker d''_{p-1,q}}{\Ima d''_{p-1,q+1}}\\
[x]&\mapsto [d'_{p,q}(x)],
\end{align*}
where the brackets $[,]$ emphasize that we are considering a class modulo $\Ima d''$ and $\overline{d'_{p,q}}$ emphasizes that it is the map induced by $d'_{p,q}$ at the quotient. Hence $E'^{2}_{p,q}=\frac{\Ker \overline{d'_{p,q}}}{\Ima\overline{d'_{p+1,q}}}$.
Thus $[x]\in \Ker \overline{d'_{p,q}}$ is equivalent to: $x \in \Ker d''_{p,q}$ and $d'_{p,q}(x)\in\Ima d''_{p-1,q+1}$. Let $y_x$ be such that $d''_{p-1,q+1}(y_x)=d'_{p,q}(x)$. We consider $z_x=d'_{p-1,q+1}(y_x) \in K_{p-2,q+1}$. We claim that $d_{p,q}^{(2)}([[x]])=[[z_x]] \in E'^{2}_{p-2,q+1}$ and it is well defined. Indeed:
\begin{enumerate}
	\item Let us show that $z_x \in \Ker d'_{p-2,q+1}$. Indeed $d'_{p-2,q+1}(z_x)=d'_{p-2,q+1}(d'_{p-1,q+1}(y_x))=0$, since $d'^{2}=0$.
	\item Let us show that $z_x \in \Ker d''_{p-2,q+1}$. Indeed
	$$d''_{p-2,q+1}(z_x)=d''_{p-2,q+1}(d'_{p-1,q+1}(y_x))=-d'_{p-1,q}(d''_{p-1,q+1}(y_x))=-d'_{p-1,q}(d'_{p,q}(x))=0.$$
	\item Let us show the independence from the choice of $y_x$. Let $y^{1}_{x}$ and $y^{2}_{x}$ such that $d''_{p-1,q+1}(y^{1}_{x})=d''_{p-1,q+1}(y^{2}_{x})=d'_{p,q}(x)$. Hence $d''_{p-1,q+1}(y^{1}_{x}-y^{2}_{x})=0$, i.e. $y^{1}_{x}-y^{2}_{x} \in \Ker d''_{p-1,q+1}$ and $[y^{1}_{x}-y^{2}_{x}] \in E'^{1}_{p-1,q+1}$ is well defined.
	Obviously $\overline{d'_{p-1,q+1}}([y^{1}_{x}-y^{2}_{x}]) \in \Ima \overline{d'_{p-1,q+1}}$. Hence
	$$ 0=[\overline{d'_{p-1,q+1}}([y^{1}_{x}-y^{2}_{x}]) ]=[[d'_{p-1,q+1}(y^{1}_{x}-y^{2}_{x})]]=[[z^{1}_{x}-z^{2}_{x}]].$$
	\item Let us show that if $x \in \Ima d''_{p,q+1} $, then $[[z_x]]=0$. Indeed let $x=d''_{p,q+1} (u)$. Let us choose $y_x=-d'_{p,q+1}(u)$. Then $z_x=d'_{p-1,q+1}(y_x)=d'_{p-1,q+1}(-d'_{p,q+1}(u))=0.$ Hence $[[z_x]]=0$.
	\item Let us show that if $[x] \in \Ima \overline{d'_{p+1,q} }$, then $[[z_x]]=0$. Indeed let $[x] = \overline{d'_{p+1,q} }([u])$. Then $x=d'_{p+1,q}(u)+d''_{p,q+1}(v)$, for a suitable $v$, and $d'_{p,q}(x)=d'_{p,q}(d''_{p,q+1}(v))=-d''_{p-1,q+1}(d'_{p,q+1}(v))$. Hence it is possible to choose $y_x=-d'_{p,q+1}(v)$ and thus $z_x=0$.
	\item The fact that $d^{(2)}_{p-2,q+1} \circ d^{(2)}_{p,q}=0$ is an obvious verification.
\end{enumerate}
\end{rem}
\section{homology}
The aim of this section is the computation of the homology for the first and third quadrants. 
Following \cite{kacrudakovE36}, let us consider the filtration on $U(\g_{<0})$ as follows: for all $i\geq 0$, $ F_{i}U(\g_{<0})$ is the subspace of $U(\g_{<0})$ spanned by elements with at most $i$ terms of $\g_{<0}$. Therefore:
\begin{align*}
\C=F_{0}U(\g_{<0}) \subset F_{1}U(\g_{<0})\subset ...\subset F_{i-1}U(\g_{<0})\subset F_{i}U(\g_{<0})\subset... \, ,
\end{align*}
where $F_{i}U(\g_{<0})=\g_{<0}F_{i-1}U(\g_{<0})+F_{i-1}U(\g_{<0})$.
We call $F_{i}M_{A}=F_{i}U(\g_{<0})\otimes V_{A}$. We have that $\nabla_{A} F_{i}M_{A} \subset F_{i+1}M_{A} $ and the filtration is bounded below. Then we can use the theory of spectral sequences; we first study $\Gr M_{A}$.
We consider the subalgebra $\g_{\bar{0}}$ of $\g$ given by the even elements, that, since the grading is consistent, is $\g_{\bar{0}}=\oplus_{i \geq -1} \g_{2i}$. On $\g_{\bar{0}}$ we consider the filtration $\g_{\bar{0}} =L_{-1}\supset L_{0}=\oplus_{i \geq 0} \g_{2i} \supset L_{1} =\oplus_{i \geq 1} \g_{2i}... \,$.
\begin{lem}
For all $j\geq-1$ and $i\geq 0$, we have:
\begin{align}
\label{filtrazioneWck6}
L_{j} F_{i}M_{A} \subset F_{i-j}M_{A}.
\end{align}
\end{lem}
\begin{proof}
The proof is analogous to Lemma 6.5 in \cite{bagnoli2}.
\end{proof}
By \eqref{filtrazioneWck6}, we know, since $\g_{\bar{0}} \cong \Gr \g_{\bar{0}}$, that the action of $\g_{\bar{0}}$ on $M_{A}$ descends on $\Gr M_{A}$.\\
 We point out that, using the Poincar\'e$-$Birkhoff$-$Witt Theorem, we have $\Gr U(\g_{<0}) \cong S(\g_{-2}) \otimes \inlinewedge(\g_{-1})$; indeed we have already noticed that in $U(\g_{<0})$, for all $i \in \left\{1,2,3,4,5,6\right\}$, $\eta^{2}_{i}=\Theta$. Therefore, as $\g_{\bar{0}}-$modules:
\begin{align*}
\Gr M_{A}=\Gr U(\g_{<0}) \otimes V_{A}\cong S(\g_{-2}) \otimes \displaywedge(\g_{-1})\otimes V_{A}.
\end{align*}
From \eqref{filtrazioneWck6}, it follows that $L_{1}$ annihilates $G_{A}:=\inlinewedge(\g_{-1}) \otimes V_{A}$. Therefore, as $\g_{\bar{0}}-$modules:
\begin{align}
\label{graduato}
\Gr M_{A} \cong S(\g_{-2}) \otimes (\displaywedge(\g_{-1}) \otimes V_{A}) \cong \Ind ^{\g_{\bar{0}}}_{L_{0}}(\displaywedge(\g_{-1}) \otimes V_{A}).
\end{align}
We observe that $\Gr M_{A}$ is a complex with the morphism induced by $\nabla_{A}$, that we still call $\nabla_{A}$.
Indeed $\nabla_{A} F_{i}M_{A} \subset F_{i+1}M_{A} $ for all $i$, therefore it is well defined the induced morphism
\begin{align*}
\nabla_{A}: \Gr_{i }M_{A}= F_{i}M_{A}/F_{i-1}M_{A}   \longrightarrow \Gr_{i+1 }M_{A}= F_{i+1}M_{A}/F_{i}M_{A},
\end{align*}
that has the same formula as $\nabla_{A}$ defined in \eqref{nablaA}, apart from the fact that the multiplication by the $w$'s must be seen as multiplication in $\Gr U(\g_{<0})$ instead of $ U(\g_{<0})$.\\
Therefore we have that $(G_{A}, \nabla_{A})$ is a subcomplex of $(\Gr M_{A}, \nabla_{A})$: indeed it is sufficient to restrict $\nabla_{A}$ to $G_{A}$; the complex $(\Gr M_{A}, \nabla_{A})$ is obtained from $(G_{A}, \nabla_{A})$ extending the coefficients to $S(\g_{-2})$.\\
We point out that also the homology spaces $H^{m,n}(G_{A})$ are annihilated by $L_{1}$. Therefore, as $\g_{\bar{0}}-$modules:
\begin{align}
\label{keyhomologyck6}
H^{m,n}(\Gr M_{A}) \cong S(\g_{-2}) \otimes H^{m,n}(G_{A}) \cong \Ind ^{\g_{\bar{0}}}_{L_{0}}(H^{m,n}(G_{A}) ).
\end{align}
From \eqref{keyhomologyck6} and Proposition \ref{spectral1}, it follows that:
\begin{prop}
\label{keyhomologyspectralck6}
If $H^{m,n}(G_{A})=0$, then $H^{m,n}(\Gr M_{A})=0$ and therefore $H^{m,n}(M_{A})=0$.
\end{prop}
We introduce the notation, for all $n_{1}\geq 0$: $V_{A^{'}}^{n_{1}}:=V_{A}^{0,n_{1}}$ and $V_{B^{'}}^{n_{1}}:=V_{B}^{0,n_{1}}$.
We will call $V_{A^{'}}=\oplus_{n_{1}} V_{A^{'}}^{n_{1}}$, $V_{B^{'}}=\oplus_{n_{1}} V_{B^{'}}^{n_{1}}$, $G_{A^{'}}=\inlinewedge(\g_{-1}) \otimes V_{A^{'}}$ and $G_{B^{'}}=\inlinewedge(\g_{-1}) \otimes V_{B^{'}}$. Let us consider the evaluation map from $V_{A}$ to $V_{A^{'}}$, which maps $x_{12},x_{13},x_{14},x_{23},x_{24},x_{34} $ to zero and is the identity on all other elements; we can compose this map with $\nabla_{3}$ and obtain a new map, that we still call $\nabla_{3}$, from $G_{A}$ to $G_{B^{'}}$. We define:
\begin{align*}
G_{A^{\circ}}=\Ker (\nabla_{3}:G_{A} \longrightarrow G_{B^{'}} ). 
\end{align*}
The map $\nabla_{A}$ is still defined on $G_{A^{\circ}}$ since $\nabla_{3}\nabla_{A}=0$.
\begin{rem}
\label{ck6circ}
From its definition, it is obvious that $G_{A^{\circ}}^{n_{1},n_{2}}=G_{A}^{n_{1},n_{2}}$ if $n_{2}>0$. Therefore:
\begin{align*}
H^{n_{1},n_{2}}(G_{A})=H^{n_{1},n_{2}}(G_{A^{\circ}}) .
\end{align*}
On the other hand $H^{n_{1},0}(G_{A})=\frac{G^{n_{1},0}_{A}}{\Ima: G^{n_{1},1}_{A}\rightarrow G^{n_{1},0}_{A} }$ and $H^{n_{1},0}(G_{A^{\circ}})=\frac{\Ker(\nabla_3: G^{n_{1},0}_{A}\rightarrow G^{n_{1}-2,0}_{B})}{\Ima: G^{n_{1},1}_{A}\rightarrow G^{n_{1},0}_{A} } $.
\end{rem} 
\begin{rem}
\label{remtecnico}
We recall that, on the elements of $V_A$, $-t$ acts as:
$$-t=\frac{x_{1}\partial_{x_{1}}+x_{2}\partial_{x_{2}}+x_{3}\partial_{x_{3}}+x_{4}\partial_{x_{4}}}{2}=\frac{x_{1}\partial_{1}+x_{2}\partial_{2}+x_{3}\partial_{3}-x_{4}\partial_{4}}{2}+x_{4}\partial_{4}.$$
We point out that:
\begin{align*}
\left[\frac{x_{1}\partial_{1}+x_{2}\partial_{2}+x_{3}\partial_{3}-x_{4}\partial_{4}}{2},w_{12}\right]=w_{12}, \quad & \, \,\left[\frac{x_{1}\partial_{1}+x_{2}\partial_{2}+x_{3}\partial_{3}-x_{4}\partial_{4}}{2},w_{14}\right]=0,\\
\left[\frac{x_{1}\partial_{1}+x_{2}\partial_{2}+x_{3}\partial_{3}-x_{4}\partial_{4}}{2},w_{13}\right]=w_{13}, \quad & \, \, \left[ \frac{x_{1}\partial_{1}+x_{2}\partial_{2}+x_{3}\partial_{3}-x_{4}\partial_{4}}{2},w_{24}\right]=0,\\
\left[\frac{x_{1}\partial_{1}+x_{2}\partial_{2}+x_{3}\partial_{3}-x_{4}\partial_{4}}{2},w_{23}\right]=w_{23}, \quad & \, \, \left[\frac{x_{1}\partial_{1}+x_{2}\partial_{2}+x_{3}\partial_{3}-x_{4}\partial_{4}}{2},w_{34}\right]=0,
\end{align*}
and
\begin{align*}
\left[x_{4}\partial_{4},w_{12}\right]=0, \quad & \, \,\left[x_{4}\partial_{4},w_{14}\right]=w_{14},\\
\left[x_{4}\partial_{4},w_{13}\right]=0, \quad & \, \,\left[x_{4}\partial_{4},w_{24}\right]=w_{24},\\
\left[x_{4}\partial_{4},w_{23}\right]=0, \quad & \, \,\left[x_{4}\partial_{4},w_{34}\right]=w_{34}.
\end{align*}
\end{rem}
Motivated by Remark \ref{remtecnico}, we introduce an additional bigrading:
\begin{align}
\label{bigradingck6}
&(V_{A})_{[p,q]}=\left\{f \in V_{A} \, : \, (x_{1}\partial_{1}+x_{2}\partial_{2}+x_{3}\partial_{3}-x_{4}\partial_{4} ).f=pf \,, \, (2x_{4}\partial_{4}).f=qf\right\},\\ \nonumber
&(G_{A})_{[p,q]}=\displaywedge(\g_{-1}) \otimes (V_{A})_{|[p,q]}.
\end{align}
We observe that, for elements in $(V_{A}^{n_{1},n_{2}})_{|[p,q]}$, we have $p+q=2n_{2}+n_{1}$, which is the eigenvalue of $-2t$ on $V_{A}^{n_{1},n_{2}}$. We point out that if $p<-\frac{q}{2}$, then $(V_{A}^{n_{1},n_{2}})_{|[p,q]}=0$. 
The bigrading can be defined also for $G_{A^{\circ}}$.\\
We define $d':=w_{12} \otimes \partial_{12}+w_{13} \otimes \partial_{13}+w_{23} \otimes \partial_{23}$ and $d''=w_{14} \otimes \partial_{14}+w_{24} \otimes \partial_{24}+w_{34}\otimes \partial_{34}$, so that $d'+d''=\nabla_{A}$. Using Remark \ref{remtecnico} and notation \eqref{notazioinewck6}, it is easy to check that $d':(G_{A})_{|[p,q]}\longrightarrow  (G_{A})_{|[p-2,q]}$ and $d'':(G_{A})_{|[p,q]}\longrightarrow  (G_{A})_{|[p,q-2]}$. By Remark \ref{remwck6}, it follows that that $d'^{2}=d''^{2}=d'd''+d''d'=0$. We point out that:
 $$\nabla_{A}:\oplus_{n_{2}+\frac{n_{1}}{2}=k}G^{n_{1},n_{2}}_{A}\longrightarrow \oplus_{n_{2}+\frac{n_{1}}{2}=k-1}G^{n_{1},n_{2}}_{A}.$$
Therefore $\oplus_{n_{2}+\frac{n_{1}}{2}=k}G^{n_{1},n_{2}}_{A}$ is a bicomplex with bigrading \eqref{bigradingck6}, differentials  $d',d''$ and total differential $\nabla_{A}=d'+d''$, the same holds for $\oplus_{n_{2}+\frac{n_{1}}{2}=k}G^{n_{1},n_{2}}_{A^{\circ}}$. Now let:
\begin{align*}
\displaywedge_{+}^{i}=\displaywedge^{i}\langle w_{14},w_{24},w_{34} \rangle \,\,\, \text{and}\,\,\,\displaywedge_{-}^{i}=\displaywedge^{i}\langle w_{12},w_{13},w_{23}\rangle.
\end{align*}
We define
\begin{align*}
&G_{A}(a,b)_{[p,q]}:=\displaywedge_{-}^{\frac{a-p}{2}}\displaywedge_{+}^{\frac{b-q}{2}}(V_{A})_{[p,q]}.
\end{align*}
We point out that $a-p$ and $b-q$ are always even: indeed $a $ (resp. $b $) is the eigenvalue of $x_{1}\partial_{1}+x_{2}\partial_{2}+x_{3}\partial_{3}-x_{4}\partial_{4}$ (resp. $2x_{4}\partial_{4}$) on elements in $\inlinewedge_{-}^{\frac{a-p}{2}}\inlinewedge_{+}^{\frac{b-q}{2}}(V_{A})_{[p,q]}$ and the elements of $\inlinewedge_{-}^{\frac{a-p}{2}}$ (resp. $\inlinewedge_{+}^{\frac{b-q}{2}}$) have even eigenvalue with respect to $x_{1}\partial_{1}+x_{2}\partial_{2}+x_{3}\partial_{3}-x_{4}\partial_{4}$ (resp. $2x_{4}\partial_{4}$), due to Remark \ref{remtecnico}. If $a $ is even (resp. odd), only even (resp. odd) values of $p$ occur. The values of $b$ and $q$ are always even. We point out that $a+b \geq 0$. Indeed this value represents the eigenvalue of $-2t$ on elements of $G_{A}(a,b)_{[p,q]}$ and $-2t$ has non negative eigenvalues on $V_{A}$ and eigenvalue 2 for elements of $\g_{-1}$. Moreover $b \geq 0$ since $(V_{A})_{[p,q]}\neq 0$ only for $q \geq 0$. We have that $G_{A}=\oplus_{a,b}G_{A}(a,b)$, where $G_{A}(a,b)=\oplus _{p,q}G_{A}(a,b)_{[p,q]}$. By Remark \ref{remtecnico}, it follows that $\nabla_{A}:G_{A}(a,b)\rightarrow G_{A}(a,b)$. The same definition holds for $G_{A^{\circ}}(a,b)_{[p,q]}$. The computation of homology spaces of $G_{A}$ and $G_{A^{\circ}}$ can be reduced to the computation for $G_{A}(a,b)$ and $G_{A^{\circ}}(a,b)$.\\
In order to compute the homology of those spaces, we need the following lemmas that state some properties of $D, D_{1},D_{2},D_{3},D_{4}$ introduced in Remark \ref{L1L2L3L4}. In particular, by Remark \ref{L1L2L3L4} we know that the elements of the $\g_0$-module $V_A =\oplus_{n_1,n_2}V^{n_1,n_2}_{A}$ are exactly the elements of $\Sym(\C^{4}) \otimes \Sym(\displaywedge^{2}\C^{4})$ that lie in $\Ker D \cap \bigcap^{4}_{i=1} \Ker  D_i$. In the following lemmas, we try to understand under which assumptions, given $P \in V_A$, there exists in $V_A$ a primitive of $P$, with respect of one of the variables in $\left\{x_{lk},  1\leq l<k\leq 4\right\}$.
From now on, given $P \in \Sym(\C^{4}) \otimes \Sym(\displaywedge^{2}\C^{4})$ and $z \in \left\{x_i,x_{lk}, 1 \leq i \leq 4, 1\leq l<k\leq 4\right\}$ such that $\partial_z P=0$, we use the notation $P(\widehat{z})$ to denote $P$ in order to emphasize that it does not depend on $z$.  Moreover, we point out that, given $P \in V_A$, we can always assume that $P$ is bihomogeneous of degree $n_1$ in the variables $x_1,x_2, x_3, x_4$ and of degree $n_2$ in the variables $x_{12},x_{13},x_{23},x_{14},x_{24},x_{34}$, for suitable $n_1$ and $n_2$: indeed $D, D_{1},D_{2},D_{3},D_{4},d',d''$ are graded operators.
\begin{lem}
\label{propL1} 
Let $i,j,h,k\in\left\{1,2,3\right\}$ with $i<j$ and $h<k$. Let $Q(\widehat{x}_{ij})\in  \Sym(\C^{4}) \otimes \Sym(\displaywedge^{2}\C^{4})$. Then there exists $\widetilde{Q}(\widehat{x}_{ij})\in  \Sym(\C^{4}) \otimes \Sym(\displaywedge^{2}\C^{4})$ such that $D_{4}\widetilde{Q}=Q$. Moreover, let $Q(\widehat{x}_{ij},\widehat{x}_{hk})\in  \Sym(\C^{4}) \otimes \Sym(\displaywedge^{2}\C^{4})$. Then there exists $\widetilde{Q}(\widehat{x}_{ij},\widehat{x}_{hk})\in  \Sym(\C^{4}) \otimes \Sym(\displaywedge^{2}\C^{4})$ such that $D_{4}\widetilde{Q}=Q$.
\end{lem} 
\begin{proof}
We prove the result for monomials. The result for a polynomial is obtained by combining the results for all the monomials involved in it.
For sake of simplicity, let us fix $i,j=1,2$; the other cases are analogous. Let $\mathcal I_{<}$ be the set of finite increasingly ordered sequences of elements in $\{1,2,3\}$; we write $I = i_1 \cdots i_r$ instead of $I = (i_1,..., i_r)$. We denote by $I^{c}$ the increasingly ordered sequence whose elements are the elements of the complement of the underlying set of $I$. If $I=k$, then $x_I=x_k$; if $I=kl$, then $x_I=x_{kl}$. We set $S=\left\{x_{1},x_{2},x_{3},x_{13},x_{23}\right\}$. We split the proof of the first thesis in the following steps.   
\begin{itemize}
	\item[i:] Let us suppose that $Q$ does not depend on any of the variables in $S$. Therefore we can choose $\widetilde{Q}=Q x_{1}x_{23}$.
	\item[ii:] Let us suppose that $Q$ depends only on one of the variables in $S$, that is $Q=P(x_{4},x_{14},x_{24},x_{34})x_I^{m}$, where $x_I \in S $. If $I\neq 3$, we can choose, up to a sign, $\widetilde{Q}=P\frac{x_I^{m+1}}{m+1}x_{I^{c}}$. If $I= 3$, we can choose $\widetilde{Q}=Px_{3}^{m}x_{1}x_{23}$.
	\item[iii:] Let us suppose that $Q$ depends only on two variables of $S$, that is $Q=P(x_{4},x_{14},x_{24},x_{34})x_{I_1}^{m_{1}}x_{I_2}^{m_{2}}$, where $x_{I_1},x_{I_2} \in S$. If $I_2=I_1^{c}$, then $x_{I_1}$ and $x_{I_2}$ are different from $x_{3}$ and we can choose, up to a sign, $\widetilde{Q}=P\frac{x_{I_1}^{m_{1}+1}}{m_{1}+1}\frac{x_{I_2}^{m_{2}+1}}{m_{2}+1}$. If  $I_2\neq I_1^{c}$, then at least one among $x_{I_1}$ and $x_{I_2}$ is different from  $x_{3}$. Let us suppose that $I_1\neq 3$. Hence we can choose, up to a sign, $\widetilde{Q}=P\frac{x_{I_1}^{m_{1}+1}}{m_{1}+1}x_{I_1^{c}}x_{I_2}^{m_{2}}$.
	\item[iv:] Let us suppose that $Q$ depends only on three variables of $S$, that is $Q=P(x_{4},x_{14},x_{24},x_{34})x_{I_1}^{m_{1}}x_{I_2}^{m_{2}}x_{I_3}^{m_{3}}$, where $x_{I_1},x_{I_2},x_{I_3} \in S$. Let us suppose that two of these three variables are complementary, for instance $I_2=I_1^{c}$. Then $x_{I_1}$ and $x_{I_2}$ are different from $x_{3}$ and we can choose, up to a sign, $\widetilde{Q}=P\frac{x_{I_1}^{m_{1}+1}}{m_{1}+1}\frac{x_{I_2}^{m_{2}+1}}{m_{2}+1}x_{I_3}^{m_{3}}$. If none of the $x_{I_k}$'s has its complementary among $x_{I_1}$,$\, x_{I_2}$,$\,x_{I_3}$, then at least one of them is different from  $x_{3}$. Let us suppose that $I_1 \neq 3$. Then we can choose, up to a sign, $\widetilde{Q}=P\frac{x_{I_1}^{m_{1}+1}}{m_{1}+1}x_{I_1^{c}}x_{I_2}^{m_{2}}x_{I_3}^{m_{3}}$.
	\item[v:] Let us suppose that $Q$ depends only on four variables of $S$, that is $Q=P(x_{4},x_{14},x_{24},x_{34})x_{I_1}^{m_{1}}x_{I_2}^{m_{2}}x_{I_3}^{m_{3}}x_{I_4}^{m_{4}}$, where $x_{I_1}, x_{I_2}, x_{I_3}, x_{I_4}\in S$. If $I_2=I_1^{c}$ and $I_4\neq I_3^{c}$, then $x_{I_1}$ and $x_{I_2}$ are different from $x_{3}$ and we can choose, up to a sign, $\widetilde{Q}=P\frac{x_{I_1}^{m_{1}+1}}{m_{1}+1}\frac{x_{I_2}^{m_{2}+1}}{m_{2}+1}x_{I_3}^{m_{3}}x_{I_4}^{m_{4}}$. 
	Now let us suppose $I_2=I_1^{c}$, $I_4= I_3^{c}$ and $m_{3} \leq m_{4}$. Let us therefore suppose that $\left\{x_{I_1},x_{I_2}\right\}=\left\{x_1,x_{23}\right\}$ and $\left\{x_{I_3},x_{I_4}\right\}=\left\{x_2,x_{13}\right\}$. We point out that
	\begin{align*}
	D_{4}\Big(P\frac{x_{I_1}^{m_{1}+1}}{m_{1}+1}\frac{x_{I_2}^{m_{2}+1}}{m_{2}+1}x_{I_3}^{m_{3}}x_{I_4}^{m_{4}}\Big)=P x_{I_1}^{m_{1}}x_{I_2}^{m_{2}}x_{I_3}^{m_{3}}x_{I_4}^{m_{4}}-\frac{m_{3}m_{4}Px_{I_1}^{m_{1}+1}x_{I_2}^{m_{2}+1}x_{I_3}^{m_{3}-1}x_{I_4}^{m_{4}-1}}{(m_{1}+1)(m_{2}+1)}.
	\end{align*}
	Moreover
	\begin{align*}
  &D_{4}\Big(m_{3}m_{4}P\frac{x_{I_1}^{m_{1}+2}x_{I_2}^{m_{2}+2}x_{I_3}^{m_{3}-1}x_{I_4}^{m_{4}-1}}{(m_{1}+1)(m_{2}+1)(m_{1}+2)(m_{2}+2)}\Big)=\\
	&\frac{m_{3}m_{4}P x_{I_1}^{m_{1}+1}x_{I_2}^{m_{2}+1}x_{I_3}^{m_{3}-1}x_{I_4}^{m_{4}-1}}{(m_{1}+1)(m_{2}+1)}-\frac{m_{3}m_{4}(m_{3}-1)(m_{4}-1)P x_{I_1}^{m_{1}+2}x_{I_2}^{m_{2}+2}x_{I_3}^{m_{3}-2}x_{I_4}^{m_{4}-2}}{(m_{1}+1)(m_{2}+1)(m_{1}+2)(m_{2}+2)}.
	\end{align*}
	We proceed by recursion until we obtain
	\begin{align*}
  &D_{4}\Big(\alpha P x_{I_1}^{m_{1}+m_{3}}x_{I_2}^{m_{2}+m_{3}}x_{I_3}x_{I_4}^{m_{4}-m_{3}+1}\Big)=\\
	&\alpha (m_{1}+m_{3})(m_{2}+m_{3})P x_{I_1}^{m_{1}+m_{3}-1} x_{I_2}^{m_{2}+m_{3}-1}x_{I_3} x_{I_4}^{m_{4}-m_{3}+1}-\alpha (m_{4}-m_{3}+1)P x_{I_1}^{m_{1}+m_{3}}x_{I_2}^{m_{2}+m_{3}}x_{I_4}^{m_{4}-m_{3}},
	\end{align*}
	where $\alpha$ is a suitable coefficient. By iv), for $P x_{I_1}^{m_{1}+m_{3}}x_{I_2}^{m_{2}+m_{3}}x_{I_4}^{m_{4}-m_{3}}$ the thesis holds since it depends only on three variables of $S$. Combining all these steps, the thesis holds for $Q$.
	\item[v:] Let us suppose that $Q=P(x_{4},x_{14},x_{24},x_{34})x_{1}^{m_{1}}x_{2}^{m_{2}}x_{3}^{m_{3}}x_{13}^{m_{4}}x_{23}^{m_{5}}$. By iv), we know that there exists $\widetilde{Q}(\widehat{x}_{12})$ such that $D_4(\widetilde{Q})=P(x_{4},x_{14},x_{24},x_{34})x_{1}^{m_{1}}x_{2}^{m_{2}}x_{13}^{m_{4}}x_{23}^{m_{5}}$. Therefore $D_4(\widetilde{Q}x_{3}^{m_{3}})=Q$. 
\end{itemize}
Now let us prove the second part of the statement. For sake of simplicity, let us fix $i,j=1,2$ and $h,k=1,3$; the other cases are analogous. We now set $S=\left\{x_{1},x_{2},x_{3},x_{23}\right\}$. We split the proof in the following steps.
\begin{itemize}
	\item[i:] Let us suppose that $Q$ does not depend on any of the variables in $S$. Therefore we can choose $\widetilde{Q}=Q x_{1}x_{23}$.
	\item[ii:] Let us suppose that $Q$ depends only on one of the variables of $S$, that is $Q=P(x_{4},x_{14},x_{24},x_{34})x_I^{m}$, where $x_I \in S$.
	If $x_I\notin\left\{ x_{2},x_{3}\right\}$, we can choose, up to a sign, $\widetilde{Q}=P\frac{x_I^{m+1}}{m+1}x_{I^{c}}$. If $x_I\in \left\{ x_{2},x_{3}\right\}$, we can choose $\widetilde{Q}=Px_I^{m}x_{1}x_{23}$.
	\item[iii:] Let us suppose that $Q$ depends only on two variables of $S$, that is $Q=P(x_{4},x_{14},x_{24},x_{34})x_{I_1}^{m_{1}}x_{I_2}^{m_{2}}$, where $x_{I_1},x_{I_2} \in S$. If $I_{2}=I_{1}^{c}$, then $x_{I_1}$ and $x_{I_2}$ are different from $x_{2}$ and $x_{3}$ and we can choose, up to a sign, $\widetilde{Q}=P\frac{x_{I_1}^{m_{1}+1}}{m_{1}+1}\frac{x_{I_2}^{m_{2}+1}}{m_{2}+1}$. Let  $I_2\neq I_1^{c}$. If $x_{I_1}=x_{2}$ and $x_{I_2}=x_{3}$, we can choose, up to a sign, $\widetilde{Q}=Px_{I_1}^{m_{1}}x_{I_2}^{m_{2}}x_{1}x_{23}$. If at least one among $x_{I_1}$ and $x_{I_2}$ is different from  $x_{2}$ and $x_{3}$, for instance $x_{I_1} \notin \left\{ x_{2},x_{3}\right\}$, we can choose, up to a sign, $\widetilde{Q}=P\frac{x_{I_1}^{m_{1}+1}}{m_{1}+1}x_{I_1^{c}}x_{I_2}^{m_{2}}$.
	\item[iv:] Let us suppose that $Q$ depends only on three variables of $S$, that is $Q=P(x_{4},x_{14},x_{24},x_{34})x_{I_1}^{m_{1}}x_{I_2}^{m_{2}}x_{I_3}^{m_{3}}$, where $x_{I_1},x_{I_2},x_{I_3} \in S$. Let us suppose that two of these three variables are complementary, for instance $I_2=I_1^{c}$. Hence $x_{I_1}$ and $x_{I_2}$ are different from $x_{2}$ and $x_{3}$ and we can choose, up to a sign, $\widetilde{Q}=P\frac{x_{I_1}^{m_{1}+1}}{m_{1}+1}\frac{x_{I_2}^{m_{2}+1}}{m_{2}+1}x_{I_3}^{m_{3}}$. If none of the $x_{I_k}$'s has its complementary among $x_{I_1},x_{I_2},x_{I_3}$, then at least one of them is different from  $x_{2} $ and $x_{3}$. Let us suppose that $x_{I_1} \notin \left\{ x_{2},x_{3}\right\}$. Hence we can choose, up to a sign, $\widetilde{Q}=P\frac{x_{I_1}^{m_{1}+1}}{m_{1}+1}x_{I_1^{c}}x_{I_2}^{m_{2}}x_{I_3}^{m_{3}}$.
	\item[v:] Let us suppose that $Q$ depends only on four variables of $S$, that is $Q=P(x_{4},x_{14},x_{24},x_{34})x_{1}^{m_{1}}x_{2}^{m_{2}}x_{3}^{m_{3}}x_{23}^{m_{4}}$. We can choose, up to a sign, $\widetilde{Q}=P\frac{x_{1}^{m_{1}+1}x_{2}^{m_{2}}x_{3}^{m_{3}}x_{23}^{m_{4}+1}}{(m_{1}+1)(m_{4}+1)}$. 
\end{itemize}
\end{proof} 
We point out that an analogous result holds also for $D,D_1,D_2,D_3$.
\begin{rem}
\label{notazioneintegrale}
We point out that from now on we will use the following notation. Let $P \in \Sym(\C^{4}) \otimes \Sym(\displaywedge^{2}\C^{4})$ and $z \in \left\{x_{lk}, 1\leq l<k\leq 4\right\}$, with $P=\sum_{l\geq 0} P_{l}(\widehat{z}) z^{l}$; by abuse of notation, we denote with $\int P d z$ the primitive $\sum_{l\geq 0} P_{l}\frac{z^{l+1}}{l+1}$, i.e. the unique primitive of $P$ with respect to $z$ that does not contain constant terms in $z$. We will write $\int P d z+Q(\widehat{z})$ when we need to consider a primitive that contains constant terms in $z$.
\end{rem}
\begin{rem}
\label{rempropL1}
We point out a consequence of Lemma \ref{propL1}. Let $P \in  \Sym(\C^{4}) \otimes \Sym(\displaywedge^{2}\C^{4})$, such that $P\in \Ker D_{4}$. When we apply $D_{4}$ to $\int P d x_{12}$, we get that $D_{4}\int P d x_{12}= Q(\widehat{x}_{12})$. By Lemma \ref{rempropL1}, there exists $\widetilde{Q}(\widehat{x}_{12})$ such that $D_4(\widetilde{Q}(\widehat{x}_{12}))=Q(\widehat{x}_{12})$. Hence $\int P d x_{12}-\widetilde{Q}(\widehat{x}_{12})$ is a primitive of $P$ that lies in $\Ker D_{4}$. Analogously, given $P(\widehat{x}_{12})\in \Ker D_{4}$, we consider the primitive $\int P(\widehat{x}_{12}) d x_{13}$. Therefore $D_4(\int P(\widehat{x}_{12}) d x_{13})=Q(\widehat{x}_{12},\widehat{x}_{13})$. By Lemma \ref{propL1}, there exists $\widetilde{Q}(\widehat{x}_{12},\widehat{x}_{13})$ such that $D_4(\widetilde{Q}(\widehat{x}_{12},\widehat{x}_{13}))=Q(\widehat{x}_{12},\widehat{x}_{13})$. Therefore $\int P(\widehat{x}_{12}) d x_{13}-\widetilde{Q}(\widehat{x}_{12},\widehat{x}_{13})$ is a primitive of $P$ that lies in $\Ker D_{4}$.
\end{rem}
\begin{rem}
\label{remtecnicoprimi}
Let $P\in \Sym(\C^{4}) \otimes \Sym(\displaywedge^{2}\C^{4})$, such that $P\in \Ker D \cap \bigcap^{4}_{i=1}\Ker D_{i}$. Let us write $P=\sum^{k}_{i= 0} x^{i}_{14}P_{i}(\widehat{x}_{14})$. We point out that, a priori, $Q=\int P dx_{14}$ does not lie in $\Ker D \cap \bigcap^{4}_{i=1}\Ker D_{i}$: indeed by direct computations it follows that:
\begin{align}
\label{appoggioprimitiva}
&D Q= \partial_{23}P_{0}, \\ \nonumber
&D_{1}Q=0,  \\ \nonumber
&D_{2}Q=-\partial_{3}P_{0},\\ \nonumber
&D_{3}Q=-\partial_{2}P_{0},\\ \nonumber
&D_{4}Q=0.
\end{align}
Hence in order to prove that there exists a primitive of $P$, with respect to $x_{14}$, that lies in $ \Ker D \cap \bigcap^{4}_{i=1}\Ker D_{i}$, it is sufficient to show that there exists $K(\widehat{x}_{14}) \in \Sym(\C^{4}) \otimes \Sym(\displaywedge^{2}\C^{4})$ such that 
\begin{align*}
&D K= -\partial_{23}P_{0}, \\
&D_{1}K=0, \\
&D_{2}K=\partial_{3}P_{0},\\
&D_{3}K=\partial_{2}P_{0},\\
&D_{4}K=0.
\end{align*}
\end{rem}
\begin{lem}
\label{propprimitiva} 
Let $P\in \Sym(\C^{4}) \otimes \Sym(\displaywedge^{2}\C^{4})$, such that $P\in \Ker D \cap \bigcap^{4}_{i=1}\Ker D_{i}$. For $s \in \left\{1,2,3\right\}$, there exists $K(\widehat{x}_{s4}) \in \Sym(\C^{4}) \otimes \Sym(\displaywedge^{2}\C^{4})$ such that $\int P d x_{s4}+K(\widehat{x}_{s4}) \in \Ker D \cap \bigcap^{4}_{i=1}\Ker D_{i} $, i.e. if $P\in V_A$, then there exists a primitive of $P$, with respect to $x_{s4}$, that lies in $V_A$.
\end{lem} 
\begin{proof}
For sake of simplicity we focus on $s=1$; the result follows analogously for $s=2,3$. Let us write $P=\sum^{k}_{i= 0} x^{i}_{14}P_{i}(\widehat{x}_{14})$. By Remark \ref{remtecnicoprimi} it is sufficient to find $K(\widehat{x}_{14})$ such that
\begin{align*}
&D K= -\partial_{23}P_{0}, \\
&D_{1}K=0, \\
&D_{2}K=\partial_{3}P_{0},\\
&D_{3}K=\partial_{2}P_{0},\\
&D_{4}K=0.
\end{align*}
Hence, if $\partial_{23}P_{0}=\partial_{2}P_{0}=\partial_{3}P_{0}=0$, it is sufficient to take $K=0$. If at least one of this derivatives is different from 0, in order to obtain such a $K$, we proceed as follows. We consider $Q=\int P dx_{14}$ that, by Remark \ref{remtecnicoprimi}, satisfies \eqref{appoggioprimitiva}. Then we recursively add to $Q$ polynomials that lie in $\Ker D \cap \bigcap^{4}_{i=1}\Ker D_{i}$ and decrease the degree in $x_{14}$, until we find $H(\widehat{x}_{14})$ that satisfies \eqref{appoggioprimitiva}. Finally we take $K(\widehat{x}_{14})=-H(\widehat{x}_{14})$ and the thesis follows. Let us see in detail this argument.  We have that $Q=\sum^{k}_{i= 0} \frac{x^{i+1}_{14}}{i+1}P_{i}(\widehat{x}_{14})$. We recall that we always assume that $P \in V_A$ is bihomogeneous of degree $n_1$ in the $x_i$'s and of degree $n_2$ in the $x_{lk}$'s, for suitable $n_1$ and $n_2$. Hence $P_{k} \in \Sym^{n_1}(\C^{4}) \otimes \Sym^{n_2-k}(\displaywedge^{2}\C^{4})$. Thus $Q \in \Sym^{n_1}(\C^{4}) \otimes \Sym^{n_2+1}(\displaywedge^{2}\C^{4})$. The fact that $D_{1}P=D_{4}P=0$, implies that $D_{1}P_{k}=D_{4}P_{k}=0$. Moreover, $D P=0$ implies $(\partial_{12}\partial_{34}-\partial_{13}\partial_{24})P_{k}=0$, $D_2 P=0$ implies $(\partial_{13}\partial_{4}+\partial_{34}\partial_{1})P_{k}=0$ and $D_3 P=0$ implies $(\partial_{12}\partial_{4}+\partial_{24}\partial_{1})P_{k}=0$. Hence $P_{k}=P_{k}(\widehat{x}_{14}) \in \Ker D\cap \bigcap^{4}_{i=1}\Ker D_{i}$. By Remark \ref{L1L2L3L4}, this means that $P_{k}$ is obtained by the action of elements of $\mathfrak{sl}_4$ on $v_{n_1,n_2-k}=x^{n_1}_{1}x^{n_2-k}_{12}$. We claim that there exists $N_{k} \in \Sym^{n_1}(\C^{4}) \otimes \Sym^{n_2+1}(\displaywedge^{2}\C^{4})$ such that $N_{k} \in \Ker D \cap \bigcap^{4}_{i=1}\Ker D_{i}$ and the term of maximum degree in $x_{14}$ of $N_k$ is exactly $\frac{x^{k+1}_{14}}{k+1}P_{k}$. Indeed let us take $v_{n_1,n_2+1}=x^{n_1}_{1}x^{n_2-k}_{12}x^{k+1}_{14}$, that is a generator of the irreducible $\mathfrak{sl}_4-$submodule of $\Sym^{n_1}(\C^{4}) \otimes \Sym^{n_2+1}(\displaywedge^{2}\C^{4})$. Due to the fact that $P_{k}$ is obtained by the action of elements of $\mathfrak{sl}_4$ on $v_{n_1,n_2-k}$ and it does not depend on $x_{14}$, by applying the same elements of $\mathfrak{sl}_4$ on $v_{n_1,n_2+1}$ is it possible to obtain $N_k=\frac{x^{k+1}_{14}}{k+1}P_{k}+T$, where $T$ contains monomials with at most degree $k$ in $x_{14}$. By Remark \ref{L1L2L3L4}, $N_k\in \Ker D \cap \bigcap^{4}_{i=1}\Ker D_{i}$; hence $Q-N_k$ satisfies \eqref{appoggioprimitiva} and has at most degree $k$ in $x_{14}$. Now let $Q-N_k=\sum^{k}_{i= 0} x^{i}_{14}S_{i}$. Since $P_0$ does not depend on $x_{14}$ and $Q-N_k$ satisfies \eqref{appoggioprimitiva}, again $S_k \in \Ker D \cap \bigcap^{4}_{i=1}\Ker D_{i}$ and we can iterate the argument.
\end{proof}
\begin{prop}
\label{propprimitiva2} 
For $a,b \in \left\{1,2,3\right\}$. Let $P=P(\widehat{x}_{a4})\in \Sym(\C^{4}) \otimes \Sym(\displaywedge^{2}\C^{4})$, such that $P\in \Ker D \cap \bigcap^{4}_{i=1}\Ker D_{i}$. Then there exists $K(\widehat{x}_{a4},\widehat{x}_{b4}) \in \Sym(\C^{4}) \otimes \Sym(\displaywedge^{2}\C^{4})$ such that $\int P d x_{b4}+K(\widehat{x}_{a4},\widehat{x}_{b4}) \in \Ker D \cap \bigcap^{4}_{i=1}\Ker D_{i} $, i.e. there exists a primitive of $P$ that lies in $\Ker D \cap \bigcap^{4}_{i=1}\Ker D_{i}$ and does not depend on $x_{a4}$.
\end{prop} 
\begin{proof}
For sake of simplicity we focus on $a=3$ and $b=1$; the proof for the other cases follows analogously, see Remark \ref{remprimitivatecnico} . Let us write $P=\sum^{k}_{i= 0} x^{i}_{14}P_{i}(\widehat{x}_{14})$. Since $P\in \Ker D \cap \bigcap^{4}_{i=1}\Ker D_{i}$, by Remark \ref{remtecnicoprimi}, it follows that $Q=\int P dx_{14}$ satisfies:
\begin{align}
\label{appoggioprimitiva2}
&D Q= \partial_{23}P_{0}, \\ \nonumber
&D_{1}Q=0, \\ \nonumber
&D_{2}Q=-\partial_{3}P_{0},\\ \nonumber
&D_{3}Q=-\partial_{2}P_{0},\\ \nonumber
&D_{4}Q=0,
\end{align}
where $P_0=P_0(\widehat{x}_{14},\widehat{x}_{34})$. If $\partial_{23}P_{0}=\partial_{2}P_{0}=\partial_{3}P_{0}=0$, it is sufficient to take $K=0$. Let us suppose that at least one of these derivatives is nonzero. Our aim is to find $K(\widehat{x}_{14},\widehat{x}_{34})$ such that
\begin{align*}
&D K= -\partial_{23}P_{0}, \\
&D_{1}K=0, \\
&D_{2}K=\partial_{3}P_{0},\\
&D_{3}K=\partial_{2}P_{0},\\
&D_{4}K=0,
\end{align*}
so that $Q+K(\widehat{x}_{14},\widehat{x}_{34}) \in \Ker D \cap \bigcap^{4}_{i=1}\Ker D_{i}$ is a primitive of $P$.
In order to find such $K(\widehat{x}_{14},\widehat{x}_{34})$, we need the following lemma.
\begin{lem}
\label{propprimitivatecnico}
Let $P=P(\widehat{y}_{3})\in \C[x_{1},..,x_{4},y_{1},...,y_{6}]$ such that $P=\sum_{i \geq 0} y^{i}_{1}P_{i}(\widehat{y}_1)$. Let us suppose that
\begin{align*}
\begin{cases}
\textbf{D} P=(\partial_{y_4}\partial_{y_3}-\partial_{y_5}\partial_{y_2}+\partial_{y_6}\partial_{y_1})P=0,\\
\textbf{D}_1 P=(\partial_{y_6}\partial_{x_4}-\partial_{y_2}\partial_{x_3}+\partial_{y_3}\partial_{x_2})P=0,\\
\textbf{D}_2 P=(\partial_{y_5}\partial_{x_4}-\partial_{y_1}\partial_{x_3}+\partial_{y_3}\partial_{x_1})P=0,\\
\textbf{D}_3 P=(\partial_{y_4}\partial_{x_4}-\partial_{y_1}\partial_{x_2}+\partial_{y_2}\partial_{x_1})P=0,\\
\textbf{D}_4 P=(\partial_{y_4}\partial_{x_3}-\partial_{y_5}\partial_{x_2}+\partial_{y_6}\partial_{x_1})P=0.
\end{cases}
\end{align*}
Then there exists $K(\widehat{y}_{1},\widehat{y}_{3})$ such that $\textbf{D} K=-\partial_{y_6} P_0$, $\textbf{D}_1 K=0$, $\textbf{D}_2 K=\partial_{x_3} P_0$, $\textbf{D}_3 K=\partial_{x_2} P_0$ and $\textbf{D}_4 K=0$.
\end{lem} 
We point out that the proof of Proposition \ref{propprimitiva2} follows by Lemma \ref{propprimitivatecnico} using the correspondence
\begin{align*} 
&x_{1} \mapsto x_{1},\,\, x_{2} \mapsto x_{2}, \,\, x_{3} \mapsto x_{3},\,\, x_{4} \mapsto x_{4},\,\, y_{1} \mapsto x_{14},\\
&y_{2} \mapsto x_{24},\,\, y_{3} \mapsto x_{34},\,\, y_{4} \mapsto x_{12},\,\, y_{5} \mapsto x_{13},\,\, y_{6} \mapsto x_{23}.
\end{align*}
\begin{proof}[Proof of Lemma \ref{propprimitivatecnico}]
Let us first fix some notation: $P_0=\sum_{l\geq 0} y^{l}_{5}U_l$, $P_0=\sum_{i\geq 0} x^{i}_{4}H_i$ and $H_0=\sum_{s\geq 0} y^{s}_{5}M_s$. This means that $U_0=U_0(\widehat{y}_{1},\widehat{y}_{3},\widehat{y}_{5})$, $H_0=H_0(\widehat{x}_{4},\widehat{y}_{1},\widehat{y}_{3})$ and $M_0=M_0(\widehat{x}_{4},\widehat{y}_{1},\widehat{y}_{3},\widehat{y}_{5})$. We point out that $\textbf{D}_1 P=\textbf{D}_4 P=0$ implies that $\textbf{D}_1 P_0=\textbf{D}_4 P_0=0$ and therefore $\textbf{D}_1 U_0=\textbf{D}_4 H_0=0$. We consider
\begin{align*}
K(\widehat{y}_{1},\widehat{y}_{3})=&\int \int \big(\partial_{x_3} P_0 \big)dy_{5}dx_{4}+\int \int \big(\partial_{y_6} H_0\big) dy_{5}dy_{2}\\
&+\sum_{j \geq 0}(-1)^{j}\int \int \big(\partial^{j}_{y_2}\partial^{j}_{x_1}\partial_{x_2} U_0 \big)dx^{j+1}_{4}dy^{j+1}_{4}\\
&+\sum_{j \geq 0}(-1)^{j}\int \int \big(\partial^{j}_{y_4}\partial^{j}_{x_3}\partial_{x_2} \partial_{y_6}M_0 \big)dy^{j+1}_{6}dx^{j+1}_{1} dy_2\\
&-\sum_{k \geq 0}\sum_{j \geq k}(-1)^{j-k}\int \int\int \big(\partial^{j}_{y_2}\partial^{j-k}_{x_1}\partial^{k}_{x_3}\partial_{x_2} \partial_{y_6}M_0\big) dx^{j+1}_{4} dy^{j+1-k}_{4}dy^{k+1}_{6}\\
&+\sum_{k \geq 0}\sum^{k}_{j=0}(-1)^{k-j}\int \int\int \big(\partial^{k+1}_{x_3} \partial^{j}_{y_2} \partial^{k-j}_{y_4} \partial_{x_2} \partial_{y_6}M_0 \big)dx^{j+1}_{4} dx^{k+1-j}_{1}dy^{k+2}_{6}.
\end{align*}
By direct computations it satisfies the thesis.
\end{proof}
From the proof of Lemma \ref{propprimitivatecnico}, the thesis follows.
\end{proof}
\begin{rem}
\label{remprimitivatecnico} 
We point out that the proof of Proposition \ref{propprimitiva2} for $a=3$ and $b=2$ follows by a similar argument. Indeed in this case $Q=\int P dx_{24}$ satisfies:
\begin{align}
\label{appoggioprimitiva2}
&D Q= -\partial_{13}P_{0}, \\ \nonumber
&D_{1}Q=-\partial_{3}P_{0}, \\ \nonumber
&D_{2}Q=0,\\ \nonumber
&D_{3}Q=\partial_{1}P_{0},\\ \nonumber
&D_{4}Q=0,
\end{align}
where $P_0=P_0(\widehat{x}_{24},\widehat{x}_{34})$. If $\partial_{13}P_{0}=\partial_{1}P_{0}=\partial_{3}P_{0}=0$, it is sufficient to take $K=0$. Otherwise, our aim is to find $K(\widehat{x}_{24},\widehat{x}_{34})$ such that
\begin{align}
\label{appoggioK23}
&D K= \partial_{13}P_{0}, \\\nonumber
&D_{1}K=\partial_{3}P_{0}, \\ \nonumber
&D_{2}K=0,\\ \nonumber
&D_{3}K=-\partial_{1}P_{0}, \\ \nonumber
&D_{4}K=0.
\end{align}
The existence of such $K(\widehat{x}_{24},\widehat{x}_{34})$ follows again by Lemma \ref{propprimitivatecnico}, once we rename the variables as follows:
\begin{align*} 
&x_{1} \mapsto x_{2},\,\, x_{2} \mapsto x_{1}, \,\, x_{3} \mapsto x_{3},\,\, x_{4} \mapsto x_{4},\,\, y_{1} \mapsto y_{2},\\
&y_{2} \mapsto y_{1},\,\, y_{3} \mapsto y_{3}, \,\, y_{4} \mapsto -y_{4},\,\, y_{5} \mapsto y_{6}, \,\,y_{6} \mapsto y_{5}.
\end{align*} 
Therefore the statement of Lemma \ref{propprimitivatecnico} turns into: $P=P(\widehat{y}_{3})$, $P\in \Ker \textbf{D} \cap \bigcap_{i=1}^{4}\Ker \textbf{D}_i $ and $P=\sum_{i \geq 0} y^{i}_{2}P_{i}(\widehat{y}_2)$. The differential operators turn into
\begin{align*} 
\textbf{D} =\partial_{y_4}\partial_{y_3}-\partial_{y_5}\partial_{y_2}+\partial_{y_6}\partial_{y_1} &\mapsto -\partial_{y_4}\partial_{y_3}-\partial_{y_6}\partial_{y_1}+\partial_{y_5}\partial_{y_2}=-\textbf{D},\\
\textbf{D}_1 =\partial_{y_6}\partial_{x_4}-\partial_{y_2}\partial_{x_3}+\partial_{y_3}\partial_{x_2}&\mapsto \partial_{y_5}\partial_{x_4}-\partial_{y_1}\partial_{x_3}+\partial_{y_3}\partial_{x_1}=\textbf{D}_2,\\
\textbf{D}_2 =\partial_{y_5}\partial_{x_4}-\partial_{y_1}\partial_{x_3}+\partial_{y_3}\partial_{x_1}&\mapsto \partial_{y_6}\partial_{x_4}-\partial_{y_2}\partial_{x_3}+\partial_{y_3}\partial_{x_2}=\textbf{D}_1,\\
\textbf{D}_3 =\partial_{y_4}\partial_{x_4}-\partial_{y_1}\partial_{x_2}+\partial_{y_2}\partial_{x_1}&\mapsto  -\partial_{y_4}\partial_{x_4}-\partial_{y_2}\partial_{x_1}+\partial_{y_1}\partial_{x_2}=-\textbf{D}_3 ,\\
\textbf{D}_4 =\partial_{y_4}\partial_{x_3}-\partial_{y_5}\partial_{x_2}+\partial_{y_6}\partial_{x_1}&\mapsto -\partial_{y_4}\partial_{x_3}-\partial_{y_6}\partial_{x_1}+\partial_{y_5}\partial_{x_2}=-\textbf{D}_4.
\end{align*} 
Hence, by Lemma \ref{propprimitivatecnico}, if follows that there exists $K(\widehat{y}_{2},\widehat{y}_{3})$ such that $\textbf{D} K=\partial_{y_5} P_0$, $\textbf{D}_1 K=\partial_{x_3} P_0$, $\textbf{D}_2 K=0$, $\textbf{D}_3 K=-\partial_{x_1} P_0$ and $\textbf{D}_4 K=0$.
Using the correspondence 
\begin{align*} 
&x_{1} \mapsto x_{1},\,\, x_{2} \mapsto x_{2}, \,\, x_{3} \mapsto x_{3},\,\, x_{4} \mapsto x_{4},\,\, y_{1} \mapsto x_{14},\\
&y_{2} \mapsto x_{24},\,\, y_{3} \mapsto x_{34},\,\, y_{4} \mapsto x_{12},\,\, y_{5} \mapsto x_{13},\,\, y_{6} \mapsto x_{23}.
\end{align*}
we obtain the existence of $K(\widehat{x}_{24},\widehat{x}_{34})$ that satisfies \eqref{appoggioK23}.\\
Similarly the proof of Proposition \ref{propprimitiva2} for $a=2$ and $b=1$ follows by Lemma \ref{propprimitivatecnico}, once we rename the variables as follows:
\begin{align*} 
&x_{1} \mapsto x_{1},\,\, x_{2} \mapsto x_{3}, \,\, x_{3} \mapsto x_{2},\,\, x_{4} \mapsto x_{4},\,\, y_{1} \mapsto y_{1},\\
&y_{2} \mapsto y_{3}, \,\, y_{3} \mapsto y_{2},\,\, y_{4} \mapsto y_{5},\,\, y_{5} \mapsto y_{4}, \,\, y_{6} \mapsto -y_{6}.
\end{align*} 
Similarly the proof of Proposition \ref{propprimitiva2} for $a=2$ and $b=3$ follows by Lemma \ref{propprimitivatecnico}, once we rename the variables as follows:
\begin{align*} 
&x_{1} \mapsto x_{3},\,\, x_{2} \mapsto x_{1}, \,\, x_{3} \mapsto x_{2},\,\, x_{4} \mapsto x_{4},\,\, y_{1} \mapsto y_{3},\\
&y_{2} \mapsto y_{1}, \,\, y_{3} \mapsto y_{2}, \,\, y_{4} \mapsto -y_{5},\,\, y_{5} \mapsto -y_{6},  \,\, y_{6} \mapsto y_{4}.
\end{align*} 
The proof of Proposition \ref{propprimitiva2} for $a=1$ and $b=2$ follows by Lemma \ref{propprimitivatecnico}, once we rename the variables as follows:
\begin{align*} 
&x_{1} \mapsto x_2,\,\, x_{2} \mapsto x_3, \,\, x_{3} \mapsto x_1,\,\, x_{4} \mapsto x_4,\,\, y_{1} \mapsto y_2,\\
&y_{2} \mapsto y_3, \,\, y_{3} \mapsto y_1, \,\, y_{4} \mapsto y_6, \,\, y_{5} \mapsto -y_4,  \,\, y_{6} \mapsto -y_5.
\end{align*} 
The proof of Proposition \ref{propprimitiva2} for $a=1$ and $b=3$ follows by Lemma \ref{propprimitivatecnico}, once we rename the variables as follows:
\begin{align*} 
&x_{1} \mapsto x_{3},\,\, x_{2} \mapsto x_{2}, \,\, x_{3} \mapsto x_{1},\,\, x_{4} \mapsto x_{4},\,\,y_{1} \mapsto y_{3},\\
&y_{2} \mapsto y_{2},\,\, y_{3} \mapsto y_{1}, \,\, y_{4} \mapsto -y_{6},\,\, y_{5} \mapsto -y_{5},\,\, \,\, y_{6} \mapsto -y_{4}.
\end{align*} 
\end{rem} 
\begin{rem} 
We fix some notation that will be used in the proof of the following result. Given $r \in \mathbb{R}$, we use the notation
\begin{align*}
\lfloor r \rfloor=\max \left\{m \in \Z, m\leq r\right\},\\
\lceil r \rceil=\min \left\{m \in \Z, m\geq r\right\}.
\end{align*}
Given $n \in \mathbb{Z}$, we will denote by $\mathcal{P}(n)$ the parity of $n$, i.e. $\mathcal{P}(n)=0$ if $n$ is even, $\mathcal{P}(n)=1$ if $n$ is odd.
\end{rem} 
Now, using the theory of spectral sequences for bicomplexes, Lemma \ref{propprimitiva} and Proposition \ref{propprimitiva2}, we are ready to prove the following result on the homology of the $G_{A}(a,b)$'s.
\begin{lem} 
\label{4.3ck6}
Let $a,b$ be such that $b \geq 0$, $a+b\geq 0$.\\
As $\langle x_{1}\partial_{1}-x_{2}\partial_{2}, x_{1}\partial_{2},x_{2}\partial_{1}, x_{2}\partial_{2}-x_{3}\partial_{3}, x_{2}\partial_{3},x_{3}\partial_{2}, x_{1}\partial_{3},x_{3}\partial_{1}\rangle-$modules:
\begin{gather*}
H^{n_{1},n_{2}}(G_{A}(a,b))= 0 \quad \text{for}\,\,\, (n_{1},n_{2})\neq (0,1)\,\, and \,\, n_{2}>0,\\
H^{0,1}(G_{A}(2,2))\cong \C,\\
H^{0,1}(G_{A}(a,b))\cong 0 \quad \text{for}\,\,\, (a,b) \neq (2,2).
\end{gather*}
Moreover $H^{n_1,0}(G_A (a,b))=0$ except for the following cases:
\begin{gather*}
H^{n_1,0}(G_{A}(a,b)) \cong \displaywedge_{-}^{\frac{a+b-n_1}{2}}\displaywedge_{+}^{0}\otimes (V_{A}^{n_1,0})_{[n_1-b,b]} \quad \text{for}\,\, b \geq 4;\\
H^{n_1,0}(G_{A}(a,2)) \cong \displaywedge_{-}^{\frac{a+2-n_1}{2}}\displaywedge_{+}^{0}\otimes (V_{A}^{n_1,0})_{[n_1-2,2]}, \quad \text{for}\,\,\,n_1\in \left\{a, a+2\right\}; \\
H^{n_1,0}(G_{A}(a,2)) \cong\\
 \frac{\displaywedge_{-}^{2}\displaywedge_{+}^{0}\otimes (V_{A}^{n_1,0})_{[n_1-2,2]}}{\left\{w_{12}w_{13}\otimes x_{4}\partial_{1}p+w_{12}w_{23}\otimes x_{4}\partial_{2}p+w_{13}w_{23}\otimes x_{4}\partial_{3}p, \,\, p(x_1 ,x_2, x_3) \in (V_{A}^{n_1,0})_{[n_1,0]}\right\}}, \quad \text{for}\,\,\,n_1=a-2;  \\
H^{n_1,0}(G_{A}(a,2)) \cong 0, \quad \text{for}\,\,\,n_1=a-4,\\
H^{n_1,0}(G_{A}(a,0)) \cong \displaywedge_{-}^{0}\displaywedge_{+}^{0}\otimes (V_{A}^{n_1,0})_{[n_1,0]}, \quad \text{for}\,\, n_1=a;\\
H^{n_1,0}(G_{A}(a,0)) \cong \left\{w_{12}\otimes f, \,\,\, f \in (V_{A}^{n_1,0})_{[n_1,0]} \,\,\,|\,\,\, f=f(x_{1},x_{2},x_{3})\,\,\, and \,\,\, x_3| f \right\} \quad \text{for}\,\, n_1=a-2;\\.
\end{gather*}
\end{lem}
\begin{proof}
	We first modify the bigrading \eqref{bigradingck6} so that we obtain a new bigrading and we can use the theory of spectral sequences for bicomplexes. For every $p,q$ we denote by $\widetilde{p}=\lceil \frac{p}{2} \rceil$ and $ \widetilde{q}=\frac{q}{2}$. We use the notation $(\widetilde{V}_{A})_{[\widetilde{p},\widetilde{q}]}=(V_{A})_{[p,q]}$ and $\widetilde{G}_{A}(a,b)_{[\widetilde{p},\widetilde{q}]}=G_{A}(a,b)_{[p,q]}$. Thus $d':\widetilde{G}_{A}(a,b)_{[\widetilde{p},\widetilde{q}]} \longrightarrow  \widetilde{G}_{A}(a,b)_{[\widetilde{p}-1,\widetilde{q}]}$, $d'':\widetilde{G}_{A}(a,b)_{[\widetilde{p},\widetilde{q}]}\longrightarrow  \widetilde{G}_{A}(a,b)_{[\widetilde{p},\widetilde{q}-1]}$ and $G_{A}(a,b)=\oplus _{\widetilde{p},\widetilde{q}}\widetilde{G}_{A}(a,b)_{[\widetilde{p},\widetilde{q}]}$. We split the proof in the following subcases.
	\begin{enumerate}[leftmargin=0.5cm]
		\item[1)] Let $b>6$ and $a+b > 6$. Let us consider $G_{A}(a,b)$ with the differential $d''$:
\begin{align*}
\xleftarrow[]{d''}\displaywedge_{-}^{\frac{a-(2\widetilde{p}-\mathcal{P}(p))}{2}}\displaywedge_{+}^{\frac{b-2\widetilde{q}}{2}+1}\otimes(\widetilde{V}_{A})_{[\widetilde{p},\widetilde{q}-1]} \xleftarrow[]{d''} \displaywedge_{-}^{\frac{a-(2\widetilde{p}-\mathcal{P}(p))}{2}}\displaywedge_{+}^{\frac{b-2\widetilde{q}}{2}}\otimes(\widetilde{V}_{A})_{[\widetilde{p},\widetilde{q}]}\xleftarrow[]{d''} \displaywedge_{-}^{\frac{a-(2\widetilde{p}-\mathcal{P}(p))}{2}}\displaywedge_{+}^{\frac{b-2\widetilde{q}}{2}-1}\otimes(\widetilde{V}_{A})_{[\widetilde{p},\widetilde{q}+1]}  \xleftarrow[]{d''}.
\end{align*}
It is the tensor product of $\inlinewedge_{-}^{\frac{a-(2\widetilde{p}-\mathcal{P}(p))}{2}}$ and the following complex, since $\inlinewedge_{-}^{\frac{a-(2\widetilde{p}-\mathcal{P}(p))}{2}}$ is not involved by $d''$:
\begin{align*}
0\xleftarrow[]{d''}\displaywedge_{+}^{3}\otimes(\widetilde{V}_{A})_{[\widetilde{p},\frac{b}{2}-3]}  \xleftarrow[]{d''}\displaywedge_{+}^{2}\otimes(\widetilde{V}_{A})_{[\widetilde{p},\frac{b}{2}-2]} \xleftarrow[]{d''} \displaywedge_{+}^{1}\otimes(\widetilde{V}_{A})_{[\widetilde{p},\frac{b}{2}-1]} \xleftarrow[]{d''} \displaywedge_{+}^{0}\otimes(\widetilde{V}_{A})_{[\widetilde{p},\frac{b}{2}]} \xleftarrow[]{d''} 0.
\end{align*}
We have that $b>6$ assures that for $\widetilde{q}$ the value $\frac{b}{2}-3$ is acceptable, since it is positive.
Let us show that this complex is exact except for the right end.
\begin{description}[leftmargin=0cm]
	\item[i] Let us consider the map $d'': \inlinewedge_{+}^{0}\otimes(\widetilde{V}_{A})_{[\widetilde{p},\frac{b}{2}]} \longrightarrow \inlinewedge_{+}^{1}\otimes(\widetilde{V}_{A})_{[\widetilde{p},\frac{b}{2}-1]}  $. Let us compute its kernel. Let $f \in \inlinewedge_{+}^{0}\otimes(\widetilde{V}_{A})_{[\widetilde{p},\frac{b}{2}]} $. We have $d''(f)=w_{14} \otimes \partial_{14}f+w_{24} \otimes \partial_{24}f+w_{34} \otimes \partial_{34}f=0$ if and only if $\partial_{14}f=\partial_{24}f=\partial_{34}f=0$. We point out that $f\in \Ker D \cap\bigcap^{4}_{i=1} \Ker D_{i}$ and $\partial_{14}f=\partial_{24}f=\partial_{34}f=0$ imply that $\partial_{12}\partial_{4}f=\partial_{13}\partial_{4}f=\partial_{23}\partial_{4}f=0$. Since $b>6$, it follows that $f$ does not depend on the $x_{ij}$'s, i.e. $f=f(x_{1},x_{2},x_{3},x_{4})$.
\item[ii] Let us consider the map $d'': \inlinewedge_{+}^{1}\otimes(\widetilde{V}_{A})_{[\widetilde{p},\frac{b}{2}-1]} \longrightarrow  \inlinewedge_{+}^{2}\otimes(\widetilde{V}_{A})_{[\widetilde{p},\frac{b}{2}-2]} $. Let us compute its kernel. Let $v=w_{14}\otimes p_{1}+w_{24}\otimes p_{2}+w_{34}\otimes p_{3}\in \inlinewedge_{+}^{1} \otimes(\widetilde{V}_{A})_{[\widetilde{p},\frac{b}{2}-1]} $. We have:
\begin{align*}
d''(v)=&w_{14}w_{24}\otimes \partial_{24} p_{1}+w_{14}w_{34}\otimes \partial_{34} p_{1}+w_{24}w_{34}\otimes \partial_{34} p_{2}+w_{24}w_{14}\otimes \partial_{14} p_{2}+w_{34}w_{14}\otimes \partial_{14} p_{3}+w_{34}w_{24}\partial_{24} p_{3}.
\end{align*}
Hence $d''(v)=0$ if and only if:
\begin{align}
\label{kerpasso1d''}
\begin{cases}
\partial_{24} p_{1}-\partial_{14} p_{2}=0,\\
\partial_{34} p_{1}-\partial_{14} p_{3}=0,\\
\partial_{34} p_{2}-\partial_{24} p_{3}=0.
\end{cases}
\end{align}
Moreover, combining \eqref{kerpasso1d''} and $p_{1},p_{2},p_{3}\in \Ker D \cap\bigcap^{4}_{i=1} \Ker D_{i}$, we obtain the following equations:
\begin{align}
\label{altrekerpasso1d''}
\partial_{12} p_{3}-\partial_{13} p_{2}+\partial_{23} p_{1}=G(\widehat{x}_{14},\widehat{x}_{24},\widehat{x}_{34}),\\ \nonumber
\partial_{12}\partial_{4} p_{1}=\partial_{14}(\partial_{2}p_{1}-\partial_{1}p_{2}),\\ \nonumber
\partial_{12}\partial_{4} p_{2}=\partial_{24}(\partial_{2}p_{1}-\partial_{1}p_{2}),\\ \nonumber
\partial_{12}\partial_{4} p_{3}=\partial_{34}(\partial_{2}p_{1}-\partial_{1}p_{2}),\\  \nonumber
\partial_{23}\partial_{4} p_{1}=\partial_{14}(\partial_{3}p_{2}-\partial_{2}p_{3}),\\  \nonumber
\partial_{23}\partial_{4} p_{2}=\partial_{24}(\partial_{3}p_{2}-\partial_{2}p_{3}),\\  \nonumber
\partial_{23}\partial_{4} p_{3}=\partial_{34}(\partial_{3}p_{2}-\partial_{2}p_{3}),\\  \nonumber
\partial_{13}\partial_{4} p_{1}=\partial_{14}(\partial_{3}p_{1}-\partial_{1}p_{3}),\\  \nonumber
\partial_{13}\partial_{4} p_{2}=\partial_{24}(\partial_{3}p_{1}-\partial_{1}p_{3}),\\  \nonumber
\partial_{13}\partial_{4} p_{3}=\partial_{34}(\partial_{3}p_{1}-\partial_{1}p_{3}).
\end{align}
In order that there exists $W\in (\widetilde{V}_{A})_{[\widetilde{p},\frac{b}{2}]}$, i.e. $W \in\Ker D \cap\bigcap^{4}_{i=1} \Ker D_{i}$, such that $d''(W)=v$, that is $\partial_{14}W=p_{1},\partial_{24}W=p_{2},\partial_{34}W=p_{3}$, it is necessary that
$$\partial_{12} p_{3}-\partial_{13} p_{2}+\partial_{23} p_{1}=0.$$
Indeed $0=D W=\partial_{12} p_{3}-\partial_{13} p_{2}+\partial_{23} p_{1}$.
We know, by \eqref{altrekerpasso1d''}, that 
$$\partial_{12} p_{3}-\partial_{13} p_{2}+\partial_{23} p_{1}=G(\widehat{x}_{14},\widehat{x}_{24},\widehat{x}_{34}).$$
Since the weight with respect to $2x_{4}\partial_4$ of $\partial_{12} p_{3}-\partial_{13} p_{2}+\partial_{23} p_{1}$ is $b-2>4$, the weight with respect to $2x_{4}\partial_4$ of $G(\widehat{x}_{14},\widehat{x}_{24},\widehat{x}_{34})$ is $b-2>4$. Hence $G$ depends on $x_4$. Moreover $G\in \Ker D \cap\bigcap^{4}_{i=1}\Ker D_i$, since $p_{1},p_{2},p_{3}\in  \Ker D \cap \bigcap^{4}_{i=1}\Ker D_i$. This implies $\partial_{12}\partial_{4}G=\partial_{13}\partial_{4}G=\partial_{23}\partial_{4}G=0$. Hence $G$ does not depend on the $x_{ij}$'s, i.e. $G=G(x_{1},x_{2},x_{3},x_{4})$. We now show that, combining \eqref{kerpasso1d''}, \eqref{altrekerpasso1d''}, $b-2>4$ and $p_{1},p_{2},p_{3}\in  \Ker D \cap \bigcap^{4}_{i=1}\Ker D_i$, we obtain $G=0$. Indeed
\begin{align*}
\partial_{4}G=&\partial_{4}\partial_{12} p_{3}-\partial_{4}\partial_{13} p_{2}+\partial_{4}\partial_{23} p_{1}=\partial_{4}\partial_{12} p_{3}- \partial_{3}\partial_{24} p_{1}+\partial_{1}\partial_{24} p_{3}+\partial_{4}\partial_{23} p_{1}\\
=&\partial_{4}\partial_{12} p_{3}+\partial_{1}\partial_{24} p_{3}-\partial_{2}\partial_{34} p_{1}=\partial_{4}\partial_{12} p_{3}+\partial_{1}\partial_{24} p_{3}-\partial_{2}\partial_{14} p_{3}=0.
\end{align*}
Hence $\partial_{12} p_{3}-\partial_{13} p_{2}+\partial_{23} p_{1}=0$. Let us show that a $v$, that satisfies \eqref{kerpasso1d''} lies in the image of $d''$.\\
 By solving \eqref{kerpasso1d''}, we obtain that $p_{1}=\int \partial_{14}p_{2}dx_{24}+S_{1}(\widehat{x}_{24})$ and $p_{3}=\int \partial_{34}p_{2}dx_{24}+\int \partial_{34}S_{1}(\widehat{x}_{24})dx_{14}+S_{3}(\widehat{x}_{14},\widehat{x}_{24})$. We consider 
$$P=\int p_{2}dx_{24}+\int S_{1}(\widehat{x}_{24})dx_{14}+\int S_{3}(\widehat{x}_{14},\widehat{x}_{24})dx_{34}. $$
$P$ is constructed so that $\partial_{14}P=p_{1},\partial_{24}P=p_{2},\partial_{34}P=p_{3}$. Then $\partial_{12} p_{3}-\partial_{13} p_{2}+\partial_{23} p_{1}=0$ implies $D P=0$.
By $D_4(p_{1})=0$ we obtain $D_4(S_{1}(\widehat{x}_{24}))=0$, and by $D_4(p_{3})=0$ we obtain $D_4(S_{3}(\widehat{x}_{24},\widehat{x}_{14}))=0$. Hence we deduce that $D_4 P=0$. We notice that:
\begin{align}
\label{Q1Q2Q3}
Q_1:=D_1 P=\partial_{23}\partial_{4} P-\partial_{3}p_{2}+\partial_{2}p_{3},\\ \nonumber
Q_2:=D_2 P=\partial_{13}\partial_{4} P-\partial_{3}p_{1}+\partial_{1}p_{3},\\ \nonumber
Q_3:=D_3 P=\partial_{12}\partial_{4} P-\partial_{2}p_{1}+\partial_{1}p_{2}.
\end{align}
By \eqref{altrekerpasso1d''} we obtain that $\partial_{14}Q_i=\partial_{24}Q_i=\partial_{34}Q_i=0$ for all $i=1,2,3$. Moreover, combining \eqref{Q1Q2Q3} and $D_4 P=0$, we obtain $\partial_{1} Q_{1}-\partial_{2} Q_{2}+\partial_{3} Q_{3}=0$. Furthermore $D_4 P=0$ implies $D_4 Q_i=0$ for all $i=1,2,3$.
Let us notice that 
\begin{align*}
 \partial_{13}Q_3-\partial_{12}Q_2=&\partial_{13}\partial_{12}\partial_{4} P-\partial_{13}\partial_{2}p_{1}+\partial_{13}\partial_{1}p_{2}-\partial_{12}\partial_{13}\partial_{4} P+\partial_{12}\partial_{3}p_{1}-\partial_{12}\partial_{1}p_{3}\\
=&-\partial_{1}(\partial_{12} p_{3}-\partial_{13} p_{2}+\partial_{23} p_{1})=0.
\end{align*}
Analogously we obtain that
\begin{align}
\label{primaQ}
\partial_{13} Q_{3}=\partial_{12} Q_{2},\\ \label{secQ}
\partial_{23}Q_3=\partial_{12}Q_1,\\ \label{terQ}
\partial_{23}Q_2=\partial_{13}Q_1.
\end{align}
Our aim is now to find $Q$ such that $D_1 Q=Q_1, D_2 Q=Q_2, D_3 Q=Q_3$ and $D_4 Q=D Q= d''(Q)=0$, so that $v=d''(P-Q)$. By \eqref{primaQ} we know that $Q_2=\int \partial_{13} Q_{3} d x_{12}+S(\widehat{x}_{12})$. By Remark \ref{rempropL1}, we know that there exists $K(\widehat{x}_{12})$ such that $D_4(\int Q_{3} d x_{12}+K(\widehat{x}_{12}))=0.$ Let $Z(\widehat{x}_{12})$ be such that $S(\widehat{x}_{12})=\partial_{13} K(\widehat{x}_{12})+Z(\widehat{x}_{12})$. Hence $Q_2=\partial_{13}(\int  Q_{3} d x_{12}+K(\widehat{x}_{12}))+Z(\widehat{x}_{12})$. By $D_4 Q_{2}=0$ and $D_4(\int Q_{3} d x_{12}+K(\widehat{x}_{12}))=0$, we get $D_4 Z(\widehat{x}_{12})=0$. By Remark \ref{rempropL1}, we know that there exists $C(\widehat{x}_{12},\widehat{x}_{13})$ such that $D_4(\int Z(\widehat{x}_{12}) d x_{13}+C(\widehat{x}_{12},\widehat{x}_{13}))=0.$ Hence $Q_2=\partial_{13}(\int  Q_{3} d x_{12}+K(\widehat{x}_{12})+\int Z(\widehat{x}_{12})d x_{13}+C(\widehat{x}_{12},\widehat{x}_{13}))$ and the expression in the parenthesis lies in $\Ker D_4$. We rename $A(\widehat{x}_{12}):=K(\widehat{x}_{12})+\int Z(\widehat{x}_{12})d x_{13}+C(\widehat{x}_{12},\widehat{x}_{13})$, so that $Q_2=\partial_{13}(\int  Q_{3} d x_{12}+A(\widehat{x}_{12}))$, with $\int  Q_{3} d x_{12}+A(\widehat{x}_{12}) \in \Ker D_4$.  
 By \eqref{secQ} and an analogous argument, we get that $Q_1=\partial_{23}(\int  Q_{3} d x_{12}+B(\widehat{x}_{12}))$, with $\int  Q_{3} d x_{12}+B(\widehat{x}_{12}) \in \Ker D_4$. 
Thus by \eqref{terQ}, we obtain $\partial_{13}\partial_{23}(A(\widehat{x}_{12})-B(\widehat{x}_{12}))=0$, which means that $B(\widehat{x}_{12})=A(\widehat{x}_{12})+Y_{1}(\widehat{x}_{12},\widehat{x}_{13})+Y_{2}(\widehat{x}_{12},\widehat{x}_{23})$. Now let us use that $\partial_{1} Q_{1}-\partial_{2} Q_{2}+\partial_{3} Q_{3}=0$. Indeed
\begin{align*}
0=&\partial_{1} Q_{1}-\partial_{2} Q_{2}+\partial_{3} Q_{3}\\
=&\partial_{1}\partial_{23}\Big(\int  Q_{3} d x_{12}+B(\widehat{x}_{12})\Big)  - \partial_{2} \partial_{13}\Big(\int  Q_{3} d x_{12} +A(\widehat{x}_{12})\Big)+\partial_{3}\partial_{12}\int  Q_{3} d x_{12}\\
=&D_{4}\Big(\int  Q_{3} d x_{12}\Big)-\partial_{2}\partial_{13}A(\widehat{x}_{12})+\partial_{1}\partial_{23}B(\widehat{x}_{12})\\
=&-D_{4}A(\widehat{x}_{12})-\partial_{2}\partial_{13}A(\widehat{x}_{12})+\partial_{1}\partial_{23}(A(\widehat{x}_{12})+Y_{1}(\widehat{x}_{12},\widehat{x}_{13}))\\
=&-D_{4}A(\widehat{x}_{12})+D_4A(\widehat{x}_{12})+\partial_{1}\partial_{23}Y_{1}(\widehat{x}_{12},\widehat{x}_{13})\\
=&\partial_{1}\partial_{23}Y_{1}(\widehat{x}_{12},\widehat{x}_{13}),
\end{align*}
where we used that $D_{4}A(\widehat{x}_{12})=-D_4(\int  Q_{3} d x_{12})$. We point out that $0=\partial_{1}\partial_{23}Y_{1}(\widehat{x}_{12},\widehat{x}_{13})=D_4 Y_{1}(\widehat{x}_{12},\widehat{x}_{13})$. We consider
\begin{align*}
Q:=\int \Big( \int  Q_{3} d x_{12}+A(\widehat{x}_{12})+Y_{1}(\widehat{x}_{12},\widehat{x}_{13})\Big)d x_{4}.
\end{align*}
By its construction $D_4Q=0$, since $D_4(\int  Q_{3} d x_{12}+A(\widehat{x}_{12}))=0$ and $D_4 Y_{1}(\widehat{x}_{12},\widehat{x}_{13})=0$. Moreover $D Q=d''(Q)=0$, since $Q$ does not depend on $x_{14},x_{24},x_{34}$. Then
\begin{align*}
D_1Q=&\partial_{23}\partial_{4}Q=\partial_{23}\Big(\int  Q_{3} d x_{12}\Big)+\partial_{23}\Big(A(\widehat{x}_{12})+Y_{1}(\widehat{x}_{12},\widehat{x}_{13})\Big)\\
=&\partial_{23}\Big(\int  Q_{3} d x_{12}\Big)+\partial_{23}(B(\widehat{x}_{12}))=Q_1,\\
D_2Q=&\partial_{13}\partial_{4}Q=\partial_{13}\Big(\int  Q_{3} d x_{12}+A(\widehat{x}_{12})\Big)=Q_2,\\
D_3Q=&\partial_{12}\partial_{4}Q=Q_3.
\end{align*}
Hence if $v\in \Ker d''$, then $v=d''(P-Q)$ with $P-Q \in \Ker D \cap\bigcap^{4}_{i=1} \Ker D_i$. Thus at this point the sequence is exact.
\item[iii] Let us consider the map $d'': \inlinewedge_{+}^{2}\otimes (\widetilde{V}_{A})_{[\widetilde{p},\frac{b}{2}-2]} \longrightarrow  \inlinewedge_{+}^{3}\otimes(\widetilde{V}_{A})_{[\widetilde{p},\frac{b}{2}-3]}  $. Let us compute its kernel. Let $v=w_{14}w_{24}\otimes p_{1}+w_{14}w_{34}\otimes p_{2}+w_{24}w_{34}\otimes p_{3}\in \inlinewedge_{+}^{2}\otimes(\widetilde{V}_{A})_{[\widetilde{p},\frac{b}{2}-2]}  $. We have:
\begin{align*}
d''(v)=w_{14}w_{24}w_{34}\otimes (\partial_{34} p_{1}-\partial_{24} p_{2}+\partial_{14} p_{3}).
\end{align*}
Therefore $d''(v)=0$ if and only if $\partial_{34} p_{1}-\partial_{24} p_{2}+\partial_{14} p_{3}=0$, that is equivalent to $p_{1}=\int (\partial_{24}p_{2}--\partial_{14}p_{3})dx_{34}+S(\widehat{x}_{34})$.
We consider $w=w_{14}\otimes (\int p_{2}dx_{34}+K_{1}(\widehat{x}_{34}))+w_{24}\otimes (\int p_{3}dx_{34}+K_{2}(\widehat{x}_{34}))$. We obtain that:
\begin{align*}
d''(w)=w_{14}w_{24}\otimes \Big( \int (\partial_{24}p_{2}-\partial_{14}p_{3})dx_{34} +\partial_{24}K_{1}(\widehat{x}_{34})-\partial_{14}K_{2}(\widehat{x}_{34})\Big)
+w_{14}w_{34}\otimes p_{2}+w_{24}w_{34}\otimes p_{3}.
\end{align*}
Hence:
\begin{align*}
v-d''(w)&=w_{14}w_{24}\otimes \big(S(\widehat{x}_{34}) -\partial_{24}K_{1}(\widehat{x}_{34})+\partial_{14}K_{2}(\widehat{x}_{34})\big)\\
&=d''\Big(-w_{24}\otimes  \Big( \int \big(S(\widehat{x}_{34}) -\partial_{24}K_{1}(\widehat{x}_{34})+\partial_{14}K_{2}(\widehat{x}_{34}) \big)dx_{14}+H(\widehat{x}_{14},\widehat{x}_{34} )\Big)\Big).
\end{align*}
Thus at this point the sequence is exact.
\item[iv] Let us consider the map $d'': \inlinewedge_{+}^{3}\otimes(\widetilde{V}_{A})_{[\widetilde{p},\frac{b}{2}-3]} \longrightarrow 0$. We have that $ \inlinewedge_{+}^{3}\otimes(\widetilde{V}_{A})_{[\widetilde{p},\frac{b}{2}-3]} \ni w_{14}w_{24}w_{34}\otimes f=d''(w_{24}w_{34}\otimes (\int f d_{14}+K(\widehat{x}_{14})))$.
Thus at this point the sequence is exact.
\end{description}
In the following diagram we use the notation $K_{\widetilde{p}}:= \left\{f \in (\widetilde{V}_{A})_{[\widetilde{p},\frac{b}{2}]}, \,\, | \,\, f=f(x_{1},x_{2},x_{3},x_{4})\right\}$.
The following is the diagram of the $E^{'1}$ spectral sequence, where the vertical maps are $d''$ and the horizontal maps are induced by $d'$:
\begin{center}
\begin{tikzpicture}
\node[black] at (0,0) {$\inlinewedge_{-}^{3}\inlinewedge_{+}^{0}\otimes K_{\lceil \frac{a}{2}\rceil-3} $};
\draw[->,black] (0,-0.4) -- (0,-0.6);
\node[black] at (3.5,0) {$\inlinewedge_{-}^{2}\inlinewedge_{+}^{0}\otimes K_{\lceil \frac{a}{2}\rceil-2} $};
\draw[->,black] (3.5,-0.4) -- (3.5,-0.6);
\node[black] at (7,0) {$\inlinewedge_{-}^{1}\inlinewedge_{+}^{0}\otimes K_{\lceil \frac{a}{2}\rceil-1}  $};
\draw[->,black] (7,-0.4) -- (7,-0.6);
\node[black] at (10.5,0) {$\inlinewedge_{-}^{0}\inlinewedge_{+}^{0}\otimes K_{\lceil \frac{a}{2}\rceil}   $};
\draw[->,black] (10.5,-0.4) -- (10.5,-0.6);
\draw[->,black] (2,0) -- (1.5,0);
\draw[->,black] (5.5,0) -- (5,0);
\draw[->,black] (9,0) -- (8.5,0);
\node[black] at (0,-1) {$0$};
\draw[->,black] (0,-1.4) -- (0,-1.6);
\node[black] at (3.5,-1) {$0 $};
\draw[->,black] (3.5,-1.4) -- (3.5,-1.6);
\node[black] at (7,-1) {$0$};
\draw[->,black] (7,-1.4) -- (7,-1.6);
\node[black] at (10.5,-1) {$0 $};
\draw[->,black] (10.5,-1.4) -- (10.5,-1.6);
\draw[->,black] (2,-1) -- (1.5,-1);
\draw[->,black] (9,-1) -- (8.5,-1);
\draw[->,black] (5.5,-1) -- (5,-1);
\node[black] at (0,-2) {$0$};
\node[black] at (3.5,-2) {$0 $};
\node[black] at (7,-2) {$0$};
\node[black] at (10.5,-2) {$0 $};
\draw[->,black] (2,-2) -- (1.5,-2);
\draw[->,black] (9,-2) -- (8.5,-2);
\draw[->,black] (5.5,-2) -- (5,-2);
\node[black] at (0,-3) {$0$};
\draw[->,black] (0,-2.4) -- (0,-2.6);
\node[black] at (3.5,-3) {$0 $};
\draw[->,black] (3.5,-2.4) -- (3.5,-2.6);
\node[black] at (7,-3) {$0$};
\draw[->,black] (7,-2.4) -- (7,-2.6);
\node[black] at (10.5,-3) {$0 $.};
\draw[->,black] (10.5,-2.4) -- (10.5,-2.6);
\draw[->,black] (2,-3) -- (1.5,-3);
\draw[->,black] (9,-3) -- (8.5,-3);
\draw[->,black] (5.5,-3) -- (5,-3);
%
%
\end{tikzpicture}
\end{center}
The only nonzero row is:
\begin{align*}
0\xleftarrow[]{d'}\displaywedge_{-}^{3}\displaywedge_{+}^{0}\otimes K_{\lceil \frac{a}{2}\rceil-3} \xleftarrow[]{d'}  \displaywedge_{-}^{2}\displaywedge_{+}^{0}\otimes K_{\lceil \frac{a}{2}\rceil-2}  \xleftarrow[]{d'} \displaywedge_{-}^{1}\displaywedge_{+}^{0}\otimes K_{\lceil \frac{a}{2}\rceil-1} \xleftarrow[]{d'}  \displaywedge_{-}^{0}\displaywedge_{+}^{0}\otimes K_{\lceil \frac{a}{2}\rceil} \xleftarrow[]{d'}0.
\end{align*}
We point out that condition $a+b \geq 6$ allows that for $\widetilde{p}$ the value $\lceil \frac{a}{2}\rceil-3$ is acceptable. Indeed, if $a$ is even, then the values of $p$ are all even and $p+q\geq 0$ implies $\widetilde{p}+\widetilde{q}\geq 0$. Hence $\widetilde{p}+\widetilde{q}=\frac{a}{2}-3+\frac{b}{2} \geq 0$ is satisfied. If $a$ is odd, then the values of $p$ are all odd and $p+q\geq 0$ implies $\widetilde{p}+\widetilde{q}\geq 1$. Hence $\widetilde{p}+\widetilde{q}=\frac{a+1}{2}-3+\frac{b}{2} \geq 1$ is satisfied.
Since in $K_{\widetilde{p}}$ there are only elements that do not depend on the $x_{jk}$'s, then $d'$ is identically zero on this row. Hence $E'^{2}=E'^{1}$. Moreover for a one row spectral sequence we have that $E'^{2}=E'^{\infty}$, then we obtain that
\begin{align*}
\oplus_{n_1,n_2}H^{n_{1},n_{2}}(G_{A}(a,b)) \cong \sum_{\widetilde{p},\widetilde{q}}E'^{\infty}_{\widetilde{p},\widetilde{q}} =\sum^{3}_{i=0}E'^{\infty}_{\lceil \frac{a}{2}\rceil-i,\frac{b}{2}}\cong \sum^{3}_{i=0}\displaywedge_{-}^{i}\displaywedge_{+}^{0}\otimes K_{\lceil \frac{a}{2}\rceil-i}.
\end{align*}
Since in $ K_{\lceil \frac{a}{2}\rceil-i}$ there are only elements that do not depend on the $x_{lk} $'s for all the $l,k \in \left\{1,2,3,4\right\}$, we obtain:
\begin{align*}
H^{n_1,0}(G_{A}(a,b)) \cong \displaywedge_{-}^{\frac{a+b-n_1}{2}}\displaywedge_{+}^{0}\otimes (V_{A}^{n_1,0})_{[n_1-b,b]}.
\end{align*}
\item[2)] 
Let $b>6$ and $a+b =:h \leq 6$. The computation of $E^{'1}$ is analogous to the previous case and we obtain the same diagram, but the only nonzero row is now:
\begin{align*}
0\xleftarrow[]{d'} \displaywedge_{-}^{\lfloor\frac{h}{2}\rfloor}\displaywedge_{+}^{0}\otimes K_{\lceil \frac{a}{2}\rceil-\lfloor\frac{h}{2}\rfloor} \xleftarrow[]{d'} \cdots  \xleftarrow[]{d'} \displaywedge_{-}^{1}\displaywedge_{+}^{0}\otimes K_{\lceil \frac{a}{2}\rceil-1} \xleftarrow[]{d'}  \displaywedge_{-}^{0}\displaywedge_{+}^{0}\otimes K_{\lceil\frac{a}{2}\rceil} \xleftarrow[]{d'}0.
\end{align*}
Indeed condition $ 0 \leq a+b=h$ allows $\widetilde{p}$ to be $\lceil \frac{a}{2}\rceil-\lfloor\frac{h}{2}\rfloor$. If $a$ is even, then $h$ is even; hence, since for all the $p,q$'s we have $p+q \geq 0 $, we have that $\widetilde{p}+\widetilde{q} \geq 0$. We have that $\widetilde{p}+\widetilde{q}=\frac{a}{2}-\frac{h}{2}+\frac{b}{2}=0$, hence $\widetilde{p}=\lceil \frac{a}{2}\rceil-\lfloor\frac{h}{2}\rfloor$ is acceptable. Similarly, if $a$ is odd, then $h$ is odd; hence, since for all the $p,q$'s we have $p+q \geq 0 $, we have that $\widetilde{p}+\widetilde{q} \geq 1$. We have that $\widetilde{p}+\widetilde{q}=\frac{a+1}{2}-\frac{h-1}{2}+\frac{b}{2}=1$, hence $\widetilde{p}=\lceil \frac{a}{2}\rceil-\lfloor\frac{h}{2}\rfloor$ is acceptable. Analogously to case 1), since in $K_{\widetilde{p}}$ there are only elements that do not depend on the $x_{jk}$'s, for all $l,k \in \left\{1,2,3,4\right\}$ and for a one row spectral sequence we have that $E'^{2}=E'^{\infty}$, we obtain that
\begin{align*}
\oplus_{n_1,n_2}H^{n_{1},n_{2}}(G_{A}(a,b)) \cong \sum_{\widetilde{p},\widetilde{q}}E'^{\infty}_{\widetilde{p},\widetilde{q}} =\sum^{\lfloor\frac{a}{2}\rfloor+\frac{b}{2}}_{i=0}E'^{\infty}_{\lceil \frac{a}{2}\rceil-i,\frac{b}{2}}=\sum^{\lfloor\frac{a}{2}\rfloor+\frac{b}{2}}_{i=0}\displaywedge_{-}^{i}\displaywedge_{+}^{0}\otimes K_{\lceil \frac{a}{2}\rceil-i},
\end{align*}
and
\begin{align*}
H^{n_1,0}(G_{A}(a,b)) \cong \displaywedge_{-}^{\frac{a+b-n_1}{2}}\displaywedge_{+}^{0}\otimes (V_{A}^{n_1,0})_{[n_1-b,b]}.
\end{align*}
\item[3)]
Let $4 \leq b \leq 6$ and $a+b > 6$.
Let us consider $G_{A}(a,b)$ with the differential $d''$; we obtain the product of $\inlinewedge_{-}^{\frac{a-(2\widetilde{p}-\mathcal{P}(p))}{2}}$ and the following complex:
\begin{align*}
0\xleftarrow[]{d''}\displaywedge_{+}^{\frac{b}{2}}\otimes(\widetilde{V}_{A})_{[\widetilde{p},0]}  \xleftarrow[]{d''}\displaywedge_{+}^{\frac{b}{2}-1}\otimes(\widetilde{V}_{A})_{[\widetilde{p},1]} \xleftarrow[]{d''}...\xleftarrow[]{d''} \displaywedge_{+}^{0}\otimes(\widetilde{V}_{A})_{[\widetilde{p},\frac{b}{2}]} \xleftarrow[]{d''} 0.
\end{align*}
In $\inlinewedge_{+}^{\frac{b}{2}}\otimes(\widetilde{V}_{A})_{[\widetilde{p},0]}$ the sequence is exact. Indeed if $b=6$, in $\inlinewedge_{+}^{3}\otimes(\widetilde{V}_{A})_{[\widetilde{p},0]}$ there are only elements of type $v=w_{14}w_{24}w_{34} \otimes  p(\widehat{x}_{4},\widehat{x}_{14},\widehat{x}_{24},\widehat{x}_{34})$, due to $\widetilde{q}=0$; thus $v=d''(w_{24}w_{34}\otimes(\int p dx_{14}+K(\widehat{x}_{14})))$. If $b=4$, in $\inlinewedge_{+}^{2}\otimes(\widetilde{V}_{A})_{[\widetilde{p},0]}$ there are only elements of type $v=\sum w_{i_1 4}w_{i_2 4} \otimes p(\widehat{x}_{4},\widehat{x}_{14},\widehat{x}_{24},\widehat{x}_{34})$, with the $i_j$'s in $\left\{1,2,3\right\}$; thus $v\in \Ima d''$, since $w_{i_14}w_{i_24} \otimes p(\widehat{x}_{4},\widehat{x}_{14},\widehat{x}_{24},\widehat{x}_{34})=d''(w_{i_1 4}\otimes(\int p dx_{i_2 4}+K(\widehat{x}_{i_2 4},\widehat{x}_{i_3 4})))$, where $i_3 $ is such that $\left\{i_1,i_2,i_3\right\}=\left\{1,2,3\right\}$. The rest of the computations is analogous to 1). Hence, as in the computation of $E'^{1}$ for the case 1), we obtain that the sequence is exact except for the right end, where the kernel is given by $K_{\widetilde{p}}=\left\{f \in (\widetilde{V}_{A})_{[\widetilde{p},\frac{b}{2}]}, \,\, | \,\, f=f(x_{1},x_{2},x_{3},x_{4})\right\}$. Therefore we obtain the following diagram for $E'^{1}$, where the vertical maps are $d''$ and the horizontal maps are induced by $d'$:
\begin{center}
\begin{tikzpicture}
\node[black] at (0,0) {$\inlinewedge_{-}^{3}\inlinewedge_{+}^{0}\otimes K_{\lceil \frac{a}{2}\rceil-3}$};
\draw[->,black] (0,-0.4) -- (0,-0.6);
\node[black] at (3.5,0) {$\inlinewedge_{-}^{2}\inlinewedge_{+}^{0}\otimes K_{\lceil \frac{a}{2}\rceil-2}$};
\draw[->,black] (3.5,-0.4) -- (3.5,-0.6);
\node[black] at (7,0) {$\inlinewedge_{-}^{1}\inlinewedge_{+}^{0}\otimes K_{\lceil \frac{a}{2}\rceil-1}$};
\draw[->,black] (7,-0.4) -- (7,-0.6);
\node[black] at (10.5,0) {$\inlinewedge_{-}^{0}\inlinewedge_{+}^{0}\otimes K_{\lceil \frac{a}{2}\rceil}$};
\draw[->,black] (10.5,-0.4) -- (10.5,-0.6);
\draw[->,black] (2,0) -- (1.5,0);
\draw[->,black] (9,0) -- (8.5,0);
\draw[->,black] (5.5,0) -- (5,0);
\node[black] at (0,-1) {$0$};
\draw[->,black] (0,-1.4) -- (0,-1.6);
\node[black] at (3.5,-1) {$0 $};
\draw[->,black] (3.5,-1.4) -- (3.5,-1.6);
\node[black] at (7,-1) {$0$};
\draw[->,black] (7,-1.4) -- (7,-1.6);
\node[black] at (10.5,-1) {$0 $};
\draw[->,black] (10.5,-1.4) -- (10.5,-1.6);
\draw[->,black] (2,-1) -- (1.5,-1);
\draw[->,black] (9,-1) -- (8.5,-1);
\draw[->,black] (5.5,-1) -- (5,-1);
\node[black] at (0,-1.9) {$\vdots$};
\node[black] at (3.5,-1.9) {$\vdots$};
\node[black] at (7,-1.9) {$\vdots$};
\node[black] at (10.5,-1.9) {$\vdots$};
\node[black] at (0,-3) {$0$};
\draw[->,black] (0,-2.4) -- (0,-2.6);
\node[black] at (3.5,-3) {$0 $};
\draw[->,black] (3.5,-2.4) -- (3.5,-2.6);
\node[black] at (7,-3) {$0$};
\draw[->,black] (7,-2.4) -- (7,-2.6);
\node[black] at (10.5,-3) {$0 .$};
\draw[->,black] (10.5,-2.4) -- (10.5,-2.6);
\draw[->,black] (2,-3) -- (1.5,-3);
\draw[->,black] (9,-3) -- (8.5,-3);
\draw[->,black] (5.5,-3) -- (5,-3);
\end{tikzpicture}
\end{center}
 We can conclude as in case 1).
\item[4)] 
Let $4 \leq b\leq 6$ and $a+b=:h \leq 6$. The computation of $E'^{1}$ is analogous to the previous case and we obtain the same diagram, but the only nonzero row is now:
\begin{align*}
0\xleftarrow[]{d'} \displaywedge_{-}^{\lfloor \frac{h}{2}\rfloor}\displaywedge_{+}^{0}\otimes K_{\lceil \frac{a}{2}\rceil-\lfloor \frac{h}{2}\rfloor} \xleftarrow[]{d'} \cdots  \xleftarrow[]{d'} \displaywedge_{-}^{1}\displaywedge_{+}^{0}\otimes K_{\lceil \frac{a}{2}\rceil-1} \xleftarrow[]{d'}  \displaywedge_{-}^{0}\displaywedge_{+}^{0}\otimes K_{\lceil \frac{a}{2}\rceil} \xleftarrow[]{d'}0.
\end{align*}
We can conclude as in case 2).
\item[5)] 
Let $b=0$ and $a > 6$. We point out that $b=0$ forces the first complex to be the tensor product of $\inlinewedge_{-}^{\frac{a-(2\widetilde{p}-\mathcal{P}(p))}{2}}$ and the following complex:
\begin{align*}
0\xleftarrow[]{d''}\displaywedge_{+}^{0}\otimes(\widetilde{V}_{A})_{[\widetilde{p},0]} \xleftarrow[]{d''} 0.
\end{align*}
Again we denote by $K_{\widetilde{p}}=(\widetilde{V}_{A})_{[\widetilde{p},0]}$; we point out that, due to $\widetilde{q}=0$, $f \in K_{\widetilde{p}}$ is such that $f=f(x_{1},x_{2},x_{3},x_{12},x_{13},x_{23})$ and $f\in \Ker D_4$.
The following is the diagram of the $E'^{1}$ spectral sequence, where the vertical maps are $d''$ and the horizontal maps are induced by $d'$:
\begin{center}
\begin{tikzpicture}
\node[black] at (0,0) {$\inlinewedge_{-}^{3}\inlinewedge_{+}^{0}\otimes K_{\lceil \frac{a}{2}\rceil-3} $};
\draw[->,black] (0,-0.4) -- (0,-0.6);
\node[black] at (3.5,0) {$\inlinewedge_{-}^{2}\inlinewedge_{+}^{0}\otimes K_{\lceil \frac{a}{2}\rceil-2} $};
\draw[->,black] (3.5,-0.4) -- (3.5,-0.6);
\node[black] at (7,0) {$\inlinewedge_{-}^{1}\inlinewedge_{+}^{0}\otimes K_{\lceil \frac{a}{2}\rceil-1}  $};
\draw[->,black] (7,-0.4) -- (7,-0.6);
\node[black] at (10.5,0) {$\inlinewedge_{-}^{0}\inlinewedge_{+}^{0}\otimes K_{\lceil \frac{a}{2}\rceil}   $};
\draw[->,black] (10.5,-0.4) -- (10.5,-0.6);
\draw[->,black] (2,0) -- (1.5,0);
\draw[->,black] (5.5,0) -- (5,0);
\draw[->,black] (9,0) -- (8.5,0);
\node[black] at (0,-1) {$0$};
\draw[->,black] (0,-1.4) -- (0,-1.6);
\node[black] at (3.5,-1) {$0 $};
\draw[->,black] (3.5,-1.4) -- (3.5,-1.6);
\node[black] at (7,-1) {$0$};
\draw[->,black] (7,-1.4) -- (7,-1.6);
\node[black] at (10.5,-1) {$0 $};
\draw[->,black] (10.5,-1.4) -- (10.5,-1.6);
\draw[->,black] (2,-1) -- (1.5,-1);
\draw[->,black] (9,-1) -- (8.5,-1);
\draw[->,black] (5.5,-1) -- (5,-1);
\node[black] at (0,-2) {$0$};
\node[black] at (3.5,-2) {$0 $};
\node[black] at (7,-2) {$0$};
\node[black] at (10.5,-2) {$0 $};
\draw[->,black] (2,-2) -- (1.5,-2);
\draw[->,black] (9,-2) -- (8.5,-2);
\draw[->,black] (5.5,-2) -- (5,-2);
\node[black] at (0,-3) {$0$};
\draw[->,black] (0,-2.4) -- (0,-2.6);
\node[black] at (3.5,-3) {$0 $};
\draw[->,black] (3.5,-2.4) -- (3.5,-2.6);
\node[black] at (7,-3) {$0$};
\draw[->,black] (7,-2.4) -- (7,-2.6);
\node[black] at (10.5,-3) {$0 $.};
\draw[->,black] (10.5,-2.4) -- (10.5,-2.6);
\draw[->,black] (2,-3) -- (1.5,-3);
\draw[->,black] (9,-3) -- (8.5,-3);
\draw[->,black] (5.5,-3) -- (5,-3);
\end{tikzpicture}
\end{center}
The only nonzero row is:
\begin{align*}
0\xleftarrow[]{d'}\displaywedge_{-}^{3}\displaywedge_{+}^{0}\otimes K_{\lceil \frac{a}{2}\rceil-3} \xleftarrow[]{d'}  \displaywedge_{-}^{2}\displaywedge_{+}^{0}\otimes K_{\lceil \frac{a}{2}\rceil-2}  \xleftarrow[]{d'} \displaywedge_{-}^{1}\displaywedge_{+}^{0}\otimes K_{\lceil \frac{a}{2}\rceil-1} \xleftarrow[]{d'}  \displaywedge_{-}^{0}\displaywedge_{+}^{0}\otimes K_{\lceil \frac{a}{2}\rceil} \xleftarrow[]{d'}0.
\end{align*}
We now compute its homology with respect to $d'$.
\begin{itemize}
	\item[i:] Let us consider $d':\inlinewedge_{-}^{0}\inlinewedge_{+}^{0}\otimes K_{\lceil \frac{a}{2}\rceil} \rightarrow \inlinewedge_{-}^{1}\inlinewedge_{+}^{0}\otimes K_{\lceil \frac{a}{2}\rceil-1}$. Let $f \in \inlinewedge_{-}^{0}\inlinewedge_{+}^{0}\otimes K_{\lceil \frac{a}{2}\rceil}$. Hence $d'(f)=w_{12}\otimes \partial_{12}f+w_{13}\otimes \partial_{13}f+w_{23}\otimes \partial_{23}f=0$ if and only if $f=f(x_{1},x_{2},x_{3},\widehat{x}_{12},\widehat{x}_{13},\widehat{x}_{23})$.
	\item[ii:] Let us consider $d':\inlinewedge_{-}^{1}\inlinewedge_{+}^{0}\otimes K_{\lceil \frac{a}{2}\rceil-1} \rightarrow \inlinewedge_{-}^{2}\inlinewedge_{+}^{0}\otimes K_{\lceil \frac{a}{2}\rceil-2}$. Let $v=w_{12} \otimes p_{1}+w_{13} \otimes p_{2}+w_{23} \otimes p_{3}\in \inlinewedge_{-}^{1}\inlinewedge_{+}^{0}\otimes K_{\lceil \frac{a}{2}\rceil-1} $. Therefore $d'(v)=0$ if and only if:
	\begin{align}
		\label{appKerd'}
	\begin{cases}
	\partial_{13}p_{1}=	\partial_{12}p_{2},\\
	\partial_{23}p_{2}=	\partial_{13}p_{3},\\
	\partial_{23}p_{1}=	\partial_{12}p_{3}.
	\end{cases}
	\end{align}
	We point out that if $v  = d'(w)$, it is necessary that $\partial_{3}p_{1}-\partial_{2}p_{2}+\partial_{1}p_{3}=0$: indeed $w \in \Ker D_4$ and $\partial_{12}w=p_1$, $\partial_{13}w=p_2$, $\partial_{23}w=p_3$ imply $0=D_4 w=\partial_{3}p_{1}-\partial_{2}p_{2}+\partial_{1}p_{3}$. From \eqref{appKerd'} and $p_{1},p_{2},p_{3} \in \Ker D_{4}$, we deduce that $\partial_{3}p_{1}-\partial_{2}p_{2}+\partial_{1}p_{3}=G(\widehat{x}_{12},\widehat{x}_{13},\widehat{x}_{23})$. By \eqref{appKerd'}, we have that $p_{3}=\int (	\partial_{23}p_{2})dx_{13}+S_{3}(\widehat{x}_{13})$ and $p_{1}=\int (	\partial_{12}p_{2})dx_{13}+\int (	\partial_{12}S_{3}(\widehat{x}_{13}))dx_{23}+S_1(\widehat{x}_{13},\widehat{x}_{23})$. Let us consider $w=\int p_{2} dx_{13}+S(\widehat{x}_{13})$. Hence:
	\begin{align*}
	v-d'(w)=w_{12} \otimes \Big(\int (	\partial_{12}S_{3}(\widehat{x}_{13}))dx_{23}+S_1(\widehat{x}_{13},\widehat{x}_{23})-\partial_{12} S(\widehat{x}_{13})\Big)+w_{23} \otimes \big(S_{3}(\widehat{x}_{13})-\partial_{23} S(\widehat{x}_{13})\big).
	\end{align*}
	By Remark \ref{rempropL1} there exists $H(\widehat{x}_{13},\widehat{x}_{23})$ such that $\widetilde{w}:=\int S_{3}(\widehat{x}_{13})dx_{23}- S(\widehat{x}_{13})+H(\widehat{x}_{13},\widehat{x}_{23}) \in \Ker D_4$. Thus
		\begin{align*}
		v-d'(w)-d'(\widetilde{w})=w_{12}\otimes T(\widehat{x}_{13},\widehat{x}_{23}),
		\end{align*}
		where $T(\widehat{x}_{13},\widehat{x}_{23}):=S_1(\widehat{x}_{13},\widehat{x}_{23})-\partial_{12}H(\widehat{x}_{13},\widehat{x}_{23})$. Since $D_4(T)=0$, we obtain $\partial_{12}\partial_{3}T(\widehat{x}_{13},\widehat{x}_{23})=0$; therefore $T(\widehat{x}_{13},\widehat{x}_{23})=T_1(x_{1},x_{2},x_{3},\widehat{x}_{12},\widehat{x}_{13},\widehat{x}_{23})+T_2(x_{1},x_{2},\widehat{x}_{3},x_{12},\widehat{x}_{13},\widehat{x}_{23})$, where $x_{12}|T_2$. Since $w_{12}\otimes T_2=d'(\int T_2 dx_{12})$, in the quotient there are left only elements of the type $w_{12}\otimes T_1(x_{1},x_{2},x_{3},\widehat{x}_{12},\widehat{x}_{13},\widehat{x}_{23})$ with $x_3 |T_1$.
	\item[iii:] Let us consider $d':\inlinewedge_{-}^{2}\inlinewedge_{+}^{0}\otimes K_{\lceil \frac{a}{2}\rceil-2} \rightarrow \inlinewedge_{-}^{3}\inlinewedge_{+}^{0}\otimes K_{\lceil \frac{a}{2}\rceil-3}$. Let $v=w_{12}w_{13}\otimes p_{1}+w_{12}w_{23}\otimes p_{2}+w_{13}w_{23}\otimes p_{3} \in \inlinewedge_{-}^{2}\inlinewedge_{+}^{0}\otimes K_{\lceil \frac{a}{2}\rceil-2}$. Thus $d'(v)=w_{12}w_{13}w_{23}\otimes( \partial_{23}p_{1}-\partial_{13}p_{2}+\partial_{12}p_{3})=0$ if and only if $\partial_{23}p_{1}-\partial_{13}p_{2}+\partial_{12}p_{3}=0$, which means $p_{3}=\int(\partial_{13}p_{2}-\partial_{23}p_{1})dx_{12}+K(\widehat{x}_{12})$. We consider $w=w_{23} \otimes (-\int p_{2}dx_{12}+H_{1}(\widehat{x}_{12}))+w_{13} \otimes (-\int p_{1}dx_{12}+H_{2}(\widehat{x}_{12}))$. Hence $$d'(w)=w_{12}w_{13}\otimes p_{1}+w_{12}w_{23}\otimes p_{2}+w_{13}w_{23}\otimes \Big(\int(\partial_{13}p_{2}-\partial_{23}p_{1})dx_{12}-\partial_{13}H_{1}(\widehat{x}_{12})+\partial_{23}H_{2}(\widehat{x}_{12})\Big).$$
	Therefore 
	\begin{align*}
	v-d'(w)&=w_{13}w_{23}\otimes (K(\widehat{x}_{12})+\partial_{13}H_{1}(\widehat{x}_{12})-\partial_{23}H_{2}(\widehat{x}_{12}))\\
	&=d'\Big(w_{13}\otimes \Big( \int ( K(\widehat{x}_{12})+\partial_{13}H_{1}(\widehat{x}_{12})-\partial_{23}H_{2}(\widehat{x}_{12}))dx_{23}+S(\widehat{x}_{12}, \widehat{x}_{23})\Big)\Big).
	\end{align*}
	Hence at this point the sequence is exact.
	\item[iv:] Let us consider $d':\inlinewedge_{-}^{3}\inlinewedge_{+}^{0}\otimes K_{\lceil \frac{a}{2}\rceil-3} \rightarrow 0$. Let $v=w_{12}w_{13}w_{23}\otimes p \in \inlinewedge_{-}^{3}\inlinewedge_{+}^{0}\otimes K_{\lceil \frac{a}{2}\rceil-3}$. It follows that $v=d'(w_{12}w_{13}\otimes (\int p dx_{23}+K(\widehat{x}_{23}))$. Hence at this point the sequence is exact.
\end{itemize}
Since for a one row spectral sequence we have that $E'^{2}=E'^{\infty}$, we obtain that
\begin{align*}
&\oplus_{n_1,n_2}H^{n_{1},n_{2}}(G_{A}(a,0)) \cong  \sum_{\widetilde{p},\widetilde{q}}E'^{\infty}_{\widetilde{p},\widetilde{q}} =\sum^{1}_{i=0}E'^{\infty}_{\lceil\frac{a}{2}\rceil-i,0}.
\end{align*}
Hence
\begin{align*}
&H^{n_1,0}(G_{A}(a,0)) \cong \displaywedge_{-}^{0}\displaywedge_{+}^{0}\otimes (V_{A}^{n_1,0})_{[n_1,0]} \quad \text{for} \,\, n_1=a\\
&H^{n_1,0}(G_{A}(a,0)) \cong \left\{w_{12}\otimes f, \,\,\, f \in (V_{A}^{n_1,0})_{[n_1,0]} \,\,\,|\,\,\, f=f(x_{1},x_{2},x_{3})\,\,\, and \,\,\, x_3| f \right\}  \quad \text{for} \,\, n_1=a-2.
\end{align*}
\item[6)] 
Let $b=0$ and $0 \leq a \leq 6$. Again $b=0$ forces the first complex to be the tensor product of $\inlinewedge_{-}^{\frac{a-(2\widetilde{p}-\mathcal{P}(p))}{2}}$ and the following complex:
\begin{align*}
0\xleftarrow[]{d''}\displaywedge_{+}^{0}\otimes(\widetilde{V}_{A})_{[\widetilde{p},0]} \xleftarrow[]{d''} 0.
\end{align*}
Again we denote by $K_{\widetilde{p}}=(\widetilde{V}_{A})_{[\widetilde{p},0]}$; we point out that $f \in K_{\widetilde{p}}$ is such that $f=f(x_{1},x_{2},x_{3},x_{12},x_{13},x_{23})$ and $f \in \Ker D_4$. As before, the only nonzero row of the $E'^{1}$ diagram is
\begin{align*}
0\xleftarrow[]{d'} \displaywedge_{-}^{\lfloor \frac{a}{2}\rfloor}\displaywedge_{+}^{0}\otimes K_{\lceil \frac{a}{2}\rceil-\lfloor \frac{a}{2}\rfloor} \xleftarrow[]{d'} \cdots  \xleftarrow[]{d'} \displaywedge_{-}^{1}\displaywedge_{+}^{0}\otimes K_{\lceil \frac{a}{2}\rceil-1} \xleftarrow[]{d'}  \displaywedge_{-}^{0}\displaywedge_{+}^{0}\otimes K_{\lceil \frac{a}{2}\rceil } \xleftarrow[]{d'}0.
\end{align*}
If $a=6$ , the computations are analogous to case 5). If $a=5$ the complex reduces to
\begin{align*}
0\xleftarrow[]{d'} \displaywedge_{-}^{2}\displaywedge_{+}^{0}\otimes K_1 \xleftarrow[]{d'}  \displaywedge_{-}^{1}\displaywedge_{+}^{0}\otimes K_{2} \xleftarrow[]{d'}  \displaywedge_{-}^{0}\displaywedge_{+}^{0}\otimes K_{3} \xleftarrow[]{d'}0.
\end{align*}
In $\inlinewedge_{-}^{2}\inlinewedge_{+}^{0}\otimes K_1 $ the sequence is exact since $\widetilde{p}+\widetilde{q}=1$ implies $n_1+2n_2=1$ for the elements of $K_1 $. Hence $n_1=1$, $n_2=0$ and the elements in $ K_1$ do not depend on $x_{12},x_{13},x_{23}$. For example $w_{12}w_{13} \otimes x_j=d'(w_{12}\otimes (x_{j}x_{13}+K(\widehat{x}_{13},\widehat{x}_{23})))$ and similarly for the others elements of $\inlinewedge_{-}^{2}\inlinewedge_{+}^{0}\otimes K_1 $. The computation of the homology in $  \inlinewedge_{-}^{1}\inlinewedge_{+}^{0}\otimes K_{2} $ and $\inlinewedge_{-}^{0}\inlinewedge_{+}^{0}\otimes K_{3}$ is analogous to case 5).\\
If $a=4$ the complex reduces to
\begin{align*}
0\xleftarrow[]{d'} \displaywedge_{-}^{2}\displaywedge_{+}^{0}\otimes K_0 \xleftarrow[]{d'}  \displaywedge_{-}^{1}\displaywedge_{+}^{0}\otimes K_{1} \xleftarrow[]{d'}  \displaywedge_{-}^{0}\displaywedge_{+}^{0}\otimes K_{2} \xleftarrow[]{d'}0.
\end{align*}
In $\inlinewedge_{-}^{2}\inlinewedge_{+}^{0}\otimes K_0 $ the sequence is exact since $\widetilde{p}+\widetilde{q}=0$ implies $n_1+2n_2=0$ for the elements of $K_0 $. Hence $n_1=0$, $n_2=0$ and the elements in $ K_0$ do not depend on $x_{1},x_{2},x_{3},x_{12},x_{13},x_{23}$. For example $w_{12}w_{13} \otimes 1=d'(w_{12}\otimes x_{13})$ and similarly for the others elements of $\inlinewedge_{-}^{2}\inlinewedge_{+}^{0}\otimes K_0 $. The computation of the homology in $  \inlinewedge_{-}^{1}\inlinewedge_{+}^{0}\otimes K_{1} $ and $\inlinewedge_{-}^{0}\inlinewedge_{+}^{0}\otimes K_{2}$ is analogous to case 5).\\
If $a=3$ the complex reduces to
\begin{align*}
0\xleftarrow[]{d'} \displaywedge_{-}^{1}\displaywedge_{+}^{0}\otimes K_1 \xleftarrow[]{d'}  \displaywedge_{-}^{0}\displaywedge_{+}^{0}\otimes K_{2} \xleftarrow[]{d'}0.
\end{align*}
Again, $\widetilde{p}+\widetilde{q}=1$ implies $n_1+2n_2=1$ for the elements in $ K_1$; hence the elements in $ K_1$ do not depend on $x_{12},x_{13},x_{23}$. As in case 5) in $\inlinewedge_{-}^{1}\inlinewedge_{+}^{0}\otimes K_1 $ at the quotient there is left $w_{12}\otimes x_{3}$. \\
If $a=2$ the complex reduces to
\begin{align*}
0\xleftarrow[]{d'} \displaywedge_{-}^{1}\displaywedge_{+}^{0}\otimes K_0 \xleftarrow[]{d'}  \displaywedge_{-}^{0}\displaywedge_{+}^{0}\otimes K_{1} \xleftarrow[]{d'}0.
\end{align*}
Again, $\widetilde{p}+\widetilde{q}=0$ implies $n_1+2n_2=0$, the elements in $ K_0$ do not depend on $x_{1},x_{2},x_{3},x_{12},x_{13},x_{23}$. Thus, in $\inlinewedge_{-}^{1}\inlinewedge_{+}^{0}\otimes K_0 $ the sequence is exact: for all $i_1,i_2 \in \left\{1,2,3\right\}$ $w_{i_1 i_2} \otimes 1=d'(x_{i_1 i_2})$.
If $a=1$ the reduces to
\begin{align*}
0 \xleftarrow[]{d'} \displaywedge_{-}^{0}\displaywedge_{+}^{0}\otimes K_{1} \xleftarrow[]{d'}0,
\end{align*}
and if $a=0$ it reduces to
\begin{align*}
0 \xleftarrow[]{d'} \displaywedge_{-}^{0}\displaywedge_{+}^{0}\otimes K_{0} \xleftarrow[]{d'}0;
\end{align*}
hence we can conclude as in case 5). Therefore, analogously to case 5), we obtain
\begin{align*}
&H^{n_1,0}(G_{A}(a,0)) \cong \displaywedge_{-}^{0}\displaywedge_{+}^{0}\otimes (V_{A}^{n_1,0})_{[n_1,0]} \quad \text{for} \,\, n_1=a,\\
&H^{n_1,0}(G_{A}(a,0)) \cong \left\{w_{12}\otimes f, \,\,\, f \in (V_{A}^{n_1,0})_{[n_1,0]} \,\,\,|\,\,\, f=f(x_{1},x_{2},x_{3})\,\,\, and \,\,\, x_3| f\right\} \quad \text{for} \,\, n_1=a-2.
\end{align*}
 \item[7)] We consider the case $b=2$ and $a+b \geq 8$. Therefore the first complex reduces to the tensor product of $\inlinewedge_{-}^{\frac{a-(2\widetilde{p}-\mathcal{P}(p))}{2}}$ and the following complex:
\begin{align*}
0\xleftarrow[]{d''}\displaywedge_{+}^{1}\otimes(\widetilde{V}_{A})_{[\widetilde{p},0]}  \xleftarrow[]{d''}  \displaywedge_{+}^{0}\otimes(\widetilde{V}_{A})_{[\widetilde{p},1]} \xleftarrow[]{d''} 0.
\end{align*}
We want to compute its homology with respect to $d''$. We already know that $\Ker:d'':\inlinewedge_{+}^{0}\otimes(\widetilde{V}_{A})_{[\widetilde{p},1]} \rightarrow \inlinewedge_{+}^{1}\otimes(\widetilde{V}_{A})_{[\widetilde{p},0]}$ is given by $\left\{f \in (\widetilde{V}_{A})_{[\widetilde{p},1]}  \,\, \, | \,\,\, f=h(x_{1},x_{2},x_{3})x_4 \right\}$. Now we claim that $$\Big(\displaywedge_{+}^{1}\otimes(\widetilde{V}_{A})_{[\widetilde{p},0]} \Big) / \Ima d'' \cong \left\{w_{14}\otimes f \,\,\, with \,\,\, f \in (\widetilde{V}_{A})_{[\widetilde{p},0]} \,\,\, | \,\,\, \partial_{23}f\neq 0\right\}.$$
Indeed it is easy to prove that every element in $\Big(\inlinewedge_{+}^{1}\otimes(\widetilde{V}_{A})_{[\widetilde{p},0]} \Big) / \Ker d''$ is equivalent to an element $v=w_{14}\otimes f(x_{1},x_{2},x_{3},x_{12},x_{13},x_{23})$ such that $f \in \Ker D \cap \bigcap^{4}_{i=1} \Ker D_i$. If $w_{14}\otimes f=d''(w)$, for $w \in \Ker D  \cap \bigcap^{4}_{i=1}\Ker D_i$, then $\partial_{14}w=f$, $\partial_{24}w=\partial_{34}w=0$ imply that $0=D w=\partial_{23}f $. Thus if $\partial_{23}f \neq 0$, $v$ is different from 0 in the quotient. Now let us suppose that $\partial_{23}f=0$: we prove that in this case $v$ is equivalent 0 in the quotient. If $\partial_{2}f=\partial_{3}f=0$, it follows directly that $v=d''(\int f dx_{14})$. Now let $f=m_{1}+m_{2}+m_{3}$, where 
$\partial_{2}m_{1}\neq 0$, $\partial_{3}m_{1}= 0$, $\partial_{2}m_{2}= 0$, $\partial_{3}m_{2} \neq 0$, $\partial_{2}m_{3}\neq 0$ and $\partial_{3}m_{3} \neq 0$. 
Let $M_i$ be the part of $m_i$ that does not depend on $x_{12}$. Hence $v=d''(F)$ where $F=\int f dx_{14}+\int(\partial_{2}m_{1}+\partial_{2}m_{3})x_{4}dx_{12}+\int(\partial_{3}M_{2}+\partial_{3}M_{3})x_{4}dx_{13}$. Indeed $F \in \Ker D \bigcap^{4}_{i=1} \cap \Ker D_i $: $D F=D_1 F=0$, $D_2 F=-\partial_3 f+\int \partial_{13}(\partial_{2}m_{1}+\partial_{2}m_{3})dx_{12}+\partial_{3}M_{2}+\partial_{3}M_{3}=-\partial_3 f+\int \partial_{12}(\partial_{3}m_{2}+\partial_{3}m_{3})dx_{12}+\partial_{3}M_{2}+\partial_{3}M_{3}=0$, $D_3 F=-\partial_2 f+\partial_{2}m_{1}+\partial_{2}m_{3}=0$ and finally $D_4 F=x_4 \partial_{2}\partial_{3}M_3-x_4 \partial_{2}\partial_{3}M_3=0$. \\
In the following we will use the notation: $K_{\widetilde{p},1}:=\left\{f \in (\widetilde{V}_{A})_{[\widetilde{p},1]}  \,\, \, | \,\,\, f=h(x_{1},x_{2},x_{3})x_4 \right\}$ and $K_{\widetilde{p},0}:=\Big(\inlinewedge_{+}^{1}\otimes(\widetilde{V}_{A})_{[\widetilde{p},0]} \Big) / \Ima d''$.
The following is the diagram of the $E'^{1}$ spectral sequence, where the vertical maps are $d''$ and the horizontal maps are induced by $d'$:
\begin{center}
\begin{tikzpicture}
\node[black] at (0,0) {$\inlinewedge_{-}^{3}\inlinewedge_{+}^{0}\otimes K_{\lceil \frac{a}{2}\rceil-3,1} $};
\draw[->,black] (0,-0.4) -- (0,-0.6);
\node[black] at (3.5,0) {$\inlinewedge_{-}^{2}\inlinewedge_{+}^{0}\otimes K_{\lceil \frac{a}{2}\rceil-2,1} $};
\draw[->,black] (3.5,-0.4) -- (3.5,-0.6);
\node[black] at (7,0) {$\inlinewedge_{-}^{1}\inlinewedge_{+}^{0}\otimes K_{\lceil \frac{a}{2}\rceil-1,1}  $};
\draw[->,black] (7,-0.4) -- (7,-0.6);
\node[black] at (10.5,0) {$\inlinewedge_{-}^{0}\inlinewedge_{+}^{0}\otimes K_{\lceil \frac{a}{2}\rceil,1}   $};
\draw[->,black] (10.5,-0.4) -- (10.5,-0.6);
\draw[->,black] (2,0) -- (1.5,0);
\draw[->,black] (5.5,0) -- (5,0);
\draw[->,black] (9,0) -- (8.5,0);
\node[black] at (0,-1) {$\inlinewedge_{-}^{3}\otimes K_{\lceil \frac{a}{2}\rceil-3,0}$};
\draw[->,black] (0,-1.4) -- (0,-1.6);
\node[black] at (3.5,-1) {$\inlinewedge_{-}^{2}\otimes K_{\lceil \frac{a}{2}\rceil-2,0} $};
\draw[->,black] (3.5,-1.4) -- (3.5,-1.6);
\node[black] at (7,-1) {$\inlinewedge_{-}^{1}\otimes K_{\lceil \frac{a}{2}\rceil-1,0}$};
\draw[->,black] (7,-1.4) -- (7,-1.6);
\node[black] at (10.5,-1) {$\inlinewedge_{-}^{0}\otimes K_{\lceil \frac{a}{2}\rceil,0} $};
\draw[->,black] (10.5,-1.4) -- (10.5,-1.6);
\draw[->,black] (2,-1) -- (1.5,-1);
\draw[->,black] (9,-1) -- (8.5,-1);
\draw[->,black] (5.5,-1) -- (5,-1);
\node[black] at (0,-2) {$0$};
\node[black] at (3.5,-2) {$0 $};
\node[black] at (7,-2) {$0$};
\node[black] at (10.5,-2) {$0 $};
\draw[->,black] (2,-2) -- (1.5,-2);
\draw[->,black] (9,-2) -- (8.5,-2);
\draw[->,black] (5.5,-2) -- (5,-2);
\node[black] at (0,-3) {$0$};
\draw[->,black] (0,-2.4) -- (0,-2.6);
\node[black] at (3.5,-3) {$0 $};
\draw[->,black] (3.5,-2.4) -- (3.5,-2.6);
\node[black] at (7,-3) {$0$};
\draw[->,black] (7,-2.4) -- (7,-2.6);
\node[black] at (10.5,-3) {$0 $.};
\draw[->,black] (10.5,-2.4) -- (10.5,-2.6);
\draw[->,black] (2,-3) -- (1.5,-3);
\draw[->,black] (9,-3) -- (8.5,-3);
\draw[->,black] (5.5,-3) -- (5,-3);
%
%
\end{tikzpicture}
\end{center}
We already know that the maps induced by $d'$ on the first row are identically zero. We now want to compute the homology of the second row with respect to the maps induced by $d'$. 
\begin{itemize}
	\item[i:] Let us consider $d':\inlinewedge_{-}^{0}\otimes K_{\lceil \frac{a}{2}\rceil,0} \rightarrow \inlinewedge_{-}^{1}\otimes K_{\lceil \frac{a}{2}\rceil-1,0}$. Let $w_{14}\otimes f \in \inlinewedge_{-}^{0}\otimes K_{\lceil \frac{a}{2}\rceil,0} $. Hence $d'(w_{14}\otimes f)=w_{14}w_{12}\otimes \partial_{12}f+w_{14}w_{13}\otimes \partial_{13}f+w_{14}w_{23}\otimes \partial_{23}f=0$ if and only if $\partial_{23} \partial_{12}f=\partial_{23} \partial_{13}f=\partial_{23} \partial_{23}f=0$. Hence, since we are interested in homogeneous polynomials in the $x_{ij}$'s, $f=p_{1}(x_{1},x_{2},x_{3})x_{23}+p_{2}(x_{1},x_{2},x_{3})x_{13}+p_{3}(x_{1},x_{2},x_{3})x_{12}$, such that $p_1 \neq 0$ and $\partial_1 p_1-\partial_2 p_2+\partial_3 p_3=0.$
	\item[ii:] Let us consider $d':\inlinewedge_{-}^{1}\otimes K_{\lceil \frac{a}{2}\rceil-1,0} \rightarrow \inlinewedge_{-}^{2} \otimes K_{\lceil \frac{a}{2}\rceil-2,0}$. Let $v=w_{12}w_{14} \otimes p_{1}+w_{13} w_{14}\otimes p_{2}+w_{23}w_{14} \otimes p_{3}\in \inlinewedge_{-}^{1}\otimes K_{\lceil \frac{a}{2}\rceil-1,0} $. Therefore $d'(v)=0$ if and only if:
	\begin{align}
		\label{appKerd'new}
	\begin{cases}
	\partial_{23}(\partial_{13}p_{1}-	\partial_{12}p_{2})=0,\\
	\partial_{23}(\partial_{23}p_{2}-	\partial_{13}p_{3})=0,\\
	\partial_{23}(\partial_{23}p_{1}-	\partial_{12}p_{3})=0.
	\end{cases}
	\end{align}
	Therefore there exist $K_{1}(\widehat{x}_{23}),K_{2}(\widehat{x}_{23}),K_{3}(\widehat{x}_{23}) \in \Ker D_4$ such that
	\begin{align*}
	\begin{cases}
	\partial_{13}p_{1}-	\partial_{12}p_{2}=K_{1}(\widehat{x}_{23}),\\
	\partial_{23}p_{2}-	\partial_{13}p_{3}=K_{2}(\widehat{x}_{23}),\\
	\partial_{23}p_{1}-	\partial_{12}p_{3}=K_{3}(\widehat{x}_{23}).
	\end{cases}
	\end{align*}
	We can consider a new representative $w_{12}w_{14} \otimes p_{1}+w_{13} w_{14}\otimes \widetilde{p}_{2}+w_{23}w_{14} \otimes\widetilde{p}_{3}\in \inlinewedge_{-}^{1}\otimes K_{\frac{a-1}{2},0} $ of $v$, such that $\widetilde{p}_{2}=p_2+\int K_{1}(\widehat{x}_{23}) dx_{12}+C_1(\widehat{x}_{12},\widehat{x}_{23})$ and $\widetilde{p}_{3}=p_3+\int K_{3}(\widehat{x}_{23}) dx_{12}+C_2(\widehat{x}_{12},\widehat{x}_{23})$. Therefore:
	\begin{align*}
	\begin{cases}
	\partial_{13}p_{1}-	\partial_{12}\widetilde{p}_{2}=0,\\
	\partial_{23}\widetilde{p}_{2}-	\partial_{13}\widetilde{p}_{3}=K_{2}(\widehat{x}_{23}),\\
	\partial_{23}p_{1}-	\partial_{12}\widetilde{p}_{3}=0,
	\end{cases}
	\end{align*}
	where for simplicity, we still call $\partial_{23}\widetilde{p}_{2}-	\partial_{13}\widetilde{p}_{3}=K_{2}(\widehat{x}_{23})$.
	We point out that $$\partial_{12}(\partial_{23}\widetilde{p}_{2}-	\partial_{13}\widetilde{p}_{3})=\partial_{23}\partial_{13}p_{1}-\partial_{13}\partial_{23}p_{1}=0.$$
	Hence $\partial_{23}\widetilde{p}_{2}-	\partial_{13}\widetilde{p}_{3}=K_{2}(\widehat{x}_{12},\widehat{x}_{23})$. Therefore $p_{1}=\int (	\partial_{12}\widetilde{p}_{3})dx_{23}+T_{1}(\widehat{x}_{23})$ and $\widetilde{p}_{2}=\int (	\partial_{13}\widetilde{p}_{3}+K_{2}(\widehat{x}_{12},\widehat{x}_{23}))dx_{23}+T_2(\widehat{x}_{23})$. Let us consider $w=-w_{14}\otimes (\int \widetilde{p}_{3} dx_{23}+S(\widehat{x}_{23}))$. Hence:
	\begin{align*}
	v-d'(w)=w_{12}w_{14} \otimes (T_{1}(\widehat{x}_{23})-\partial_{12}S(\widehat{x}_{23}))+w_{13}w_{14} \otimes ( K_{2}(\widehat{x}_{12},\widehat{x}_{23})x_{23}+T_{2}(\widehat{x}_{23})-\partial_{13}S(\widehat{x}_{23})),
	\end{align*}
	which, if we let $K(\widehat{x}_{23}):=T_{2}(\widehat{x}_{23})-\partial_{13}S(\widehat{x}_{23})$, is equivalent in the quotient to
	\begin{align*}
	v-d'(w)=w_{13}w_{14} \otimes ( K_{2}(\widehat{x}_{12},\widehat{x}_{23})x_{23}+K(\widehat{x}_{23})),
	\end{align*}
	Now we point out that $D_4 K_{2}(\widehat{x}_{12},\widehat{x}_{23})=-\partial_{2}\partial_{13}K_{2}(\widehat{x}_{12},\widehat{x}_{23})=0$ implies that $K_{2}(\widehat{x}_{12},\widehat{x}_{23})=M_{1}(\widehat{x}_{2},\widehat{x}_{12},\widehat{x}_{23})+M_{2}(\widehat{x}_{12},\widehat{x}_{13},\widehat{x}_{23})$, with $x_{13}|M_{1}$; since we are interested in homogeneous polynomials in the $x_{ij}$'s, we consider the two cases $K_{2}(\widehat{x}_{12},\widehat{x}_{23})=M_{1}(\widehat{x}_{2},\widehat{x}_{12},\widehat{x}_{23})$ and $K_{2}(\widehat{x}_{12},\widehat{x}_{23})= M_{2}(\widehat{x}_{12},\widehat{x}_{13},\widehat{x}_{23})$.\\
	Let $K_{2}(\widehat{x}_{12},\widehat{x}_{23})=M_{1}(\widehat{x}_{2},\widehat{x}_{12},\widehat{x}_{23})$; therefore 
	\begin{align*}
	v-d'(w)=w_{13}w_{14} \otimes ( M_{1}(\widehat{x}_{2},\widehat{x}_{12},\widehat{x}_{23})x_{23}+K(\widehat{x}_{23})).
	\end{align*}
	We consider $F:=\int M_{1}(\widehat{x}_{2},\widehat{x}_{12},\widehat{x}_{23})x_{23}dx_{13}$; thus $D_4 F=\partial_1 \int M_{1}(\widehat{x}_{2},\widehat{x}_{12},\widehat{x}_{23})dx_{13}=C(\widehat{x}_{12},\widehat{x}_{23})$. By Lemma \ref{propL1}, there exists $H(\widehat{x}_{12},\widehat{x}_{23}) $ such that $D_4(F+ H)=0$.
	Hence
	\begin{align*}
	v-d'(w)+d'(w_{14}\otimes(F+H))=w_{13}w_{14} \otimes (K(\widehat{x}_{23})-\partial_{13}H(\widehat{x}_{12},\widehat{x}_{23}))-w_{23}w_{14} \otimes \Big(\int M_{1}(\widehat{x}_{2},\widehat{x}_{12},\widehat{x}_{23})dx_{13}\Big)=0,
	\end{align*}
	in the quotient.\\
	Let $K_{2}(\widehat{x}_{12},\widehat{x}_{23})= M_{2}(\widehat{x}_{12},\widehat{x}_{13},\widehat{x}_{23})$. In this case, $w_{13}w_{14} \otimes ( M_{2}(\widehat{x}_{12},\widehat{x}_{13},\widehat{x}_{23})x_{23}+K(\widehat{x}_{23}))$ lies in the image of $d'$ if and only if $\partial_{2}M_2=0$. Indeed if $\partial_{2}M_2=0$, then $F=M_{2}(\widehat{x}_{12},\widehat{x}_{13},\widehat{x}_{23})x_{13}x_{23}$ is such that $D_4 F=\partial_{1}M_2 x_{13}=C(\widehat{x}_{12},\widehat{x}_{23})$ and, by Lemma \ref{propL1}, there exists $H(\widehat{x}_{12},\widehat{x}_{23}) $ such that $D_4(F+H)=0$. Hence
	\begin{align*}
	v-d'(w)+d'(w_{14}\otimes(F+H))=w_{13}w_{14} \otimes (K(\widehat{x}_{23})-\partial_{13}H(\widehat{x}_{12},\widehat{x}_{23}))+w_{23}w_{14} \otimes (- M_{2}(\widehat{x}_{12},\widehat{x}_{13},\widehat{x}_{23})x_{13})=0,
	\end{align*}
	in the quotient.\\
	Finally, if $\partial_{2}M_2 \neq 0$, then $w_{13}w_{14} \otimes ( M_{2}(\widehat{x}_{12},\widehat{x}_{13},\widehat{x}_{23})x_{23}+K(\widehat{x}_{23}))$ does not lie in the image. Indeed let us suppose that it is the image of $w_{14}\otimes Q$; hence $\partial_{23}\partial_{12}Q=\partial_{23}^{2}Q=0$ and $\partial_{13}Q=M_2 x_{23}+C(\widehat{x}_{23})$. Then $Q=M_2 x_{13}x_{23}+\int C(\widehat{x}_{23})dx_{13}+H(\widehat{x}_{13})$. But $\partial_{23}\partial_{12}Q=0$ implies $\partial_{23}\partial_{12}H=0$ and $\partial_{23}^{2}Q=0$ implies $\partial_{23}^{2}H=0$. But, since $H(\widehat{x}_{13})$ must be homogeneous of degree 2 in the $x_{ij}$'s, therefore $H(\widehat{x}_{13})=E(x_1,x_2,x_3)x^{2}_{12}$. But this implies that $D_4 Q \neq 0$, since $-\partial_2 M_2 x_{23} \neq 0$.\\
	Hence at the quotient there are left only elements of type
	$$w_{13}w_{14} \otimes ( N_{1}(\widehat{x}_{12},\widehat{x}_{13},\widehat{x}_{23})x_{23}+N_2(\widehat{x}_{23})),$$
	with $\partial_2 N_1\neq 0$.
	\item[iii:] Let us consider $d':\inlinewedge_{-}^{2}\otimes K_{\lceil \frac{a}{2}\rceil-2,0} \rightarrow \inlinewedge_{-}^{3}\otimes K_{\lceil \frac{a}{2}\rceil-3,0}$. Let $v=w_{12}w_{13}w_{14}\otimes p_{1}+w_{12}w_{23}w_{14}\otimes p_{2}+w_{13}w_{23}w_{14}\otimes p_{3} \in \inlinewedge_{-}^{2}\otimes K_{\lceil \frac{a}{2}\rceil-2,0}$. Thus $d'(v)=-w_{12}w_{13}w_{23}w_{14}\otimes( \partial_{23}p_{1}-\partial_{13}p_{2}+\partial_{12}p_{3})=0$ if and only if $\partial_{23}(\partial_{23}p_{1}-\partial_{13}p_{2}+\partial_{12}p_{3})=0$. Thus $\partial_{23}p_{1}-\partial_{13}p_{2}+\partial_{12}p_{3}=K(\widehat{x}_{23})$. We can consider $\widetilde{p}_{2}=p_{2}+\int K(\widehat{x}_{23})dx_{13}+C(\widehat{x}_{13},\widehat{x}_{23}) \in \Ker D_4$. We point out that $w_{12}w_{13}w_{14}\otimes p_{1}+w_{12}w_{23}w_{14}\otimes\widetilde{ p}_{2}+w_{13}w_{23}w_{14}\otimes p_{3}$ is a new representative of $v$ such that $\partial_{23}p_{1}-\partial_{13}\widetilde{p}_{2}+\partial_{12}p_{3}=0$. Therefore $p_{1}=\int(\partial_{13}\widetilde{p}_{2}-\partial_{12}p_{3})dx_{23}+H(\widehat{x}_{23})$. We consider $w=-w_{12}w_{14} \otimes (\int \widetilde{p}_{2}dx_{23}+H_{1}(\widehat{x}_{23}))-w_{13}w_{14} \otimes (\int p_{3}dx_{23}+H_{2}(\widehat{x}_{23}))$.
	Hence $$d'(w)=w_{12}w_{13}w_{14}\otimes \Big(\int(\partial_{13}\widetilde{p}_{2}-\partial_{12}p_{3})dx_{23}+\partial_{13}H_{1}(\widehat{x}_{23})-\partial_{12}H_{2}(\widehat{x}_{23})\Big)+w_{12}w_{23}w_{14}\otimes \widetilde{p}_{2}+w_{13}w_{23}w_{14}\otimes p_{3}.$$
	Therefore 
	\begin{align*}
	v-d'(w)&=w_{12}w_{13}w_{14}\otimes (H(\widehat{x}_{23})-\partial_{13}H_{1}(\widehat{x}_{23})+\partial_{12}H_{2}(\widehat{x}_{23})) ,
	\end{align*}
	that is zero in $\inlinewedge_{-}^{2}\otimes K_{\lceil \frac{a}{2}\rceil-2,0}$. Hence at this point the sequence is exact.
	\item[iv:] Let us consider $d':\inlinewedge_{-}^{3}\otimes K_{\lceil \frac{a}{2}\rceil-3,0} \rightarrow 0$. Let $v=w_{12}w_{13}w_{23}w_{14}\otimes p \in \inlinewedge_{-}^{3}\otimes K_{\lceil \frac{a}{2}\rceil-3,0}$. If follows that $v=d'(-w_{12}w_{13}w_{14}\otimes (\int p dx_{23}+K(\widehat{x}_{23}))$. Hence at this point the sequence is exact.
\end{itemize}
The following is the diagram of the $E'^{2}$:
\begin{center}
\begin{tikzpicture}
\node[black] at (0,0) {$\inlinewedge_{-}^{3}\inlinewedge_{+}^{0}\otimes K_{\lceil \frac{a}{2}\rceil-3,1} $};
\draw[->,black] (0,-0.4) -- (0,-0.6);
\node[black] at (3.5,0) {$\inlinewedge_{-}^{2}\inlinewedge_{+}^{0}\otimes K_{\lceil \frac{a}{2}\rceil-2,1} $};
\draw[->,black] (3.5,-0.4) -- (3.5,-0.6);
\node[black] at (7,0) {$\inlinewedge_{-}^{1}\inlinewedge_{+}^{0}\otimes K_{\lceil \frac{a}{2}\rceil-1,1}  $};
\draw[->,black] (7,-0.4) -- (7,-0.6);
\node[black] at (10.5,0) {$\inlinewedge_{-}^{0}\inlinewedge_{+}^{0}\otimes K_{\lceil \frac{a}{2}\rceil,1}   $};
\draw[->,black] (10.5,-0.4) -- (10.5,-0.6);
\draw[->,black] (2,0) -- (1.5,0);
\draw[->,black] (5.5,0) -- (5,0);
\draw[->,black] (9,0) -- (8.5,0);
\node[black] at (0,-1) {$0$};
\draw[->,black] (0,-1.4) -- (0,-1.6);
\node[black] at (3.5,-1) {$0 $};
\draw[->,black] (3.5,-1.4) -- (3.5,-1.6);
\node[black] at (7,-1) {$\widetilde{K}_{\lceil \frac{a}{2}\rceil-1,0} $};
\draw[->,black] (7,-1.4) -- (7,-1.6);
\node[black] at (10.5,-1) {$\widetilde{K}_{\lceil \frac{a}{2}\rceil,0} $};
\draw[->,black] (10.5,-1.4) -- (10.5,-1.6);
\draw[->,black] (2,-1) -- (1.5,-1);
\draw[->,black] (9,-1) -- (8.5,-1);
\draw[->,black] (5.5,-1) -- (5,-1);
\node[black] at (0,-2) {$0$};
\node[black] at (3.5,-2) {$0 $};
\node[black] at (7,-2) {$0$};
\node[black] at (10.5,-2) {$0 $};
\draw[->,black] (2,-2) -- (1.5,-2);
\draw[->,black] (9,-2) -- (8.5,-2);
\draw[->,black] (5.5,-2) -- (5,-2);
\node[black] at (0,-3) {$0$};
\draw[->,black] (0,-2.4) -- (0,-2.6);
\node[black] at (3.5,-3) {$0 $};
\draw[->,black] (3.5,-2.4) -- (3.5,-2.6);
\node[black] at (7,-3) {$0$};
\draw[->,black] (7,-2.4) -- (7,-2.6);
\node[black] at (10.5,-3) {$0 $,};
\draw[->,black] (10.5,-2.4) -- (10.5,-2.6);
\draw[->,black] (2,-3) -- (1.5,-3);
\draw[->,black] (9,-3) -- (8.5,-3);
\draw[->,black] (5.5,-3) -- (5,-3);
\end{tikzpicture}
\end{center}
where 
\begin{align*}
\widetilde{K}_{\lceil \frac{a}{2}\rceil,0}=&\big\{w_{14}\otimes f, \,\,\,|\,\,f=p_{1}(x_{1},x_{2},x_{3})x_{23}+p_{2}(x_{1},x_{2},x_{3})x_{13}+p_{3}(x_{1},x_{2},x_{3})x_{12}\in K_{\lceil \frac{a}{2}\rceil,0}\\
&\,\,\,|\,\,p_1 \neq 0,\,\, \partial_1 p_1-\partial_2 p_2+\partial_3 p_3=0\big\},
\end{align*}
and 
$$ \widetilde{K}_{\lceil \frac{a}{2}\rceil-1,0}=\left\{w_{13}w_{14} \otimes ( N_1(\widehat{x}_{12},\widehat{x}_{13},\widehat{x}_{23})x_{23}+N_2(\widehat{x}_{23})),
	\,\,\,|\,\,\partial_2 N_1\neq 0 \right\}.$$
	In order to compute $E'^{3}$, we need to understand what happens applying $d_{\widetilde{p},\widetilde{q}}^{(2)}:E'^{2}_{\widetilde{p},\widetilde{q}}\rightarrow E'^{2}_{\widetilde{p}-2,\widetilde{q}+1}$.
	We recall that $d^{(2)}$ works as explained in Remark \ref{d2esplicito} and we will use notation introduced in that remark.
	\begin{itemize}
		\item[a:] Let us compute the Kernel of $d_{\lceil \frac{a}{2}\rceil,0}^{(2)}:E'^{2}_{\lceil \frac{a}{2}\rceil,0}\rightarrow E'^{2}_{\lceil \frac{a}{2}\rceil-2,1}$.
		Let $x=w_{14}\otimes(p_{1}(x_{1},x_{2},x_{3})x_{23}+p_{2}(x_{1},x_{2},x_{3})x_{13}+p_{3}(x_{1},x_{2},x_{3})x_{12})$. Hence $$d'(x)=-w_{12}w_{14}\otimes p_{3}-w_{13}w_{14}\otimes p_{2}-w_{23}w_{14}\otimes p_{1}.$$ We point out that $d'(x)=d''(y_x)$, where 
		\begin{align*}
		y_x=&-w_{12}\otimes(p_{3}x_{14}+x_{4}\partial_2 p_{3}x_{12}+x_{4}\partial_3 p_{3}x_{13})-w_{13}\otimes(p_{2}x_{14}+x_{4}\partial_2 p_{2}x_{12}+x_{4}\partial_3 p_{2}x_{13})\\
		&-w_{23}\otimes(p_{1}x_{14}+x_{4}\partial_2 p_{1}x_{12}+x_{4}\partial_3 p_{1}x_{13}).
		\end{align*}
		Hence $$z_x=d'(y_x)=w_{12}w_{13}\otimes(x_{4}\partial_1 p_{1}) +w_{12}w_{23}\otimes(x_{4}\partial_2 p_{1})+w_{13}w_{23}\otimes(x_{4}\partial_3 p_{1}).$$ Therefore $x$ lies in the Kernel of $d_{\lceil \frac{a}{2}\rceil,0}^{(2)}$ if and only if $\partial_1 p_{1}=\partial_2 p_{1}=\partial_3 p_{1}=0$. But, since $a>6$, this case is not possible. Hence $\Ker (d_{\lceil \frac{a}{2}\rceil,0}^{(2)})=0$. 
		\item[b:] Let us compute the Kernel of $d_{\lceil \frac{a}{2}\rceil-1,0}^{(2)}:E'^{2}_{\lceil \frac{a}{2}\rceil-1,0}\rightarrow E'^{2}_{\lceil \frac{a}{2}\rceil-3,1}$.
		Let $x=w_{13}w_{14} \otimes ( N_1(\widehat{x}_{12},\widehat{x}_{13},\widehat{x}_{23})x_{23}+N_2(\widehat{x}_{23}))$. We point out that $$d'(x)=w_{12}w_{13}w_{14} \otimes ( \partial_{12}N_2(\widehat{x}_{23}))-w_{13}w_{23}w_{14} \otimes  N_1(\widehat{x}_{12},\widehat{x}_{13},\widehat{x}_{23}).$$ Hence $d'(x)=d''(y_x)$, where
		\begin{align*}
		y_x=&w_{12}w_{13}\otimes \Big(\partial_{12}N_2(\widehat{x}_{23})x_{14}+x_{4}\int ( \partial_2\partial_{12}N_2(\widehat{x}_{23}))dx_{12}+x_{4}\int ( \partial_3\partial_{12}N_2(\widehat{x}_{23}))dx_{13} \Big)\\
		&-w_{13}w_{23} \otimes  (N_1(\widehat{x}_{12},\widehat{x}_{13},\widehat{x}_{23})x_{14}+x_{4}\partial_2 N_1(\widehat{x}_{12},\widehat{x}_{13},\widehat{x}_{23})x_{12}+x_{4}\partial_3 N_1(\widehat{x}_{12},\widehat{x}_{13},\widehat{x}_{23})x_{13}).
		\end{align*}
		Therefore $z_x=d'(y_x)=-w_{12}w_{13}w_{23}\otimes x_{4}\partial_2 N_1(\widehat{x}_{12},\widehat{x}_{13},\widehat{x}_{23})$. Since the elements in $E'^{2}_{\frac{a-1}{2},0}$ satisfy $\partial_2 N_1(\widehat{x}_{12},\widehat{x}_{13},\widehat{x}_{23}) \neq 0$, $\Ker (d_{\lceil \frac{a}{2}\rceil-1,0}^{(2)})=0$.
		\end{itemize}
	Thus $E'^{3}_{\lceil \frac{a}{2}\rceil,0}=E'^{3}_{\lceil \frac{a}{2}\rceil-1,0}=0$ and
	\begin{gather*}
H^{n_1,0}(G_{A}(a,2)) \cong \displaywedge_{-}^{\frac{a+2-n_1}{2}}\displaywedge_{+}^{0}\otimes (V_{A}^{n_1,0})_{[n_1-2,2]}, \quad \text{for}\,\,\,n_1\in \left\{a, a+2\right\}; \\
H^{n_1,0}(G_{A}(a,2)) \cong\\
 \frac{\displaywedge_{-}^{2}\displaywedge_{+}^{0}\otimes (V_{A}^{n_1,0})_{[n_1-2,2]}}{\left\{w_{12}w_{13}\otimes x_{4}\partial_{1}p+w_{12}w_{23}\otimes x_{4}\partial_{2}p+w_{13}w_{23}\otimes x_{4}\partial_{3}p, \,\, p(x_1 ,x_2, x_3) \in (V_{A}^{n_1,0})_{[n_1,0]}\right\}}, \quad \text{for}\,\,\,n_1=a-2;  \\
H^{n_1,0}(G_{A}(a,2)) \cong 0, \quad \text{for}\,\,\,n_1=a-4.
	\end{gather*}
	Let us show more explicitly what are the representatives of nonzero elements in $H^{a-2,0}(G_{A}(a,2)) $. Let $v=w_{12}w_{13}\otimes x_4 f_1+w_{12}w_{23}+\otimes  x_4 f_2+w_{13}w_{23}\otimes  x_4 f_3 \in \inlinewedge_{-}^{2}\inlinewedge_{+}^{0}\otimes (V_{A}^{a-2,0})_{[a-4,2]}$. In the quotient $H^{a-2,0}(G_{A}(a,2)) $, $v$ is equivalent to $v-w_{12}w_{13}\otimes x_{4}\partial_{1}p-w_{12}w_{23}\otimes x_{4}\partial_{2}p-w_{13}w_{23}\otimes x_{4}\partial_{3}p$ with $\partial_{3}p=f_3$. Therefore we can consider a new representative that we still denote by $w_{12}w_{13}\otimes x_4 f_1+w_{12}w_{23}\otimes  x_4 f_2$. If $\partial_3 f_2 \neq 0$, $v$ is nonzero in the quotient: indeed there exists not $p$ such that $\partial_{2}p=f_2$ and $\partial_{3}p=0$. If $\partial_3 f_2 = 0$, we can choose a new representative for $v$ that we still denote by $w_{12}w_{13}\otimes x_4 f_1$. It is nonzero in the quotient if and only if $\partial_2 f_1 \neq 0$ or $\partial_3 f_1 \neq 0$: indeed if $\partial_2 f_2 =\partial_3 f_2 = 0$, it is sufficient to choose $p=\int f_1 dx_1$. Therefore the nonzero elements of the quotient $H^{a-2,0}(G_{A}(a,2))$ can be written as $w_{12}w_{13}\otimes x_4 f_1+w_{12}w_{23}+\otimes  x_4 f_2$ with $\partial_3 f_2 \neq 0$ or $w_{12}w_{13}\otimes x_4 f_1$, with $\partial_2 f_1 \neq 0$ or $\partial_3 f_1 \neq 0$. This explicit expressions will be used in the proof of Proposition \ref{prop2}.
\item[8)] We consider the case $b=2$ and $0 \leq a+b < 8$.
	As in 7), the first complex reduces to the tensor product of $\inlinewedge_{-}^{\frac{a-(2\widetilde{p}-\mathcal{P}(p))}{2}}$ and the following complex:
\begin{align*}
0\xleftarrow[]{d''}\displaywedge_{+}^{1}\otimes(\widetilde{V}_{A})_{[\widetilde{p},0]}  \xleftarrow[]{d''}  \displaywedge_{+}^{0}\otimes(\widetilde{V}_{A})_{[\widetilde{p},1]} \xleftarrow[]{d''} 0.
\end{align*}
The computation of its homology is analogous to 7). Similarly to 7), we will use the notation: $K_{\widetilde{p},1}:=\left\{f \in (\widetilde{V}_{A})_{[\widetilde{p},1]}  \,\, \, | \,\,\, f=h(x_{1},x_{2},x_{3})x_4 \right\}$ and $K_{\widetilde{p},0}:=\Big(\inlinewedge_{+}^{1}\otimes(\widetilde{V}_{A})_{[\widetilde{p},0]} \Big) / \Ima d'' \cong \left\{w_{14}\otimes f , \,\, f  \in (\widetilde{V}_{A})_{[\widetilde{p},0]} \,\, | \,\, \partial_{23}f\neq 0\right\}$.
\begin{itemize}
\item[a)] Let $a=5$. The following is the diagram of the $E'^{1}$ spectral sequence, where the vertical maps are $d''$ and the horizontal maps are induced by $d'$:
\begin{center}
\begin{tikzpicture}
\node[black] at (0,0) {$\inlinewedge_{-}^{3}\inlinewedge_{+}^{0}\otimes K_{0,1} $};
\draw[->,black] (0,-0.4) -- (0,-0.6);
\node[black] at (3.5,0) {$\inlinewedge_{-}^{2}\inlinewedge_{+}^{0}\otimes K_{1,1} $};
\draw[->,black] (3.5,-0.4) -- (3.5,-0.6);
\node[black] at (7,0) {$\inlinewedge_{-}^{1}\inlinewedge_{+}^{0}\otimes K_{2,1}  $};
\draw[->,black] (7,-0.4) -- (7,-0.6);
\node[black] at (10.5,0) {$\inlinewedge_{-}^{0}\inlinewedge_{+}^{0}\otimes K_{3,1}   $};
\draw[->,black] (10.5,-0.4) -- (10.5,-0.6);
\draw[->,black] (2,0) -- (1.5,0);
\draw[->,black] (5.5,0) -- (5,0);
\draw[->,black] (9,0) -- (8.5,0);
\node[black] at (0,-1) {$0$};
\draw[->,black] (0,-1.4) -- (0,-1.6);
\node[black] at (3.5,-1) {$\inlinewedge_{-}^{2}\otimes K_{1,0} $};
\draw[->,black] (3.5,-1.4) -- (3.5,-1.6);
\node[black] at (7,-1) {$\inlinewedge_{-}^{1}\otimes K_{2,0}$};
\draw[->,black] (7,-1.4) -- (7,-1.6);
\node[black] at (10.5,-1) {$\inlinewedge_{-}^{0}\otimes K_{3,0} $};
\draw[->,black] (10.5,-1.4) -- (10.5,-1.6);
\draw[->,black] (2,-1) -- (1.5,-1);
\draw[->,black] (9,-1) -- (8.5,-1);
\draw[->,black] (5.5,-1) -- (5,-1);
\node[black] at (0,-2) {$0$};
\node[black] at (3.5,-2) {$0 $};
\node[black] at (7,-2) {$0$};
\node[black] at (10.5,-2) {$0 $};
\draw[->,black] (2,-2) -- (1.5,-2);
\draw[->,black] (9,-2) -- (8.5,-2);
\draw[->,black] (5.5,-2) -- (5,-2);
\node[black] at (0,-3) {$0$};
\draw[->,black] (0,-2.4) -- (0,-2.6);
\node[black] at (3.5,-3) {$0 $};
\draw[->,black] (3.5,-2.4) -- (3.5,-2.6);
\node[black] at (7,-3) {$0$};
\draw[->,black] (7,-2.4) -- (7,-2.6);
\node[black] at (10.5,-3) {$0 $.};
\draw[->,black] (10.5,-2.4) -- (10.5,-2.6);
\draw[->,black] (2,-3) -- (1.5,-3);
\draw[->,black] (9,-3) -- (8.5,-3);
\draw[->,black] (5.5,-3) -- (5,-3);
\end{tikzpicture}
\end{center}
We already know that the maps induced by $d'$ on the first row are zero. We now want to compute the homology of the second row. The computation is similar to the one done in case 7). Indeed in $\inlinewedge_{-}^{2}\otimes K_{1,0}$ the sequence is still exact: $(\widetilde{p},\widetilde{q})=(1,0)$ implies that, for the elements of $(\widetilde{V}_{A})_{[\widetilde{p},0]}$, $n_1+2n_2=1$, i.e. $n_1=1$ and $n_2=0$. This means that in the quotient $\inlinewedge_{-}^{2}\otimes K_{1,0}$ each element is of the form $v=\sum a_{i_1 i_2 i_3 i_4 i_5}w_{i_1 i_2}w_{i_3 i_4}w_{14}\otimes x_{i_{5}}$, with the $i_{j}$'s in $\left\{1,2,3\right\}$ and $a_{i_1 i_2 i_3 i_4 i_5} \in \C$; therefore $v$ is 0 in the quotient. The rest of the computations are analogous to case 7).
\item[b)] Let $a=4$. The following is the diagram of the $E'^{1}$ spectral sequence, where the vertical maps are $d''$ and the horizontal maps are induced by $d'$:
\begin{center}
\begin{tikzpicture}
\node[black] at (0,0) {$0 $};
\draw[->,black] (0,-0.4) -- (0,-0.6);
\node[black] at (3.5,0) {$\inlinewedge_{-}^{2}\inlinewedge_{+}^{0}\otimes K_{0,1} $};
\draw[->,black] (3.5,-0.4) -- (3.5,-0.6);
\node[black] at (7,0) {$\inlinewedge_{-}^{1}\inlinewedge_{+}^{0}\otimes K_{1,1}  $};
\draw[->,black] (7,-0.4) -- (7,-0.6);
\node[black] at (10.5,0) {$\inlinewedge_{-}^{0}\inlinewedge_{+}^{0}\otimes K_{2,1}   $};
\draw[->,black] (10.5,-0.4) -- (10.5,-0.6);
\draw[->,black] (2,0) -- (1.5,0);
\draw[->,black] (5.5,0) -- (5,0);
\draw[->,black] (9,0) -- (8.5,0);
\node[black] at (0,-1) {$0$};
\draw[->,black] (0,-1.4) -- (0,-1.6);
\node[black] at (3.5,-1) {$\inlinewedge_{-}^{2}\otimes K_{0,0} $};
\draw[->,black] (3.5,-1.4) -- (3.5,-1.6);
\node[black] at (7,-1) {$\inlinewedge_{-}^{1}\otimes K_{1,0}$};
\draw[->,black] (7,-1.4) -- (7,-1.6);
\node[black] at (10.5,-1) {$\inlinewedge_{-}^{0}\otimes K_{2,0} $};
\draw[->,black] (10.5,-1.4) -- (10.5,-1.6);
\draw[->,black] (2,-1) -- (1.5,-1);
\draw[->,black] (9,-1) -- (8.5,-1);
\draw[->,black] (5.5,-1) -- (5,-1);
\node[black] at (0,-2) {$0$};
\node[black] at (3.5,-2) {$0 $};
\node[black] at (7,-2) {$0$};
\node[black] at (10.5,-2) {$0 $};
\draw[->,black] (2,-2) -- (1.5,-2);
\draw[->,black] (9,-2) -- (8.5,-2);
\draw[->,black] (5.5,-2) -- (5,-2);
\node[black] at (0,-3) {$0$};
\draw[->,black] (0,-2.4) -- (0,-2.6);
\node[black] at (3.5,-3) {$0 $};
\draw[->,black] (3.5,-2.4) -- (3.5,-2.6);
\node[black] at (7,-3) {$0$};
\draw[->,black] (7,-2.4) -- (7,-2.6);
\node[black] at (10.5,-3) {$0 $.};
\draw[->,black] (10.5,-2.4) -- (10.5,-2.6);
\draw[->,black] (2,-3) -- (1.5,-3);
\draw[->,black] (9,-3) -- (8.5,-3);
\draw[->,black] (5.5,-3) -- (5,-3);
\end{tikzpicture}
\end{center}
We already know that the maps induced by $d'$ on the first row are zero. We now want to compute the homology of the second row. The computation of the homology for the second row is similar to case a). Indeed, in $\inlinewedge_{-}^{2}\otimes K_{0,0}$ the sequence is exact: $(\widetilde{p},\widetilde{q})=(0,0)$ implies that, for the elements of $(\widetilde{V}_{A})_{[0,0]}$, $n_1+2n_2=0$, i.e. $n_1=0$ and $n_2=0$. This means that in the quotient $\inlinewedge_{-}^{2}\otimes K_{0,0}$ each element is of the form $v=\sum a_{i_1 i_2 i_3 i_4} w_{i_1 i_2}w_{i_3 i_4}w_{14}\otimes 1$, with the $i_j$'s in $\left\{1,2,3\right\}$ and $a_{i_1 i_2 i_3 i_4}  \in \C$; therefore $v$ is 0 in the quotient. The rest of the computations are analogous to  case 7).
\item[c)] Let $a=3,1,0,-1,-2$. The computations are analogous to the previous cases.
\item[d)] Let $a=2$. Therefore in this case the $E'^{1}$ diagram is
\begin{center}
\begin{tikzpicture}
\node[black] at (0,0) {$0 $};
\draw[->,black] (0,-0.4) -- (0,-0.6);
\node[black] at (3.5,0) {$0 $};
\draw[->,black] (3.5,-0.4) -- (3.5,-0.6);
\node[black] at (7,0) {$\inlinewedge_{-}^{1}\inlinewedge_{+}^{0}\otimes K_{0,1}  $};
\draw[->,black] (7,-0.4) -- (7,-0.6);
\node[black] at (10.5,0) {$\inlinewedge_{-}^{0}\inlinewedge_{+}^{0}\otimes K_{1,1}   $};
\draw[->,black] (10.5,-0.4) -- (10.5,-0.6);
\draw[->,black] (2,0) -- (1.5,0);
\draw[->,black] (5.5,0) -- (5,0);
\draw[->,black] (9,0) -- (8.5,0);
\node[black] at (0,-1) {$0$};
\draw[->,black] (0,-1.4) -- (0,-1.6);
\node[black] at (3.5,-1) {$0 $};
\draw[->,black] (3.5,-1.4) -- (3.5,-1.6);
\node[black] at (7,-1) {$\inlinewedge_{-}^{1}\otimes K_{0,0}$};
\draw[->,black] (7,-1.4) -- (7,-1.6);
\node[black] at (10.5,-1) {$\inlinewedge_{-}^{0}\otimes K_{1,0} $};
\draw[->,black] (10.5,-1.4) -- (10.5,-1.6);
\draw[->,black] (2,-1) -- (1.5,-1);
\draw[->,black] (9,-1) -- (8.5,-1);
\draw[->,black] (5.5,-1) -- (5,-1);
\node[black] at (0,-2) {$0$};
\node[black] at (3.5,-2) {$0 $};
\node[black] at (7,-2) {$0$};
\node[black] at (10.5,-2) {$0 $};
\draw[->,black] (2,-2) -- (1.5,-2);
\draw[->,black] (9,-2) -- (8.5,-2);
\draw[->,black] (5.5,-2) -- (5,-2);
\node[black] at (0,-3) {$0$};
\draw[->,black] (0,-2.4) -- (0,-2.6);
\node[black] at (3.5,-3) {$0 $};
\draw[->,black] (3.5,-2.4) -- (3.5,-2.6);
\node[black] at (7,-3) {$0$};
\draw[->,black] (7,-2.4) -- (7,-2.6);
\node[black] at (10.5,-3) {$0 $.};
\draw[->,black] (10.5,-2.4) -- (10.5,-2.6);
\draw[->,black] (2,-3) -- (1.5,-3);
\draw[->,black] (9,-3) -- (8.5,-3);
\draw[->,black] (5.5,-3) -- (5,-3);
\end{tikzpicture}
\end{center}
We already know that the maps induced by $d'$ on the first row are zero. We now want to compute the homology of the second row. We point out that $\inlinewedge_{-}^{1}\otimes K_{0,0}=0$ since $(\widetilde{p},\widetilde{q})=(0,0)$ implies $(p,q)=(0,0)$, that is $n_1+2n_2=0$. Hence $(n_1,n_2)=(0,0)$. This means that in the quotient $\inlinewedge_{-}^{1}\otimes K_{0,0}$ each element is of the form $w_{ij}w_{14}\otimes 1$, with $i,j \in \left\{1,2,3\right\}$ and therefore it is 0 in the quotient. 
On the other hand, in $\inlinewedge_{-}^{0}\otimes K_{1,0} $, we have that $(\widetilde{p},\widetilde{q})=(1,0)$ implies $n_1+2n_2=2$. Hence the only nonzero element in the quotient $\inlinewedge_{-}^{0}\otimes K_{1,0} $ is $w_{14}\otimes x_{23}$. Therefore in this case the $E'^{2}$ diagram is
\begin{center}
\begin{tikzpicture}
\node[black] at (0,0) {$0 $};
\draw[->,black] (0,-0.4) -- (0,-0.6);
\node[black] at (3.5,0) {$0 $};
\draw[->,black] (3.5,-0.4) -- (3.5,-0.6);
\node[black] at (7,0) {$\inlinewedge_{-}^{1}\inlinewedge_{+}^{0}\otimes K_{0,1}  $};
\draw[->,black] (7,-0.4) -- (7,-0.6);
\node[black] at (10.5,0) {$\inlinewedge_{-}^{0}\inlinewedge_{+}^{0}\otimes K_{1,1}   $};
\draw[->,black] (10.5,-0.4) -- (10.5,-0.6);
\draw[->,black] (2,0) -- (1.5,0);
\draw[->,black] (5.5,0) -- (5,0);
\draw[->,black] (9,0) -- (8.5,0);
\node[black] at (0,-1) {$0$};
\draw[->,black] (0,-1.4) -- (0,-1.6);
\node[black] at (3.5,-1) {$0 $};
\draw[->,black] (3.5,-1.4) -- (3.5,-1.6);
\node[black] at (7,-1) {$0$};
\draw[->,black] (7,-1.4) -- (7,-1.6);
\node[black] at (10.5,-1) {$\widetilde{K}_{1,0} ,$};
\draw[->,black] (10.5,-1.4) -- (10.5,-1.6);
\draw[->,black] (2,-1) -- (1.5,-1);
\draw[->,black] (9,-1) -- (8.5,-1);
\draw[->,black] (5.5,-1) -- (5,-1);
\node[black] at (0,-2) {$0$};
\node[black] at (3.5,-2) {$0 $};
\node[black] at (7,-2) {$0$};
\node[black] at (10.5,-2) {$0 $};
\draw[->,black] (2,-2) -- (1.5,-2);
\draw[->,black] (9,-2) -- (8.5,-2);
\draw[->,black] (5.5,-2) -- (5,-2);
\node[black] at (0,-3) {$0$};
\draw[->,black] (0,-2.4) -- (0,-2.6);
\node[black] at (3.5,-3) {$0 $};
\draw[->,black] (3.5,-2.4) -- (3.5,-2.6);
\node[black] at (7,-3) {$0$};
\draw[->,black] (7,-2.4) -- (7,-2.6);
\node[black] at (10.5,-3) {$0 $.};
\draw[->,black] (10.5,-2.4) -- (10.5,-2.6);
\draw[->,black] (2,-3) -- (1.5,-3);
\draw[->,black] (9,-3) -- (8.5,-3);
\draw[->,black] (5.5,-3) -- (5,-3);
\end{tikzpicture}
\end{center}
where $\widetilde{K}_{1,0}=\left\{w_{14}\otimes x_{23}\right\}$. In this case, the computation of $E'^{3}$ is different from the case \textbf{A7)} only for the fact that here $\Ker (d^{(2)}_{1,0})=\left\{w_{14}\otimes x_{23}\right\}$. Therefore we obtain that $H^{0,1}(G_{A}(2,2))\cong \C.$
\end{itemize}
\end{enumerate}
	\end{proof}
\begin{prop}
\label{prop2}
As $\g_{0}-$modules:
\begin{align*}
H^{n_{1},n_2}(G_{A^{\circ}}) \cong
\begin{cases}
 \C \, \,  &\text{if}  \,\,(n_{1},n_{2})=(0,0) \,\, \text{or}\,\, (n_{1},n_{2})=(0,1),\\
0 \, \, &\text{if}  \,\,(n_{1},n_{2})\neq(0,0),(0,1),(1,0).
\end{cases}
\end{align*}
\end{prop}
\begin{proof}
By Lemma \ref{4.3ck6}, decomposition $G_{A}=\oplus_{a,b}G_{A}(a,b)$ and Remark \ref{ck6circ} we obtain that $H^{n_{1},n_{2}}(G_{A^{\circ}}) \cong 0$ if $n_{2}>0$ and $(n_1,n_2) \neq (0,1)$, $H^{0,1}(G_{A^{\circ}}) \cong \C$ and $H^{0,0}(G_{A^{\circ}}) \cong \C$.\\
Let us now show that $H^{n_1,0}(G_{A^{\circ}})=0 $ for $n_1>1$.
By Lemma \ref{4.3ck6}, decomposition $G_{A}=\oplus_{a,b}G_{A}(a,b)$, we know that $\frac{G^{n_1,0}_{A}}{ \Ima(\nabla_{A}:G^{n_1,1}_{A}\longrightarrow G^{n_1,0}_{A})}\cong \oplus_{a,b}H^{n_1,0}(G_{A}(a,b))$. But we are interested in the homology of $G_{A^{\circ}}$.\\
We have that the kernel of the map induced by $\nabla_{3}$ between $H^{n_1,0}(G_{A})$ and $H^{n_{1}-2,0}(G_{B})$, with $n_{1}>1$, is actually isomorphic to $\frac{\Ker (\nabla_{3}:  G^{n_1,0}_{A} \rightarrow  G^{n_{1}-2,0}_{B})}{\Ima(\nabla_{A}:G^{n_{1},1}_{A}\longrightarrow G^{n_{1},0}_{A})} $. 
We show that it is 0: it is sufficient to show that $\nabla_{3}$ restricted to $\frac{G^{n_1,0}_{A}}{ \Ima(\nabla_{A})}$ is injective. Following the proof of Lemma \ref{4.3ck6} and using the explicit expression of $\nabla_3$ in \eqref{explicitnabla3}, we have that:
\begin{enumerate}
	\item Let $b \geq 4$ and $n_1>1$. Let $\sum_{j}u_j \otimes f_j \in \inlinewedge^{\frac{a+b-n_1}{2}}_{-} \otimes (V^{n_1,0}_{A})_{[n_1-b,b]}$. Then, since $b \geq 4$ implies that $\partial^{2}_4 f_j \neq 0$ for all $j$'s, we have that $\nabla_3 (\sum_{j}u_j \otimes f_j )\neq 0$ since it is the sum of linearly independent terms including $\sum_{j}u_j w_{34}w_{24}w_{14}\otimes \partial^{2}_4 f_j  \neq 0$.
	\item Let $b = 0$ and $n_1>1$. Let $f\in \inlinewedge^{0}_{-} \otimes (V_{A}^{n_1,0})_{[n_1,0]}$. It is straightforward that $\nabla_3 (f) \neq 0.$
	\item Let $b = 0$ and $n_1>1$. Let $w_{12} \otimes f \in \inlinewedge^{1}_{-} \otimes (V^{n_1,0}_{A})_{[n_1,0]}$, with $x_3 |f $.
	Then $\nabla_3 (w_{12} \otimes f)=-w_{12}w_{13}w_{34}w_{23}\otimes \partial^{2}_3 f-w_{12}w_{13}w_{14}w_{23}\otimes \partial_{1}\partial_{3}f +w_{12}w_{23}w_{24}w_{13}\otimes \partial_{2}\partial_{3}f$. Since $n_1 > 1$ and $x_3 |f $ imply that at least one among $ \partial^{2}_{3}f $, $\partial_{1}\partial_{3}f$ and $\partial_{1}\partial_{3}f$ is different from 0, we have $\nabla_3 (w_{12} \otimes f)\neq 0$.
	\item Let $b = 2$ and $n_1 >1$. Let $f\in \inlinewedge^{0}_{-} \otimes (V_{A}^{n_1,0})_{[n_1-2,2]}$. It is straightforward that $\nabla_3 (f) \neq 0.$
	\item Let $b = 2$ and $n_1 >1$. Let $v=w_{12}\otimes f_1+w_{13}\otimes f_2+w_{23}\otimes f_3 \in \inlinewedge^{1}_{-} \otimes (V_{A}^{n_1,0})_{[n_1-2,2]}$. We have that $\nabla_3 (v)$ is the sum of linearly independent terms including 
	\begin{align*}
	&-w_{12}w_{14}w_{24}w_{13}\otimes \partial_{1}\partial_{4}f_1+w_{12}w_{24}w_{14}w_{23} \otimes \partial_{2}\partial_{4}f_1-w_{12}w_{34}(w_{24}w_{13}+w_{23}w_{14})\otimes \partial_{3}\partial_{4}f_1\\
	&-w_{13}w_{14}w_{12}w_{34}\otimes \partial_{1}\partial_{4}f_2+w_{13}w_{24}(w_{34}w_{12}+w_{14}w_{23})\otimes \partial_{2}\partial_{4}f_2-w_{13}w_{34}w_{23}w_{14}\otimes \partial_{3}\partial_{4}f_2\\
	&-w_{23}w_{14}(w_{24}w_{13}+w_{12}w_{34})\otimes \partial_{1}\partial_{4}f_3+w_{23}w_{24}w_{34}w_{12}\otimes \partial_{2}\partial_{4}f_3-w_{23}w_{34}w_{24}w_{13}\otimes \partial_{3}\partial_{4}f_3.
 \end{align*}
Since $b = 2$ and $n_1 >1$ imply that for all $i=1,2,3 $ there exists $j=j(i) $ such that $\partial_{j}\partial_{4}f_i \neq 0$, we obtain that $\nabla_3 (v)\neq 0.$
	\item Let $b = 2$ and $n_1>1$. Let $v\in H^{n_1,0}(G_{A^{\circ}})$. 
 If $v=w_{12}w_{13} \otimes x_4 f_1$, with $\partial_{2}f_1 \neq 0$ or $\partial_{3}f_1 \neq 0$, we have that $\nabla_3 (v)$ is the sum of linearly independent terms including $w_{12}w_{13}w_{24}w_{14}w_{23} \otimes \partial_{2}f_1-w_{12}w_{13}w_{34}w_{23}w_{14}\otimes \partial_{3}f_1$. Hence we obtain that $\nabla_3 (v)\neq 0.$
If $v=w_{12}w_{13} \otimes x_4 f_1+w_{12}w_{23} \otimes x_4 f_2$, with $\partial_{3}f_2 \neq 0$, we have that $\nabla_3 (v)$ is the sum of linearly independent terms including
	\begin{align*}
	&w_{12}w_{13}w_{24}w_{14}w_{23} \otimes \partial_{2}f_1-w_{12}w_{13}w_{34}w_{23}w_{14}\otimes \partial_{3}f_1\\
	&-w_{12}w_{23}w_{14}w_{24}w_{13}\otimes \partial_{1}f_2-w_{12}w_{23}w_{34}w_{24}w_{13}\otimes \partial_{3}f_2.
	\end{align*}
	Hence $-w_{12}w_{23}w_{34}w_{24}w_{13}\otimes \partial_{3}f_2 \neq 0$ implies that $\nabla_3 (v)\neq 0.$
\end{enumerate}
\end{proof}
\begin{prop}
\label{prop25}
As $\g_{0}-$modules:
\begin{align*}
H^{1,0}(\Gr M_{A^{\circ}}) \cong 0.
\end{align*}
\end{prop}
\begin{proof}
The proof is similar to Proposition \ref{prop2}. By Lemma \ref{4.3ck6}, \eqref{graduato} and decomposition $\Gr M_{A}=\Sym(\g_{-2})\otimes G_A=\oplus_{a,b}\Sym(\g_{-2})\otimes G_{A}(a,b)$, we know that $\frac{\Gr M^{1,0}_{A}}{ \Ima(\nabla_{A}:\Gr M^{1,1}_{A}\longrightarrow \Gr M^{1,0}_{A})}\cong \oplus_{a,b}H^{1,0}(\Gr M_{A}(a,b))$. But we are interested in the homology of $\Gr M^{1,0}_{A^{\circ}}$.\\
We have that the Kernel of the map induced by $\nabla_{5}$ between $H^{1,0}(\Gr M_{A})$ and $H^{0,1}(\Gr M_{C})$ is actually isomorphic to $\frac{\Ker (\nabla_{5}:  \Gr M^{1,0}_{A} \rightarrow  \Gr M^{0,1}_{C})}{\Ima(\nabla_{A}:\Gr M^{1,1}_{A}\longrightarrow \Gr M^{1,0}_{A})} $. 
We show that it is 0: it is sufficient to show that $\nabla_{5}$ restricted to $\frac{\Gr M^{1,0}_{A}}{ \Ima(\nabla_{A})}$ is injective. Following the proof of Proposition \ref{prop2} and using the explicit expressions obtained for the elements of $H^{1,0}(\Gr M_{A})$ in the proof of Lemma \ref{4.3ck6}, we need to prove that $\nabla_{5}$  is not zero on the elements $1 \otimes x_i$, for $i=1,2,3,4$, $\gamma_{12}w_{12}\otimes x_4+\gamma_{13}w_{13}\otimes x_4+\gamma_{23}w_{23}\otimes x_4$, with $\gamma_{ij}\in \C$, and $w_{12}\otimes x_3$. But this follows directly by Remark \ref{teoremavettorisingolarick6ultimo}, since $\nabla_5 (x_1)=\vec{m}_{5c}$: indeed using the explicit expression of $\vec{m_{5c}}$ in \eqref{m5c}, it is possible to show by direct computations that $\nabla_5 (1 \otimes x_i)=x_{i}\partial_{1}.\vec{m}_{5c}\neq 0$, $\nabla_5 (\sum_{1 \leq i<j\leq 3}\gamma_{ij}w_{ij}\otimes x_4)=\sum_{1 \leq i<j\leq 3}\gamma_{ij}w_{ij}\otimes x_{4}\partial_{1}.\vec{m}_{5c}\neq 0$ and $\nabla_5 (w_{12}\otimes x_3)=w_{12}\otimes x_{3}\partial_{1}.\vec{m}_{5c}\neq 0$.
\end{proof}
Finally we can show the following result, that follows from Propositions \ref{prop2} and \ref{keyhomologyspectralck6}.
\begin{prop}
As $\g_{0}-$modules:
\begin{align*}
H^{n_{1},n_{2}}(M_{A}) \cong
\begin{cases}
\C \, \,  &\text{if}   \,\, n_{1}=n_{2}=0,\\
0 \, \, &  \text{otherwise}.
\end{cases}
\end{align*}
\end{prop}
\begin{proof}
By Propositions \ref{prop2}, \ref{prop25} and \ref{keyhomologyspectralck6} we obtain  that, for $ (n_{1},n_{2}) \neq (0,0),(0,1)$, $H^{n_{1},n_{2}}(M_{A}) \cong 0$.\\
We point out that the singular vectors that determine the maps $\nabla_{A}$ for $n_{1}=0$ are singular vectors also in the case of $K_{6}$ (see \cite[Theorem 4.1, Remark 4.2]{ck6} and \cite[Theorem 5.1]{kac1}). Since the maps $\nabla_{A}: M_{A}^{0,n_{2}}\rightarrow  M_{A}^{0,n_{2}-1}$ are completely determined by the image of $v$, highest weight vector of $V_{A}^{0,n_{2}}$, and, due to equivariance, the action of $\g_{\leq 0}$ on $v$, we obtain that for $n_{1}=0$ the maps coincide with the maps in the case of $K_{6}$. The homology in this case was computed in \cite{kac1}. It was shown that it is different from zero only for $n_{1}=n_{2}=0$ and $H^{0,0}(M_{A}) \cong \C$.
\end{proof}
Now we focus on the third quadrant.
\begin{rem}
\label{cckck6dualeconformeshifted}
A consequence of results in \cite{cantacasellikac} on conformal duality is that, in the case of $CK_{6}$, the conformal dual of $\Ind(F)$, where $F=F(n_{1},n_{2},n_{3},n_0)$ is an irreducible $\g_{0}-$module, corresponds to the shifted dual $\Ind(F^{\vee})$, where $F^{\vee} \cong F(n_{3},n_{2},n_{1},-n_{0}+4)$. We will use the results about duality for shifted duals.
\end{rem}
\begin{lem}
\label{lemmatorsione}
The modules $M_B^{n_1,0}/\Ima \nabla_3=M(n_1,0,0,-\frac{n_1}{2}+2)/\Ima \nabla_3$ are torsion-free as $\C[\Theta]-$modules.
\end{lem}
\begin{proof}
Given $v \in M_B^{n_1,0}$ such that $v \notin \Ima \nabla_3$, we show that $\Theta v \notin \Ima \nabla_3$. We recall that, as $\g-$module, $\Ima \nabla_3$ is generated by the highest weight singular vector of degree 3 $w_{12}w_{13}w_{14} \otimes x^{n_1}_{1}$. In this proof we will use that $[\Theta,\g_i]=\g_{i-2}$ for all $i \geq 0$ and $\g_{i}=\g^{i}_{1}$ for all $i \geq 2$ (see Remarks \ref{L3} and \ref{remlowest}). We split the proof in the following cases.
\begin{enumerate}
	\item Let $v \in M_B^{n_1,0}$, such that $v \notin \Ima \nabla_3$, be a vector of degree 0. Therefore it is straightforward that $\Theta v \notin \Ima \nabla_3$, since it has degree 2.
	\item Let $v \in M_B^{n_1,0}$, such that $v \notin \Ima \nabla_3$, be a vector of degree 1. Let us suppose that $\Theta v \in \Ima \nabla_3$.
	Let $g_{3}\in \g_3$. Then $g_3.(\Theta v) \in \Ima \nabla_3$, but
	$$g_3.(\Theta v)=\Theta(g_3. v) +[g_3, \Theta].v.$$
	Since $v$ has degree 1, then $g_3. v=0$. Hence, for all $g_{3}\in \g_3$, $[g_3, \Theta].v \in \Ima \nabla_3$. Since $[g_3, \Theta].v$  is a vector of degree 0 that lies in $\Ima \nabla_3$, we obtain that $[g_3, \Theta].v=0$. Therefore $\g_1.v=0$, that implies $\g_{+}.v=0$. This implies that there exist in $M_B^{n_1,0}$ highest weight singular vectors of degree 1, but this contradicts the classification of singular vectors given in \cite{ck6}.\\
	\item Let $v \in M_B^{n_1,0}$, such that $v \notin \Ima \nabla_3$, be a vector of degree 2. Let us suppose that $\Theta v \in \Ima \nabla_3$.
	Let $g_{3}\in \g_3$. Then $g_3.(\Theta v) \in \Ima \nabla_3$, but
	$$g_3.(\Theta v)=\Theta(g_3. v) +[g_3, \Theta].v.$$
	Since $v$ has degree 2, then $g_3. v=0$. Hence, for all $g_{3}\in \g_3$, $[g_3, \Theta].v \in \Ima \nabla_3$. Since $[g_3, \Theta].v$  is a vector of degree 1 that lies in $\Ima \nabla_3$, we obtain that $[g_3, \Theta].v=0$. Therefore $\g_1.v=0$, that implies $\g_{+}.v=0$. This implies that there exist in $M_B^{n_1,0}$ highest weight singular vectors of degree 2, but this contradicts the classification of singular vectors given in \cite{ck6}.\\
	\item Let $v \in M_B^{n_1,0}$, such that $v \notin \Ima \nabla_3$, be a vector of degree $d\geq 3$. Let us suppose that $\Theta v \in \Ima \nabla_3$.
	Let $g_{j}\in \g_j$, for $j \geq d+1$. Then $g_{j}.(\Theta v) \in \Ima \nabla_3$, but
	$$g_{j}.(\Theta v)=\Theta(g_{j}. v) +[g_{j}, \Theta].v.$$
	Since $v$ has degree $d$, then $g_j. v=0$. Hence, for all $g_{j}\in \g_j$ with $j \geq d+1$, $[g_j, \Theta].v \in \Ima \nabla_3$. Since $[g_j, \Theta].v$  is a vector of degree less or equal to 1 that lies in $\Ima \nabla_3$, we obtain that $[\g_j, \Theta].v=0$ for all $j \geq d+1$. Hence $\g_j.v=0$ for all $j \geq d-1$.
	Let $g_{d}\in \g_d$. Then $g_{d}.(\Theta v) \in \Ima \nabla_3$, but
	$$g_{d}.(\Theta v)=\Theta(g_{d}. v) +[g_{d}, \Theta].v.$$
	Since $\g_j.v=0$ for all $j \geq d-1$, then $g_{d}. v=0$ and therefore $[g_{d}, \Theta].v \in \Ima \nabla_3$. But $[g_{d}, \Theta].v \in \Ima \nabla_3$ and the fact that $[g_{d}, \Theta].v $ has degree 2 imply that $[g_{d}, \Theta].v =0$. Hence $\g_{d-2}.v=0$. Thus $\g_j.v=0$ for all $j \geq d-2$.
	Now let us suppose that $d$ is even. Let $g_{d-2} \in \g_{d-2}$. Then $g_{d-2}.(\Theta v) =\Theta(g_{d-2}. v) +[g_{d-2}, \Theta].v \in \Ima \nabla_3$ and $g_{d-2}. v=0$ imply that $\g_{d-4}.v \in \Ima \nabla_3$. Now let $g_{d-4} \in \g_{d-4}$. Hence $g_{d-4}.(\Theta v) =\Theta(g_{d-4}. v) +[g_{d-4}, \Theta].v \in \Ima \nabla_3$ and $g_{d-4}. v \in \Ima \nabla_3$ imply that $\g_{d-6}.v \in \Ima \nabla_3$. Iterating this argument we obtain that $\g_{d-2k}.v \in \Ima \nabla_3$ for $k \geq 2$. For $k=\frac{d}{2}$, we obtain $\g_0.v \in \Ima \nabla_3$. But $t.v=(-d-\frac{n_1}{2}+2)v \in \Ima \nabla_3$ implies $v \in \Ima \nabla_3$, since $-d-\frac{n_1}{2}+2\neq 0$ for $d\geq 3$. Hence we obtain a contradiction.\\
	Finally let us suppose that $d$ is odd. Let $g_{d-1} \in \g_{d-1}$. Similarly, $g_{d-1}.(\Theta v) =\Theta(g_{d-1}. v) +[g_{d-1}, \Theta].v \in \Ima \nabla_3$ and $g_{d-1}. v=0$ imply that $\g_{d-3}.v \in \Ima \nabla_3$. Iterating this argument we obtain that $\g_{d-(2k+1)}.v \in \Ima \nabla_3$ for $k \geq 1$. For $k=\frac{d-1}{2}$, we obtain $\g_0.v \in \Ima \nabla_3$. But $t.v=(-d-\frac{n_1}{2}+2)v \in\Ima \nabla_3$ implies $v \in \Ima \nabla_3$, since $-d-\frac{n_1}{2}+2\neq 0$ for $d\geq 3$. Hence we obtain a contradiction.
	\end{enumerate} 
\end{proof}
\begin{lem}
\label{lemmatorsione2}
The module $M_C^{0,1}/\Ima \nabla_5=M(0,0,1,\frac{9}{2})/\Ima \nabla_3$ is torsion-free as $\C[\Theta]-$module.
\end{lem}
\begin{proof}
The proof is analogous to the proof of Lemma \ref{lemmatorsione}.
\end{proof}
\begin{prop}
As $\g_{0}-$modules:
\begin{align*}
H^{n_{2},n_{3}}(M_{C}) =
\begin{cases}
\C \,\,\, &if \,\,\, (n_{2},n_{3})=(1,0),\\
0 \,\,\, &otherwise.
\end{cases}
\end{align*}
\end{prop}
\begin{proof}
 We point out that the singular vectors that determine the maps $\nabla_{C}$ for $n_{3}=0$ are singular vectors also in the case of $K_{6}$ (see \cite[Theorem 4.1, Remark 4.2]{ck6} and \cite[Theorem 5.1]{kac1}). Since the maps $\nabla_{C}: M_{C}^{n_2,0}\rightarrow  M_{C}^{n_{2}+1,0}$ are completely determined by the image of $v$, highest weight vector of $V_{C}^{n_2,0}$, and, due to equivariance, the action of $\g_{\leq 0}$ on $v$, we obtain that for $n_{3}=0$ the maps coincide with the maps in the case of $K_{6}$. The homology in this case was computed in \cite{kac1}. We know that:
\begin{align*}
H^{n_{2},0}(M_{C})=
\begin{cases}
\C \,\,& \text{for} \,\, n_{2}=1,\\
0 \,\,& \text{otherwise}.
\end{cases}
\end{align*}
We can use duality to compute the remaining homology spaces for the third quadrant for $n_{2} > 0$ and $n_{3}>0$. Indeed we have that, for $n_{2} > 0$ and $n_{3}>0$, the maps
\begin{align*}
M_{C}^{n_{2}-1,n_3} \xrightarrow{\nabla_{C}} M_{C}^{n_{2},n_{3}}  \xrightarrow{\nabla_{C}} M_{C}^{n_{2}+1,n_3}
\end{align*}
are dual to:
\begin{align*}
M_{A}^{n_1,n_2+1} \xrightarrow{\nabla_{A}} M_{A}^{n_{1},n_{2}}  \xrightarrow{\nabla_{A}} M_{A}^{n_{1},n_{2}-1},
\end{align*}
where $n_{3}=n_{1}$.
We showed that the previous sequence is exact in $M_{A}^{n_{1},n_{2}}$ and $M_{A}^{n_{1},n_{2}-1}$. Therefore we have that $\frac{M_{A}^{n_{1},n_{2}}}{\Ima (\nabla_{A})} \cong \frac{M_{A}^{n_{1},n_{2}}}{\Ker (\nabla_{A})}$ is isomorphic to a submodule of a free module, hence it is a finitely generated torsion free $\C[\Theta]-$module. The same holds for $\frac{M_{A}^{n_{1},n_2-1}}{\Ima (\nabla_{A})} $. Hence, by Remark \ref{cckck6dualeconformeshifted} and Proposition \ref{esattezzafuntoreduale}, we obtain exactness in $M_{C}^{n_{2},n_{3}}$ for $n_{2}>0$, $n_{3}>0$.
Similarly, we use duality to compute the remaining homology spaces for the third quadrant for $n_{2} = 0$ and $n_{3}>1$. Indeed we have that, for $n_{3}>1$, the maps
\begin{align*}
M_{B}^{0,n_3-2} \xrightarrow{\widetilde{\nabla}_{3}} M_{C}^{0,n_{3}}  \xrightarrow{\nabla_{C}} M_{C}^{1,n_3}
\end{align*}
are dual to:
\begin{align*}
M_{A}^{n_1,1} \xrightarrow{\nabla_{A}} M_{A}^{n_{1},0}  \xrightarrow{\nabla_{3}} M_{B}^{n_{1}-2,0},
\end{align*}
where $n_{3}=n_{1}$.
We showed that the previous sequence is exact in $M_{A}^{n_{1},0}$. Therefore we have that $\frac{M_{A}^{n_{1},0}}{\Ima (\nabla_{A})} \cong \frac{M_{A}^{n_{1},0}}{\Ker (\nabla_{A})}$ is isomorphic to a submodule of a free module, hence it is a finitely generated torsion free $\C[\Theta]-$module. By Lemma \ref{lemmatorsione}, $\frac{M_{B}^{n_{1}-2,0}}{\Ima (\nabla_{3})} $ is a torsion free $\C[\Theta]-$module. Hence, by Remark \ref{cckck6dualeconformeshifted} and Proposition \ref{esattezzafuntoreduale}, we obtain exactness in $M_{C}^{0,n_{3}}$ for $n_{3}>1$.
Finally, with a similar argument, using Lemma \ref{lemmatorsione2}, Remark \ref{cckck6dualeconformeshifted} and Proposition \ref{esattezzafuntoreduale}, we obtain exactness in $M_{C}^{0,1}$.
\end{proof}
\begin{ack}
	The author would like to thank Nicoletta Cantarini, Fabrizio Caselli and Victor Kac for useful comments and suggestions.
	\end{ack}
	\textbf{Data availability statement:} not applicable.


\begin{thebibliography}{-}
\bibitem{bagnolitesi} Bagnoli L. \textit{Finire irreducible modules over the conformal superalgebras $K'_4$ and $CK_6$}, PhD Thesis in Mathematics, University of Bologna (2021)
\bibitem{bagnolicaselli} Bagnoli L., Caselli F. \textit{Classification of finite irreducible conformal modules for $K'_4$}, J. Math. Phys. 63, 091701 (2022) https://doi.org/10.1063/5.0098441
\bibitem{bagnoli1} Bagnoli L. \textit{An upper bound on the degree of singular vectors for $E(1,6)$}, arXiv:2106.01990
\bibitem{bagnoli2} Bagnoli L. \textit{Computation of the homology of the complexes of finite Verma modules for $K'_4$}, arXiv:2108.00458, Accepted in Algebras and Representation Theory.
\bibitem{kac1} Boyallian C., Kac V.G., Liberati, J. \textit{Irreducible modules over finite simple Lie conformal superalgebras of type K}, J. Math. Phys. 51, 063507 (2010).
	\bibitem{ck6} Boyallian C., Kac V. G., Liberati J. \textit{Classification of finite irreducible modules over the Lie conformal superalgebra $CK_{6}$}, Comm. Math. Phys. 317, 503-546 (2013).
	\bibitem{bklr} Boyallian C., Kac  V. G., Liberati J., Rudakov A. \textit{Representations of simple finite Lie conformal superalgebras of type W and S}, J. Math. Phys. 47, 043513 (2006).
	\bibitem{cantacaselliE510} Cantarini N., Caselli F. \textit{Low Degree Morphisms of E(5, 10)-generalized Verma Modules}, Algebras and Representation Theory 23, 2131-2165 (2020).
	\bibitem{cantacasellikac} Cantarini N., Caselli F., Kac, V. \textit{Lie conformal superalgebras and duality of modules over linearly compact Lie superalgebras}, Adv. Math. 378, 107523 (2021).
	\bibitem{cantacasellikacE510} Cantarini N., Caselli F., Kac, V. \textit{Classification of degenerate Verma modules for E(5,10)}, Commun. Math. Phys. 385, 963-1005 (2021).
	\bibitem{chengcantakac} Cantarini N., Cheng S.-J., Kac V. \textit{Errata to Structure of Some $\Z$-graded Lie Superalgebras of Vector Fields}, Transf. Groups 9, 399-400 (2004).
	\bibitem{chengkac} Cheng S.-J., Kac, V.G. \textit{Conformal modules}, Asian J. Math. 1, 181-193 (1997).
	\bibitem{chengkacE} Cheng S.-J., Kac, V.G. \textit{Erratum: Conformal modules}, Asian J. Math. 2, 153-156 (1998).
	\bibitem{new6} Cheng S.-J., Kac V.G. \textit{A new $N = 6$ superconformal algebra}, Comm. Math. Phys. 186, 219-231 (1997).
	\bibitem{chengkac2} Cheng S.-J., Kac, V. \textit{Structure of some $\Z$-graded Lie superalgebras of vector fields}, Transf. Groups, 4, 219-272 (1999).
	\bibitem{chenglam} Cheng, S.-J., Lam, N. \textit{Finite conformal modules over N=2,3,4 superconformal algebras}, J. Math. Phys. 42, 906-933 (2001).
	\bibitem{dandrea} D'Andrea A, Kac V. \textit{Structure theory of finite conformal algebras}, Sel. Math. 4, 377-418 (1998).
	\bibitem{fattorikac} Fattori D., Kac V. G. \textit{Classification of finite simple Lie conformal superalgebras}, J. Algebra 258, 23-59 (2002), Special issue in celebration of Claudio Procesi's 60th birthday.
	\bibitem{kac1vertex} Kac V.G. \textit{Vertex algebras for beginners}, 2nd ed., University Lecture Series, vol.10, AMS, Providence (1998)
	\bibitem{kac98} Kac V.G. \textit{Classification of Infinite-Dimensional Simple Linearly Compact Lie Superalgebras}, Advanced in Mathematics 139, 1-55 (1998). 
	\bibitem{kacrudakovE36} Kac V.G., Rudakov A. \textit{Representations of the Exceptional Lie Superalgebra $E(3,6)$ II: Four Series of Degenerate Modules}, Comm. Math. Phys. 222, 611-661 (2001)
	\bibitem{kacrudakov} Kac V. G., Rudakov A. \textit{Representations of the exceptional Lie superalgebra E(3,6). I. Degeneracy conditions}, Transform. Groups 7, 67-86 (2002). 
	\bibitem{kacrudakovE38} Kac V. G., Rudakov A. \textit{Complexes of modules over the exceptional Lie superalgebras E(3, 8) and E(5, 10)}, Int. Math. Res. Not. 19, 1007-1025 (2002).
	\bibitem{E36III} Kac V. G., Rudakov A. \textit{Representations of the exceptional Lie superalgebra E(3, 6). III. Classification of singular vectors}, J. Algebra Appl. 4, 15-57 (2005).
	\bibitem{maclane} Mac Lane S. \textit{Homology} (fourth printing) Springer-Verlag, (1995).
	\bibitem{zm} Mart\'inez C., Zelmanov E., \textit{Irreducible representations of the exceptional Cheng-Kac superalgebra}, Trans. Amer. Math. Soc. 366, 5853-5876 (2014).
	\bibitem{rudakovE510} Rudakov A. \textit{Morphisms of Verma modules over exceptional Lie superalgebra E(5, 10)}, arXiv 1003.1369v1.
	\bibitem{Shche} Shchepochkina, I. \textit{The five exceptional simple Lie superalgebras of vector fields}, Funktsional Anal. i Prilozhen 33(3), 59-72, 96 (2000). transl. in Funct. Anal. Appl. 33 (1999), 3 208-219.
\end{thebibliography}
\end{document}